\documentclass[11pt,reqno]{amsart}    

\usepackage[margin=1in]{geometry}
\usepackage{amsfonts,amsmath,amssymb,enumerate}    
\usepackage{amsthm, bm}    
\usepackage{mathrsfs}
\usepackage{setspace}
\usepackage{comment}
\onehalfspace
\usepackage{tikz}
\usetikzlibrary{arrows,matrix}
\usetikzlibrary{decorations.markings,shapes.geometric,shapes.misc}
\usetikzlibrary{decorations.pathreplacing}  

\tikzset{cross/.style={cross out, draw=black, minimum size=2*(#1-\pgflinewidth), inner sep=0pt, outer sep=0pt}, cross/.default={1pt}}

\numberwithin{equation}{section}
\usepackage{mathtools}
\theoremstyle{plain}% default
\newtheorem{thm}{Theorem}[section]
\newtheorem{lem}[thm]{Lemma}
\newtheorem{prop}[thm]{Proposition}
\newtheorem{cor}[thm]{Corollary}

\newtheorem*{prop*}{Proposition}
\theoremstyle{definition}

\newtheorem{exmp}[thm]{Example}
\theoremstyle{remark}
\newtheorem{rem}[thm]{Remark}

\newtheorem*{rem*}{Remark}
\newtheorem*{ack}{Acknowledgements}

\newcommand{\be}{\begin{equation}}    
\newcommand{\ee}{\end{equation}}    
\newcommand{\beu}{\begin{equation*}}    
\newcommand{\eeu}{\end{equation*}}    
\newcommand{\bea}{\begin{eqnarray}}    
\newcommand{\eea}{\end{eqnarray}}    
\newcommand{\beaa}{\begin{eqnarray*}}    
\newcommand{\eeaa}{\end{eqnarray*}}    
\newcommand{\bmx}{\begin{pmatrix}}    
\newcommand{\emx}{\end{pmatrix}}

\newcommand{\del}{\partial}    
\newcommand{\g}{{\mathfrak g}}

\newcommand{\n}{{\mathfrak n}}    
\newcommand{\h}{{\mathfrak h}}

\newcommand{\p}{{\mathfrak p}}

\newcommand{\Heis}{H}

\newcommand{\mf}{\mathfrak}
\newcommand{\mc}{\mathcal}    

\newcommand{\gh}{{\widehat \g}}    

\newcommand{\nG}{{\n_{\langle G_i\rangle}}}
\newcommand{\gs}{\g^{\Gamma}}

\newcommand{\half}{\frac{1}{2}}

\newcommand{\nn}{\nonumber}

\newcommand{\8}{{\infty}}

\newcommand{\eps}{\epsilon}    
\newcommand{\tr}{\,{\rm tr}}    
    
\newcommand{\rank}{{\rm rank}}

\newcommand{\ad}{{\rm ad}}

\newcommand{\Z}{{\mathbb Z}}

\newcommand{\C}{{\mathbb C}}

\newcommand{\id}{{\mathrm{id}}}

\newcommand{\goi}[2]{=}    
\newcommand{\Hom}{\mathrm{Hom}}

\newcommand{\on}{.}    

\newcommand{\Cx}{\mathbb C^\times}

\newcommand{\btp}{\begin{tikzpicture}[baseline=0pt,scale=0.9,line width=0.25pt]}    
\newcommand{\etp}{\end{tikzpicture}}

\newcommand{\atp}[1]{}

\newcommand{\path}{\longrightarrow}

\newcommand{\ha}{\mbox{\small $\frac{1}{2}$}}

\DeclareMathOperator{\res}{res}

\newcommand*{\longhookrightarrow}{\ensuremath{\lhook\joinrel\relbar\joinrel\rightarrow}}

\newcommand{\VV}{{\mathbb V}}

\newcommand{\Vcrit}{{\mathbb V_0^{-h^\vee}}}
\newcommand{\WW}{{\mathbb W}}
\newcommand{\MM}{{\mathbb M}}
\newcommand{\M}{\mathcal M}

\newcommand{\vac}{v_0}
\newcommand{\wac}{w_0}
\newcommand{\resd}{\vee}
\newcommand{\pdd}{{{\frac{1}{(n-1)!}\frac{\del^{n-1}}{\del u^{n-1}}}}}
\newcommand{\lsigma}{L_{\sigma}}

\newcommand{\Gaud}{\mathscr Z}
\newcommand{\Mh}{\mathsf M} 
\newcommand{\ie}{\textit{i.e. }}

\DeclareMathOperator{\Ind}{Ind}
\DeclareMathOperator{\Coind}{Coind}
\DeclareMathOperator{\End}{End}

\DeclareMathOperator{\Vect}{Vect}

\DeclareMathOperator{\Weyl}{Weyl}
\DeclareMathOperator{\Der}{Der}

\author{Beno\^{\i}t Vicedo} 
\author{Charles Young}
\address{\vspace{-.15cm} 
School of Physics, Astronomy and Mathematics, University of Hertfordshire, College Lane, Hatfield AL10 9AB, UK.}  
\email{benoit.vicedo@gmail.com}  
\email{charlesyoung@cantab.net}

\begin{document} 
\title[Cyclotomic Gaudin models: construction and Bethe ansatz]{Cyclotomic Gaudin models:\\ construction and Bethe ansatz}

\begin{abstract} 

To any simple Lie algebra $\g$ and automorphism $\sigma:\g\to \g$ we associate a \emph{cyclotomic Gaudin algebra}. This is a large commutative subalgebra of $U(\g)^{\otimes N}$ generated by a hierarchy of \emph{cyclotomic Gaudin Hamiltonians}. It reduces to the Gaudin algebra in the special case $\sigma = \id$. 

We go on to construct joint eigenvectors and their eigenvalues for this hierarchy of cyclotomic Gaudin Hamiltonians, in the case of a spin chain consisting of a tensor product of Verma modules. To do so we generalize an approach to the Bethe ansatz due to Feigin, Frenkel and Reshetikhin involving vertex algebras and the Wakimoto construction. 
As part of this construction, we make use of a theorem concerning \emph{cyclotomic coinvariants}, which we prove in a companion paper. 

As a byproduct, we obtain a cyclotomic generalization of the Schechtman-Varchenko formula for the weight function.

\end{abstract}
\maketitle
\setcounter{tocdepth}{1}
\tableofcontents

\section{Introduction and overview}
Let $\omega\in \Cx$ be a root of unity of order $T\in \Z_{\geq 1}$, and fix a collection of $N\in \Z_{\geq 1}$ non-zero points $z_i\in \Cx$ whose orbits, under the multiplicative action of $\omega$, are pairwise disjoint: $z_i \neq \omega^k z_j$ for all $k\in \Z_T=\Z/T\Z$ and all $i\neq j$. 
Let $\g$ be a simple Lie algebra over $\C$ and $\sigma$ an automorphism of $\g$ whose order divides $T$.
In this paper we consider the following family of $N$ elements of $U(\g)^{\otimes N}$:
\begin{equation} \label{Ham intro}
\mathcal{H}_i := \sum_{p = 0}^{T-1} \sum_{\substack{j=1\\j \neq i}}^N \frac{I^{a (i)} \sigma^p I_a^{(j)}}{z_i - \omega^{-p} z_j} + \sum_{p = 1}^{T-1} \frac{I^{a (i)} \sigma^p I_a^{(i)}}{(1 - \omega^p) z_i}, \qquad i = 1, \ldots, N,
\end{equation}
where $\{ I^a\}$ is a basis of $\g$, $\{I_a\}$ is its dual basis with respect to a non-degenerate invariant inner product on $\g$, and for any $A \in U(\g)$ we use the standard notation $A^{(i)} := 
%1\otimes \ldots \otimes A \otimes \ldots \otimes 1 
1^{\otimes (i-1)} \otimes A \otimes 1^{\otimes (N-i)}$.
%\in U(\g)^{\otimes N}$.
% with $A$ in the $i^{\rm th}$ tensor factor. 

%A direct computation reveals that the elements \eqref{Ham intro} generate a commutative subalgebra of $U(\g)^{\otimes N}$. %, that is $[\mathcal{H}_i, \mathcal{H}_j] = 0$ for any $i, j$.

We refer to the elements $\mathcal{H}_i \in U(\g)^{\otimes N}$ of \eqref{Ham intro} as \emph{quadratic cyclotomic Gaudin Hamiltonians}. A direct calculation reveals that they commute amongst themselves. In the special case $\sigma = \text{id}$, $T = 1$, the $\mathcal{H}_i$ reduce to the quadratic Hamiltonians of the celebrated Gaudin model \cite{Gaudin,Gbook} associated to the Lie algebra $\g$, namely
\begin{equation}
\mathcal{H}_i^{\rm Gaudin} = \sum_{\substack{j=1\\j \neq i}}^N \frac{I^{a (i)}  I_a^{(j)}}{z_i - z_j} , \qquad i = 1, \ldots, N.\label{Ghams}
\end{equation}
If one assigns to each marked point $z_i$ a $\g$-module $V_i$ then the Gaudin Hamiltonians are represented by a collection of mutually-commuting linear operators in $\End\left(\bigotimes_{i=1}^N V_i\right)$. Physically, one thinks of these Hamiltonians as describing the dynamics of a ``long-range spin-chain'' consisting of $N$ ``spins'' whose pairwise interactions depend rationally on the marked points $z_i$ in the complex plane. In this language, the Hamiltonians \eqref{Ham intro} describe a generalization of the model in which each spin interacts not only with the other spins but also with their images under the multiplicative action of $\omega$ (and in which there is a self-interaction between each spin and its own images). Observe that the kinematics of the model are unaltered: the algebra of observables remains $U(\g)^{\otimes N}$ and the space of states remains $\bigotimes_{i=1}^N V_i$.

In the study of the Gaudin model, the central problem is the spectral problem: one wishes to find joint eigenvalues and eigenvectors of the mutually commuting Hamiltonians. When the $V_i$ are finite-dimensional irreducible representations of $\g$ (and actually also when they are Verma modules) this problem has been solved 
using various forms of the Bethe ansatz. At the most concrete level, the content of the present paper is to do the same for the cyclotomic Gaudin Hamiltonians of \eqref{Ham intro}. 

\bigskip

The Gaudin model has deep connections to (among others): 
the KZ equations and conformal field theory \cite{FFR,RV}; the geometric Langlands program \cite{Fre07}; and, via the Bethe ansatz, Schubert calculus \cite{MTV2}. Because of this central role in the theory of quantum integrable systems, generalizations of the Gaudin model are of great interest -- and in fact a wealth of different generalizations exist.
To fit the Hamiltonians of \eqref{Ham intro} into the picture, let us recall the possibilities.
First, one can consider, in place of the simple Lie algebra $\g$, a Lie algebra of affine type \cite{FF07} or a Lie superalgebra \cite{MVY}. 
Another possibility is to keep the simple Lie algebra $\g$ but modify the kinematics \cite{FFT,FFRb}.  
In a different direction, keeping both $\g$ and the kinematics fixed, a Gaudin model can be associated to any skew-symmetric solution $r(u,v)$ of the classical Yang-Baxter equation on $\g \otimes \g$ with spectral parameter \cite{Gbook}. Such solutions fall, by the Belavin-Drinfeld \cite{BelavinDrinfeld} classification, into three classes -- rational, trigonometric and elliptic -- and the Hamiltonians \eqref{Ghams} correspond to the rational solution $r(u,v) = I^a\otimes I_a/(u-v)$.
(By replacing the Riemann sphere $\C P^1$, in which the marked points $z_i$ lie, by a more general Riemann surface, one arrives at quantized Hitchen systems \cite{BeilinsonDrinfeld,EnriquezRubtsov}.)
From this point of view, the quadratic Hamiltonians \eqref{Ham intro} turn out to correspond to certain \emph{non}-skew-symmetric solutions to the classical Yang-Baxter equation, and viewed this way they were introduced by T.~Skrypnyk  in \cite{Skrypnyk1}. We comment more on this interpretation below. Another way to understand the origin of the Hamiltonians \eqref{Ham intro} is to recall that the usual Gaudin Hamiltonians \eqref{Ghams} are extracted from the poles of a Lax matrix which encodes the tensor product of evaluation representations of a half-loop algebra $\g \otimes \C[[t]]$. By replacing $\g \otimes \C[[t]]$ by a twisted half-loop algebra $(\g \otimes \C[[t]])^\sigma$, some special cases of the Hamiltonians \eqref{Ham intro} were introduced in \cite{CrampeYoung}.

\bigskip

In order to solve the spectral problem for the cyclotomic Gaudin Hamiltonians in \eqref{Ham intro}, we adopt the approach of B.~Feigin, E.~Frenkel and N.~Reshetikhin in \cite{FFR}. In fact much of the motivation for the present work can best be understood in terms of the details of this approach -- so let us review these details (in outline; cf. \S\ref{ZTGsec} below). 

A central role is played by the \emph{vacuum Verma module at critical level}, $\Vcrit$, over the untwisted affine Lie algebra $\gh$. Recall that $\gh$ is the extension of the loop algebra $\g \otimes \C((t))$ by a one-dimensional centre $\C K$ and that by definition 
$\Vcrit := U(\gh) \otimes_{U(\g \otimes \C[[t]]\oplus \C K)} \C \vac$ 
is the $\gh$-module induced from a \emph{vacuum} vector $\vac$ that is annihilated by all non-negative modes $A[k] := A\otimes t^k\in \g \otimes \C[[t]]$, $k\geq 0$, and on which the central charge $K$ takes the \emph{critical} value. (This critical value is $-h^\vee$ when we work with respect to the normalization of the invariant inner product on $\g$ in which long roots have square length 2, where $h^\vee$ is the dual Coxeter number of $\g$.) 
At this critical value of the central charge, the 
vector
\be S := \half  I^a[-1] I_a[-1] \vac, \label{S intro}\ee
where $I_a[-1] := I_a\otimes t^{-1}$, becomes \emph{singular}, meaning that like the vacuum it too is annihilated by all non-negative modes: $A[k] \on S = 0$ for all $A\in \g$ and all $k\geq 0$. 

To go from this singular vector $S\in \Vcrit$ to the quadratic Gaudin Hamiltonians \eqref{Ghams}, one needs the next key idea from \cite{FFR}, which is the interplay of ``local'' and ``global'' objects. We think of the tensor factors $V_i$ of the Gaudin spin chain $\bigotimes_{i=1}^N V_i$ as being local data, each $V_i$ assigned to its point $z_i\in \C$. The $V_i$'s are by definition $\g$-modules, but they become $\g \otimes \C[[t]]$-modules if we specify that strictly positive modes act as zero. More precisely, each $V_i$ becomes a module over the Lie algebra $\g \otimes \C[[t-z_i]]$ of $\g$-valued Taylor series in a local coordinate $t-z_i$ near the point $z_i$. Let now $u\in \C$ be another point, distinct from the $z_i$, and assign the $\gh$-module $\Vcrit$ to it: that is, regard $\gh$ as a central extension of the Lie algebra $\g \otimes \C((t-u))$ of $\g$-valued Laurent series in the local coordinate near $u$. 
These are the ``local'' objects. The ``global'' objects will be rational functions of the complex plane that vanish at infinity and that have poles at most at finitely many specified points. In fact, let $\g_u$ be the Lie algebra of $\g$-valued rational functions that vanish at infinity and have poles at most at the point $u$ to which $\Vcrit$ is assigned. A function $f(t)\in\g_u$ acts on $\Vcrit$, via its Laurent expansion about $u$. But it also acts on each $V_i$ via its Taylor expansion about $z_i$. Thus $f(t)$ acts on the tensor product $\bigotimes_{i=1}^N V_i\otimes \Vcrit$. The quotient by this action, $(\bigotimes_{i=1}^N V_i \otimes \Vcrit) \big/ \g_u$, is called the space of \emph{coinvariants} with respect to $\g_u$. Each equivalence class can be shown to contain a unique representative in $\bigotimes_{i=1}^N V_i\otimes \C \vac$, and in this way the space of coinvariants is identified as a vector space with   
$\bigotimes_{i=1}^N V_i\otimes \C\vac \cong_\C \bigotimes_{i=1}^N V_i$. 
%$(\bigotimes_{i=1}^N V_i \otimes \Vcrit) \big/ \g_u\cong_\C\bigotimes_{i=1}^N V_i\otimes \C\vac \cong_\C \bigotimes_{i=1}^N V_i$. 
For every vector $X\in \Vcrit$ we then have a linear map $X(u)\in \End(\bigotimes_{i=1}^N V_i)$ which sends $v\in \bigotimes_{i=1}^N V_i$ to the class of $v\otimes X$.  %\xrightarrow{\cdot\otimes X} \bigotimes_{i=1}^N V_i \otimes X \twoheadrightarrow  (\bigotimes_{i=1}^N V_i \otimes X) \big/ \g_u \cong_\C  \bigotimes_{i=1}^N V_i$. 
This map $X(u)$ depends rationally on $u$, with potential poles at the points $z_i$. 

In particular, to the singular vector $S\in \Vcrit$ of \eqref{S intro} is associated a map $S(u)\in \End(\bigotimes_{i=1}^N V_i)$ depending rationally on $u$, with poles at the $z_i$. The residue of this map at the point $z_i$ is precisely the representative in $\End(\bigotimes_{i=1}^N V_i)$ of the Hamiltonian $\mathcal{H}_i^{\textrm{Gaudin}}$. 

Thus far the fact that $S\in \Vcrit$ is singular has not been used. But the real merit of the above construction is that, given any two singular vectors $Z_1,Z_2\in \Vcrit$, the linear operators $Z_1(u)$ and $Z_2(v)$ may be shown to commute (where we now introduce two copies of $\Vcrit$, assigned to distinct points $u$ and $v$ in the complex plane). The space of singular vectors of the vacuum Verma module at critical level is known to be very large \cite{FeiginFrenkel}.\footnote{In type A, an explicit formula for a generating set of singular vectors was given by A.~Chervov and A.~Molev in \cite{ChervovMolev}, based on earlier work by D.~Talalaev \cite{Talalaev}. See also \cite{Molev} for analogous formulae in types B, C and D.} The image of this space of singular vectors is then a large commutative subalgebra of $\End(\bigotimes_{i=1}^NV_i)$, the \emph{Gaudin algebra}\footnote{Here we follow \cite{Freview} in calling this commutative subalgebra a \emph{Gaudin algebra}. In \cite{MTV1} the same object is referred to as a \emph{Bethe algebra}.}
for the data $(\g; V_1,\dots,V_N; z_1,\dots,z_N)$. It is generated by the quadratic Gaudin Hamiltonians \eqref{Ghams} together, when $\rank(\g) \geq 2$, with a hierarchy of \emph{higher Gaudin Hamiltonians}. (In fact it forms a maximal commutative subalgebra of $U(\g)^{\otimes N}$  containing the quadratic Gaudin Hamiltonians \cite{Rybnikov}.)

\bigskip

One would now like to obtain the cyclotomic Hamiltonians \eqref{Ham intro} by suitably generalizing this procedure. Some or all of the above objects should be appropriately twisted by the automorphism $\sigma:\g\to\g$, and the question is, which ones? 

At first sight it is perhaps tempting to think that the affine algebra $\gh$ should be replaced by a twisted affine algebra $\gh^\sigma$, and the $\gh$-module $\Vcrit$ with a module over $\gh^\sigma$. 
This turns out to be the wrong approach, however. 
The cyclic group $\Z_T$ acts on $\g$ by powers of the automorphism $\sigma$ and on the complex plane by rotations about the origin
through
multiplication by powers of $\omega$. But $\gh$ and $\Vcrit$ are local data assigned to the point $u$ and, since the construction above relies on $u$ being suitably generic, we should suppose that $u$ is not the origin. So we do not naturally obtain the projector $\sum_{k\in \Z_T} \omega^{-k} \sigma^k: \g \otimes \C((t)) \to (\g \otimes \C((t)))^\sigma$ onto the twisted loop algebra, $(\g \otimes \C((t)))^\sigma:= \{ X(t) \in \g \otimes \C((t)): X(\omega t) =\sigma X(t)\}$.
What is more, it is a good idea to avoid twisting $\gh$ and $\Vcrit$ if possible, because the structure of these objects is so key in the construction of \cite{FFR}. 

In fact -- and this is the main theme of the present paper -- the appropriate objects to twist are the ``global'' ones, not the ``local'' ones. Thus, we replace $\g_u$ by the algebra of rational functions $f(t)$ that vanish at infinity, that have poles at most at $u$ \emph{and its image points} $\omega^ku$, \emph{and} that obey the equivariance condition
\be f(\omega t) = \sigma f(t).\ee 
The content of \S\ref{ZTGsec} below is to show that the construction of a large commutative subalgebra of $U(\g)^{\otimes N}$ goes through with such twist-equivariant global objects. We call the resulting algebra 
the \emph{cyclotomic Gaudin algebra}.
See Theorem \ref{Zthm}, which is the main result of the first half of the paper. 

This cyclotomic Gaudin algebra contains the quadratic Hamiltonians \eqref{Ham intro} and also (again, when $\rank(\g)\geq 2$) a hierarchy of \emph{higher cyclotomic Gaudin Hamiltonians}. Let us comment on the latter. The change to the allowed rational functions $f(t)$ is apparently minor, so one might suspect that the resulting Hamiltonians would all be correspondingly minor alterations of their untwisted counterparts, much as the first term on the right in \eqref{Ham intro} is nothing but the usual Hamiltonian \eqref{Ghams} suitably ``decorated'' with $\sigma$'s and $\omega$'s. But this is not at all the case. The ``leading terms'' are always of this form, but there is in general a long tail of intricate ``correction'' terms, which in a certain sense arise from ``self-interaction'' due to the twisting. See Remark \ref{sirem} below. 
(Let us note in passing that this feature is closely linked to the difference between the simple closed form that exists for the normal-ordered product of fields for modules over vertex algebras and the much more complicated \cite{Doyon} and/or implicit \cite{Li1,Li3} form of the normal-ordered product for fields in twisted- and quasi-modules over vertex algebras.)

\bigskip

Having defined the cyclotomic Gaudin algebra, the next and larger task is to address the spectral problem using a Bethe Ansatz. Sections \ref{sec:wc}--\ref{sec:Hdef} are devoted to this in the case where the representations $V_i$ are Verma modules with highest weights $\lambda_i\in \h^*$.
Let us begin by summarizing the results. We construct joint eigenvectors -- called \emph{Bethe vectors} --  of the cyclotomic Gaudin Hamiltonians. Each such vector is labelled by a collection of $m\in \Z_{\geq 0}$ pairs $(w_i,c(i))$, $1\leq i\leq m$, where $w_i\in \Cx$ is a \emph{Bethe root} and where $c(i)\in I$ is a node of the Dynkin diagram of $\g$ which one thinks of as labelling a lowering operation $F_{c(i)}$ in the direction of the simple root $\alpha_{c(i)}$. The explicit form of these Bethe vectors is 
\begin{align} \psi_\Gamma = 
 (-1)^m\!\!\!\!\!\! \sum_{\substack{\bm n\in P_{m,N}\\  (k_1,\dots,k_m) \in \Z_T^m}}
  \bigotimes_{i=1}^N \frac{ \check\sigma^{k_{n^i_1}}(F_{c(n^i_{1})})\check\sigma^{k_{n^i_2}}(F_{c(n^i_2)})\dots \check\sigma^{k_{n^i_{p_{i}-1}}}(F_{c(n^i_{p_{i}-1})}) \check\sigma^{k_{n^i_{p_i}}}(F_{c(n^i_{p_i})}) \mathsf{v}_{\lambda_i}}
   {\Big(\omega^{k_{n^i_1}} w_{n^i_1} - \omega^{k_{n^i_2}} w_{n^i_2}\Big)\dots
    \Big(\omega^{k_{n^i_{p_{i}-1}}}w_{n^i_{p_{i}-1}} - \omega^{k_{n^i_{p_i}}} w_{n^i_{p_i}}\Big)
        \Big(\omega^{k_{n^i_{p_i}}}w_{n^i_{p_i}} -   z_i\Big)  }.
\label{SVintro}\end{align}
where the sum $\bm n\in P_{m,N}$ is over ordered partitions of the labels $\{1,\dots,m\}$ into $N$ parts, and where $\check\sigma(X):=\omega\sigma(X)$. (For details see \S\ref{sec:bv}.) This formula is a natural cyclotomic analog of the \emph{Schechtmann-Varchenko formula}, \cite{SV,BabujianFlume}. 

The Bethe roots $w_j$ are required to obey a collection of \emph{cyclotomic Bethe equations}:
\be
0= \sum_{r=0}^{T-1} \sum_{i=1}^N\frac{\langle \alpha_{c(j)},\lsigma^r\lambda_i\rangle}{w_j-\omega^rz_i} - \sum_{r=0}^{T-1} \sum_{\substack{k=1\\k\neq j}}^m \frac{\langle \alpha_{c(j)},\lsigma^r\alpha_{c(k)}\rangle}{w_j-\omega^rw_k} +
\frac{1}{w_j}
\left(\sum_{r=1}^{T-1} \frac{\langle \alpha_{c(j)},\lsigma^r\alpha_{c(j)}\rangle}{\omega^r - 1} + \langle \alpha_{c(j)}, \lambda_0 \rangle \right)
\label{tbeintro}\ee
for each $j\in \{1,\dots,m\}$.
Here $\lsigma$ is the induced action of $\sigma$ on weight space, $\h^*$, given by $\lsigma(\lambda)(h) = \lambda(\sigma^{-1} h)$. The second main result of the paper, Theorem \ref{evalthm}, is that provided these cyclotomic Bethe equations are satisfied then the corresponding Bethe vector $\psi_\Gamma$ is a simultaneous eigenvector of the full hierarchy of cyclotomic Gaudin Hamiltonians (with eigenvalues as defined in \S\ref{sec:evecs}).

\bigskip

Let us discuss the form of the equations \eqref{tbeintro}. The first two terms  are the natural cyclotomic analogs of the corresponding terms in the Bethe equations for the usual Gaudin model (for which see \cite{FFR}). Then there are two terms in $1/w_j$. The first of these, $\sum_{r=1}^{T-1} \langle \alpha_{c(j)},\lsigma^r\alpha_{c(j)}\rangle/(\omega^rw_j - w_j)$, can be regarded as a ``self-interaction'' between the Bethe root $w_j$ and its twist-images $\omega^rw_j$.
However, the final term in \eqref{tbeintro}, namely $\langle \alpha_{c(j)}, \lambda_0 \rangle/w_j$, is a more subtle new feature of the cyclotomic case. The weight $\lambda_0$ is given, we find, by
\begin{equation*}
\lambda_0(h)  := \sum_{r=1}^{T-1} \frac {\tr_\n (\sigma^{-r} \ad_h)} {1 - \omega^r} =  \sum_{r=1}^{T-1} \frac{1}{1 - \omega^r} \sum_{\substack{\alpha\in \Delta^+\\\sigma^r(\alpha)=\alpha}} \left( \prod_{p=0}^{r-1} \tau_{\sigma^p(\alpha)}^{-1} \right) \alpha(h),
\end{equation*}
where $\g=\n^-\oplus \h \oplus \n$ is a Cartan decomposition of $\g$; see \S\ref{sec: origin}. This definition depends on the data $(\g,\sigma,T)$ but not, for example, on the marked points $z_i$ or the choice of $\g$-modules assigned to them.
From the present perspective, this term originates in the need to preserve a crucial property of the Wakimoto construction when the global objects (the rational functions) are twist-equivariant. Namely, we need that any linear functional invariant under twist-equivariant rational functions valued in a certain Heisenberg Lie algebra is also, automatically, invariant under twist-equivariant rational functions valued in $\g$. See \S\ref{sec:functoriality} below, and also a companion paper \cite{VY} in which we discuss cyclotomic coinvariants in the more general framework of vertex Lie algebras. In \cite{VY} we establish in particular a result (Theorem \ref{thm:coinv}, below) that will be needed here. 
The upshot is that it is necessary to assign, to the origin, a certain carefully chosen one-dimensional module over the twisted loop algebra $(\h \otimes \C((t)))^\sigma$, and this gives rise to the final term in \eqref{tbeintro}.

\bigskip

This paper is structured as follows. The construction of the cyclotomic Gaudin algebra by means of coinvariants is given in \S\ref{ZTGsec}. 

In \S\ref{sec:wc} we recall the \emph{Wakimoto construction} (or \emph{free-field realization}) for $\Vcrit$, which is a homomorphism of vertex algebras $\rho:\Vcrit\to \WW_0$ to a Heisenberg vertex algebra $\WW_0$, and check that this homomorphism is equivariant with respect to $\sigma$, for a natural definition of $\sigma$ on $\WW_0$.

In \S\ref{sec:Hdef} we recall the definition of Wakimoto modules and then go on to use them to construct the Bethe vectors and prove (Theorem \ref{evalthm}) that they are simultaneous eigenvectors of the hierarchy of cyclotomic Gaudin Hamiltonains provided the cyclotomic Bethe equations are satisfied. 

In \S\ref{sec:examples} we discuss some implications and special cases of Theorem \ref{evalthm}. In particular we extract the explicit form of the eigenvalues of the quadratic cyclotomic Gaudin Hamiltonians \eqref{Ham intro}. 
In the special case when $\g$ is of type A, B or C, $\sigma$ is an inner automorphism of order $2$, and  the $V_i$ are vector representations, these eigenvalues were obtained by T. Skrypnyk \cite{Skrypnyk2}, who also found Bethe equations in agreement with \eqref{tbeintro}. 

Appendix \ref{sec:tert} contains the statement and proof of a $\Gamma$-equivariant version of the Strong Residue Theorem. In Appendix \ref{app:SV} we carry out diagrammatic calculations similar in spirit to those of \cite{SV} in order to establish the cyclotomic analog of the Schechtmann-Varchenko formula, given above, for the Bethe vector. 

\bigskip

We close this introduction by noting some open questions.
\medskip

The first concerns symmetries of the cyclotomic Gaudin model and the issue of \emph{completeness} of the Bethe ansatz. The Gaudin algebra commutes with the copy $\Delta^N\g$ of $\g$ in $U(\g)^{\otimes N}$ and, when the $V_i$ are Verma modules, Bethe vectors $\psi\in \bigotimes_{i=1}^N V_i$ are singular (where singular now means singular for  $\g$, i.e. $\n.\psi=0$). The Bethe ansatz is said to be complete, for a given collection of marked points $z_i$ and representations $V_i$, if the Bethe vectors form a basis of the space of singular vectors of $\bigotimes_{i=1}^N V_i$. Completeness is known to hold in many cases, but not in all \cite{MV3}. Now, the cyclotomic Gaudin algebra commutes only with stable subalgebra $\g^\sigma := \{X\in \g: \sigma X= X\}$. We expect, but have not proved here, that the cyclotomic Bethe vectors are singular for $\g^\sigma$. Assuming this is so, the interesting question is then whether they form a basis of such singular vectors in $\bigotimes_{i=1}^NV_i$. If they do then the system of cyclotomic Bethe equations would then provide a new way to study the decomposition of tensor products of $\g$-modules into their irreducible components with respect to $\g^\sigma$. In particular in the one-point case $N=1$, one has a new approach to the branching rules from  $\g$ to $\g^{\sigma}$. Moreover, usually the bases provided by Bethe vectors have additional meaning -- they are for example related in certain limits to crystal bases \cite{VarchenkoCrystal} -- and there should be analogous statements in the cyclotomic case.

\medskip

As previously noted, one can regard the Gaudin Hamiltonians \eqref{Ham intro} as being associated to certain non-skew-symmetric solutions of the classical Yang-Baxter equation. To see why that is, recall first the construction of the generalised Gaudin Hamiltonians associated to skew-symmetric solutions of the classical Yang-Baxter equation. Let $r(u, v)$ be a $(\g \otimes \g)$-valued function depending on two complex parameters $u, v \in \C$ such that $r_{12}(u, v) = - r_{21}(v, u)$ and satisfying the classical Yang-Baxter equation
\begin{equation} \label{CYBE intro}
\big[ r_{12}(z_1, z_2), r_{13}(z_1, z_3) \big] + \big[ r_{12}(z_1, z_2), r_{23}(z_2, z_3) \big] + \big[ r_{32}(z_3, z_2), r_{13}(z_1, z_3) \big] = 0.
\end{equation}
One may associate to this classical $r$-matrix the following family of $N$ elements in $U(\g)^{\otimes N}$:
\begin{equation*}
\mathcal{H}_i^r = \sum_{\substack{j=1\\ j \neq i}}^N r_{ji}(z_j, z_i).
\end{equation*}
In particular, the fact that these Hamiltonians are mutually commuting is a direct consequence of the classical Yang-Baxter equation \eqref{CYBE intro}.
More generally, Gaudin models can also be associated to non-skew-symmetric $r$-matrices; that is, to $(\g \otimes \g)$-valued functions $r(u,v)$ satisfying the classical Yang-Baxter equation but not the skew-symmetry condition $r_{12}(u, v) = - r_{21}(v, u)$ \cite{Skrypnyk1}. 
We expect the cyclotomic Gaudin Hamiltonians \eqref{Ham intro} to be associated to the following non-skew-symmetric $r$-matrix 
\begin{equation*}
r(u,v) = \sum_{p=0}^{T-1} \frac{\sigma^p I^a \otimes I_a}{\omega^{-p} u - v} = \frac{I^a \otimes I_a}{u - v} + r^0(u, v), \qquad \text{where} \quad r^0(u, v) = \sum_{p=1}^{T-1} \frac{\sigma^p I^a \otimes I_a}{\omega^{-p} u - v}.
\end{equation*}
The appearance of this non-skew-symmetric $r$-matrix in our construction stems from the particular choice we made of complementary subalgebra to $\bigoplus_{i=1}^N \g \otimes \C[[t - z_i]]$. Indeed, the above $r$-matrix corresponds to this decomposition of $\bigoplus_{i=1}^N \g \otimes \C((t - z_i))$ under the Adler-Kostant-Symes scheme \cite{AKS}.

\medskip 

The quadratic cyclotomic Gaudin Hamiltonians are closely related to a system of cyclotomic KZ equations. For inner automorphisms $\sigma$ at least, such systems have been introduced and studied in \cite{Brochier}. It would be very interesting to see whether the cyclotomic Schechtman-Varchenko formula \eqref{SVintro} can be used in to construction solutions to these equations, generalizing the usual case.

\medskip

Finally, while we work in the present paper with the (relatively) concrete formulation given in \cite{FFR}, the solution to the Gaudin model by Bethe ansatz has since been recast in the geometrical language of opers and Miura opers -- for a review see \cite{Freview} -- and it would be interesting to understand the cyclotomic Bethe ansatz above in this language.

\begin{ack}
We would like to thank P. Etingof and A. Varchenko for helpful suggestions and comments. 
We are grateful to E. Mukhin for interesting discussions. 
\end{ack}

\section{The cyclotomic Gaudin model}\label{ZTGsec}

Fix a $T\in \Z_{\geq 1}$ and pick a primitive $T$th root of unity $\omega\in \Cx$. Let $$\Gamma := \{1,\omega,\omega^2,\dots,\omega^{T-1}\},$$ which is a copy of the cyclic group $\Z_T\cong \Z/T\Z$ of order $T$. $\Gamma$ acts on $\C$ by multiplication, $(\omega,z) \mapsto \omega z$. 
We choose a collection $\bm z= \{z_1,\dots,z_N\}$ of $N\in \Z_{\geq 1}$ non-zero points in the complex plane whose $\Gamma$-orbits are pairwise disjoint:
\be \Gamma z_i\cap  \Gamma z_j = \emptyset \quad\text{for all $1\leq i\neq j\leq N$}. \ee
Note that the condition $z_i\neq 0$ is equivalent to the demand that $\Gamma$ act freely on $\Gamma z_i$.

Let $\g$ be a simple Lie algebra over $\C$, and $\langle\cdot,\cdot\rangle$ the non-degenerate invariant inner product on $\g$ with respect to which the square of the length of the long roots is 2. 

In this section we construct the Hamiltonians of \emph{cyclotomic Gaudin models}. 
These Hamiltonians generate a large commutative subalgebra $\Gaud^\Gamma_{\bm z}$ of $U(\g)^{\otimes N}$. In the special case $\Gamma=\{1\}$ we recover the Gaudin algebra $\Gaud_{\bm z}$ as in e.g. \cite{Freview}.

\subsection{Notation: Formal variables and series expansions}\label{defssec}
We work over $\C$. 
Let $t$ be a formal variable. We write $\C[t]$ for the ring of polynomials in $t$, $\C[[t]]$ for the ring of formal power series, and $\C((t))$ for the field of formal Laurent series. 
%As usual, $\C[[t]]$ is by definition the completion of $\C[t]$ with respect to the topology in which $t^k\C[t]$, $k=0,1,2,\dots$, are a base of the open sets containing $0$, and $\C((t))$ is the localization of $\C[[t]]$ at the zero ideal. 
Let $\C(t)$ denote the field of fractions of $\C[t]$,
%, i.e. the localization of $\C[t]$ at the zero ideal. 
i.e. the field of rational functions of $t$ with complex coefficients.

Given a finite set $\bm x=\{x_1,\dots,x_p\}$ of $p\in \Z_{\geq 1}$ pairwise distinct points in the complex plane, we write $\C_{\bm x}(t)\subset \C(t)$ for the localization of $\C[t]$ by the multiplicative subset generated by $t-x_i$, $1\leq i\leq p$. The elements of $\C_{\bm x}(t)$ are rational functions of $t$ that have poles at most at the points $x_i$. They form, in particular, a $\C$-algebra. Let $\C^\8_{\bm x}(t)$ be the subalgebra consisting of rational functions of $t$ that, in addition, vanish at $\8$. That is,
\be \C^\8_{\bm x}(t) = \left\{\frac{p(t)}{q(t)} : p(t), q(t)\in \C[t],\, \deg p < \deg q,\, q(z)\neq 0 \text{ for all } z\in \C\setminus \bm x  \right\}.\ee

For any complex number $z$, define 
\be \iota_{t-z} : \C(t) \to \C((t-z)) \ee
to be the map that returns the formal Laurent expansion about the point $z$. Given $f(t)\in \C(t)$, $\iota_{t-z}f(t)$ is by definition computed by replacing every occurrence of $t$ by $(t-z)+z$ and then expanding in powers of $(t-z)$, which one is to regard as a new formal variable, the ``formal local coordinate at $t=z$''. (If $z$ is not a pole of $f(t)$ then the result is actually in $\C[[t-z]]$.) 

The \emph{residue map} $\res_{t-z} : \C((t-z)) \to \C$ is defined by
\be \res_{t-z} \sum_{k=-K}^\8 a_k (t-z)^k = a_{-1}.\ee
(For brevity we shall sometimes write the map $\res_{t-z}\circ\iota_{t-z}$ from $\C(t)\to \C$ just as $\res_{t-z}$.)

We write $f'$ for the derivative of an element $f$ of $\C((t-z))$, or $\C(t)$, with respect to its argument.

\subsection{The algebra $\gh_N$.}
Consider any one of the points $z_i$, $i=1,\dots,N$. 
Let $\gh_{(i)}$ denote the copy of the affine algebra $\gh$ obtained by taking the extension of $\g\otimes\C((t-z_i))$, by a one-dimensional centre $\C K_i$, defined by the co-cycle
\be \Omega_i(f_{z_i},g_{z_i}) = \res_{t-z_i} \langle f_{z_i}, g'_{z_i} \rangle K_i,\qquad f_{z_i},g_{z_i}\in \g\otimes\C((t-z_i)).\ee
Concretely, given $a\in \g$ and $n\in \Z$, we shall write $a \otimes (t-z_i)^n\in \g\otimes\C((t-z_i))$ as $a[n]^{(i)}$, or just $a[n]$ when there is no risk of ambiguity. Then the commutation relations of $\gh_{(i)}$ are
\begin{equation} \label{com rel gh}
\left[ a[n], b[m]\right] = [a,b][n+m] + n \langle a,b\rangle \delta_{n+m,0} K_i, \qquad [K_i,\gh_{(i)}]=0.
\end{equation}

Next, let $\gh_N$ denote the extension of the direct sum $\bigoplus_{i=1}^N \g\otimes\C((t-z_i))$, by a one-dimensional centre $\C K$, defined by the cocycle 
\begin{equation} \label{cocycle}
\Omega(f, g) :=  \sum_{i = 1}^N \res_{t-z_i} \langle f_{z_i},  g'_{z_i} \rangle\, K.
\end{equation}
where $f = (f_{z_i})_{1\leq i\leq N}$ and $g = (g_{z_i})_{1\leq i\leq N}$ are in $\bigoplus_{i=1}^N \g\otimes\C((t-z_i))$. 
In other words, $\gh_N$ is the quotient of the direct sum $\bigoplus_{i=1}^N \gh_{(i)}$ by the ideal spanned by $K_i - K_j$, $i= 1, \ldots, N$, so that all the central generators $K_i$ are identified to a single one
which we then call $K$. 
Thus, the commutation relations of $\gh_N$ are, explicitly,
\begin{equation} \label{Kac-Moody alg}
\left[ a[n]^{(i)}, b[m]^{(j)}\right] = \delta_{ij} \left([a,b][n+m]^{(i)} + n \langle a,b\rangle \delta_{n+m,0} K\right), \qquad [K,\gh_N]=0.
\end{equation}

\subsection{Induced $\gh_N$-modules}\label{sec: indg}
Let now $\M_{(i)}$ be a $\g$-module, for each $i=1,\dots, N$, and define 
\be \M_N :=  \M_{(1)} \otimes\dots\otimes \M_{(N)}.\ee 
We think of $\M_{(i)}$ as being \emph{assigned} to the point $z_i$, and we turn $\M_{(i)}$ into a module $\M^k_{(i)}$ over $\g\otimes\C[[t-z_i]] \oplus \C K_i$ by declaring that $\g\otimes(t-z_i)\C[[t-z_i]]$ acts trivially and $K_i$ acts by multiplication by $k\in \C$. 
Then we define $\MM^k_{(i)}$, the induced representation of $\gh_{(i)}$ of level $k$, as follows:
\be
\MM_{(i)}^k = \Ind_{\g\otimes\C[[t-z_i]] \oplus \C K_i}^{\gh_{(i)}} \M_{(i)}^k :=  U(\gh_{(i)})\otimes_{U(\g\otimes\C[[t-z_i]] \oplus \C K_i)} \M_{(i)}^k.
\ee
Similarly, the tensor product $\M^k_{N} := \M^k_{(1)} \otimes\dots\otimes \M^k_{(N)}$ is a module over the Lie subalgebra
\be
\gh_N^+ := \bigoplus_{i=1}^N \g\otimes\C[[t-z_i]] \oplus \C K\label{gnpdef}
\ee
of $\gh_N$, and we have the induced $\gh_N$-module
\be
\MM^k_N := \Ind_{\gh_N^+}^{\gh_N} \M^k_{N} := U(\gh_N) \otimes_{U(\gh_N^+)} \M^k_N .\label{vvdef}
\ee
There is an isomorphism of $\gh_N$-modules
\be
\MM^k_N \cong \MM^k_{(1)} \otimes \dots \otimes \MM^k_{(N)}. 
\ee

\subsection{Complements of $\gh_N^+$ in $\gh_N$} \label{sec: complements}
The following lemma relates the induced module $\MM^k_N$ back to $\M^k_N$.
\begin{lem}\label{complem}
Suppose $\p$ is any Lie algebra for which there is an embedding $\p\hookrightarrow \gh_N$ of Lie algebras such that $\gh_N\cong_\C\gh_N^+\oplus \p$ as vector spaces. Then there is an isomorphism of vector spaces
\be \MM^k_N \big/ \p \cong_\C \M_N, \ee
where
\begin{equation} \label{coninv p}
\MM^k_N \big/ \p := \MM^k_N \Big/ \left(\p. \MM^k_N\right)
\end{equation}
is called the space of \emph{coinvariants with respect to $\p$}. 
\end{lem}
\begin{proof}
By virtue of the decomposition $\gh_N\cong_\C \gh_N^+ \oplus \p$, we have the isomorphism of vector spaces $U(\gh_N) \cong_\C U(\p) \otimes_{\C} U(\gh^+_N)$ (see for instance \cite[Proposition 2.2.9]{Dixmier}) so that, as a $\p$-module, $\MM^k_N$ is isomorphic to the free module $U(\p) \otimes_{\C} \M^k_N$ generated by $\M^k_N$. Hence  $\MM^k_N \big/ \p \cong_\C \M^k_N \cong_\C \M_N$.
\end{proof}
It is clear that one possible choice of $\p$ is the Lie subalgebra $\gh_N^-:=\bigoplus_{i=1}^N \g\otimes (t-z_i)^{-1}\C[(t-z_i)^{-1}]$.
However, it turns out to be more interesting to consider, instead, choices of $\p$ that are ``global''. To indicate roughly what is meant by ``global'', let us first recall the strong residue theorem in the present language (cf. \cite[\S9.2.9]{FB}).

\subsection{Local and global data, and the strong residue theorem}
A rational function in $\C^\8_{\bm z}(t)$, \S\ref{defssec}, is to be thought of as a ``global'' object, while an element of $\C((t-z_i))$ is  ``local'' data associated to the point $z_i$. There is a map
\be \iota: \C^\8_{\bm z}(t) \longhookrightarrow \bigoplus_{i=1}^N \C((t-z_i));\quad
    f(t) \longmapsto (\iota_{t-z_1}f(t),\dots,\iota_{t-z_N}f(t)) \ee
that associates to the global object $f(t)$ a tuple of local data: its Laurent-expansions. It is clear that $\ker \iota = \{0\}\subset \C^\8_{\bm z}(t)$, so this map is an injection. 
One can ask which tuples of local data can be ``globalized'', i.e. which ones lie in the image in the image of $\iota$. The strong residue theorem answers this question.
\begin{lem}[Strong residue theorem]\label{srt} An element $(f_1,\dots,f_N)\in \bigoplus_{i=1}^N \C((t-z_i))$ is in $\iota(\C^\8_{\bm z}(t))$ if and only if $$0=\sum_{i=1}^N \res_{t-z_i} f_i\, \iota_{t-z_i}(g)$$
for every $g\in \C^\8_{\bm z}(t)$.
\end{lem}
\begin{proof} For the ``only if'' direction, we must check that  
\be\label{ii} 0=\sum_{i=1}^N \res_{t-z_i}\iota_{t-z_i} f(t) g(t)
%=\sum_{i=1}^N \res_{t-z_i}\iota_{t-z_i} f(t) \iota_{t-z_i} g(t) 
\ee 
for all $f(t),g(t)\in \C^\8_{\bm z}(t)$. To see this, consider writing $f(t)g(t)$ in its partial fraction decomposition and then taking the large-$t$ expansion. The coefficient of $t^{-1}$ is the sum in \eqref{ii}, and it must vanish because $f,g\in \C^\8_{\bm z}(t)$ implies $f,g\sim \mc O(1/t)$  and hence $fg \sim \mc O(1/t^2)$ for large $t$.\footnote{We choose to work with a formal variable $t$. To connect the weak (i.e. ``only if'') direction with the usual complex-analytic residue theorem, note that if $f(t),g(t)\in \C_{\bm z}^\8(t)$ are viewed as meromorphic functions of a complex variable $t$ then $f(t)g(t)dt$ is a one-form on $\C P^1$ with no pole at $\8$ (since $f(t)g(t)$ has a double zero there). So $0=\sum_{i=1}^N \res_{t-z_i} \iota_{t-z_i}f(t)g(t)$ is indeed the statement that the sum of all residues of this meromorphic one-form, $f(t)g(t)dt$, on $\C P^1$ is zero. 
} 
 
For the ``if'' direction, first observe that 
\be\bigoplus_{i=1}^N \C((t-z_i))\cong_\C \iota(\C^\8_{\bm z}(t)) \oplus \bigoplus_{i=1}^N \C[[t-z_i]].\label{gl}\ee
Indeed, let $f_i^-\in (t-z_i)^{-1} \C[(t-z_i)^{-1}]$ denote the pole part of $f_i\in \C((t-z_i))$; then $(f_1,\dots,f_N)$ splits uniquely as the direct sum of the function $f(t) := \sum_{i=1}^N f_{i}^-(t-z_i)\in \C^\8_{\bm z}(t)$ and the tuple $(f_1 - \iota_{t-z_1}f,\dots,f_N-\iota_{t-z_N}f) \in \bigoplus_{i=1}^N\C[[t - z_i]]$. Now from \eqref{ii}  we have that
\be \sum_{i=1}^N \res_{t-z_i} f_i  \iota_{t-z_i} (g) =\sum_{i=1}^N \res_{t-z_i} \left(f_i - \iota_{t-z_i}f\right) \iota_{t-z_i} g.\ee
Since $f_i - \iota_{t-z_i}f \in \C[[t - z_i]]$ one sees, by considering $g = 1/(t-z_i)^k\in \C^\8_{\bm z}(t)$ for every $1\leq i\leq N$ and every $k\in \mathbb{Z}_{>0}$, that this vanishes only if $f_i -\iota_{t-z_i}f=0$ for each $i$, as required.
\end{proof}
Note that \eqref{gl} says that any tuple of Laurent series can be split uniquely into a ``global'' part, living in $\C^\8_{\bm z}(t)$, and a residual tuple of Taylor series.

Now it follows from Lemma \ref{srt}  that the cocycle \eqref{cocycle} vanishes on restriction to the image of the embedding $\id\otimes \iota:\g\otimes \C^\8_{\bm z}(t)\hookrightarrow \bigoplus_{i=1}^N \g\otimes\C((t-z_i))$.
Consequently $\id\otimes \iota$ lifts to an embedding
\be\nn \g\otimes \C^\8_{\bm z}(t)\hookrightarrow \gh_N,\ee
and moreover, given \eqref{gl}, the image of this map is a complementary subspace to $\gh_N^+$ in $\gh_N$. This choice, $\p=\g\otimes \C^\8_{\bm z}(t)$, was the choice of complementary subspace $\p$ made in \cite{FFR}. 

In the present paper, we again choose a complementary subspace consisting of ``global'' objects; but in place of  $\g\otimes \C^\8_{\bm z}(t)$, the global objects we use will be $\Gamma$-equivariant, as follows.

\subsection{$\Gamma$-equivariant global data}

Let now $\sigma:\g\to \g$ be an automorphism whose order divides $T$. 
Define 
\be \g_{\bm z}^{\Gamma,k} :=   \{ f\in \g\otimes \C^\8_{\Gamma\bm z}(t): \sigma f(t) = \omega^{-k} f(\omega t) \} \ee 
and in particular
%\be \gs_{\bm z} := \left\{ f\in \g\otimes \C^\8_{\Gamma\bm z}(t) : \sigma f(t) = f(\omega t) \right\}\ee 
\be \g_{\bm z}^{\Gamma,0} = \gs_{\bm z} :=   \left(\g\otimes \C^\8_{\Gamma\bm z}(t)\right)^\Gamma.\label{gsdef}\ee 
That is, $\gs_{\bm z}$ is the Lie algebra of $\g$-valued rational functions $f$ in the formal variable $t$ that vanish at infinity, that have poles at most at the points $\{ \omega^k z_i: 1\leq k\leq T,\,1\leq i \leq N\}$, and that obey the condition $\sigma f(t) = f(\omega t)$. 

There is an injection
\be \iota: \g_{\bm z}^{\Gamma,k}  \longhookrightarrow \bigoplus_{i=1}^N \g\otimes \C((t-z_i));\quad
    f(t) \longmapsto (\iota_{t-z_1}f(t),\dots,\iota_{t-z_N}f(t)).\label{tin}\ee
Note that the ``local'' data is the same as before: we still merely take the Laurent expansions at the marked points $z_i$. We then have the following analog of Lemma \ref{srt}. 

\begin{lem}\label{tert}
An element $(f_1,\dots,f_N)\in \bigoplus_{i=1}^N \g\otimes \C((t-z_i))$ is in $\iota(\g_{\bm z}^{\Gamma,k})$ if and only if $$0=\sum_{i=1}^N  \res_{t-z_i} \langle f_i, \, \iota_{t-z_i}(g)\rangle$$
for every $g\in \g_{\bm z}^{\Gamma,-k-1}$.
\end{lem}
\begin{proof} See Appendix \ref{sec:tert}, where a generalization of this lemma is proved.
\end{proof}
\begin{cor}\label{gGcor} There is an embedding of Lie algebras $\gs_{\bm z}\hookrightarrow \gh_N$ such that $\gh_N\cong_\C \gh_N^+ \oplus \gs_{\bm z}$ as vector spaces.
\end{cor}
\begin{proof} 
Note that if $g\in\g_{\bm z}^\Gamma$ then $\del_t g\in \g_{\bm z}^{\Gamma,-1}$. Hence, by the ``only if'' direction of Lemma \ref{tert}, the cocycle \eqref{cocycle} vanishes identically when restricted to the image of the embedding $\g_{\bm z}^\Gamma  \longhookrightarrow \bigoplus_{i=1}^N \g\otimes \C((t-z_i))$. Therefore this embedding lifts to an embedding  $\gs_{\bm z}\hookrightarrow \gh_N$. It is straightforward to check that  $\gh_N\cong_\C \gh_N^+ \oplus \gs_{\bm z}$ (compare \eqref{gll} in Appendix \ref{sec:tert} below).  
\end{proof}
Combining this corollary with Lemma \ref{complem}, we can regard $\M_N$ as the space of coinvariants of $\MM^k_N$ under the action of the Lie algebra $\gs_{\bm{z}}$: 
\begin{equation} \label{coinvariant T}
 \MM^k_N \big/ \gs_{\bm{z}}\cong_\C \M_N .
\end{equation}

\subsection{The cyclotomic swapping procedure}\label{sec:twistedswapping}

We are now in a position to define the Hamiltonians of the cyclotomic Gaudin model.
To the points $z_1,\dots,z_N$ we add a further non-zero marked point, say $u \in \Cx$. We treat $u$ as though it were an additional $z_i$ in the above construction, in the sense that we require the points $z_1,\dots,z_N,u$ to have pairwise disjoint $\Gamma$-orbits and we assign to the point $u$ a copy of $\gh$ -- say $\gh_{(u)}$, with central generator $K_{(u)}$ -- and alter accordingly the definition of $\gh_N$. However, whereas the $\g$-modules $\M_{(i)}$ assigned to the points $z_i$ are thus far unspecified, to the point $u$ we assign the trivial one-dimensional $\g$-module $\C \vac$ generated by a vector $\vac$ with $\g\on \vac =0$. By declaring that  $(\g \otimes \C[[t-u]])\on \vac = 0$ and $K_{(u)}\on \vac = k\vac$ we make $\C v_0$ into a module over $\g \otimes \C[[t-u]]\oplus \C K_{(u)}$. The induced $\gh_{(u)}$ module,
\be \VV_0^k  = \Ind_{\g\otimes\C[[t-u]] \oplus \C K_{(u)}}^{\gh_{(u)}} \C \vac :=  U(\gh_{(u)})\otimes_{U(\g\otimes\C[[t-u]] \oplus \C K_{(u)})} \C \vac,\ee
is called the \emph{vacuum Verma module at level $k$}. 

We now have the Lie algebra $\gs_{\bm z,u}:=  \left(\g\otimes \C^\8_{\Gamma\bm z\cup\Gamma u}(t)\right)^\Gamma$, and isomorphisms of vector spaces
\begin{equation} \label{coinvariant}
\big( \MM^k_N\otimes \VV^k_0 \big) \big/ \g_{\bm{z}, u}^{\Gamma} \cong_\C \M_N \otimes \C \vac \cong_\C \M_N .
\end{equation}
That is, the space of coinvariants is \emph{again} $\M_N$.
This allows one to construct from any $X\in \VV^k_0$ an endomorphism 
\be X(u) : \M_N  \to \M_N,\label{Xu}\ee 
as follows: 
\be X(u) : \M_N  \hookrightarrow \MM^k_N \underset{\cdot\otimes X}{\longrightarrow} \MM^k_N \otimes \VV^k_0 \twoheadrightarrow \big( \MM^k_N \otimes \VV^k_0 \big) \big/ \g_{\bm{z}, u}^{\Gamma}  \cong_\C \M_N ,\ee
where $\M_N \hookrightarrow \MM^k_N$ is the natural embedding. 
This map has the property that 
\be \left[\left(X\left(u\right)\on x\right) \otimes \vac\right] = \left[x\otimes X\right].\ee 
where we write $[\,\,\cdot\,\,]$ for the equivalence class.

To understand how, given $X\in \VV_0^k$, the endomorphism $X(u) \in \End(\M_N)$ is actually to be computed, let $A \in \g$ and consider the $\Gamma$-equivariant rational function 
\begin{equation} \label{swap function}
f(t) = \sum_{k = 0}^{T-1} \frac{\sigma^k A}{(\omega^{-k} t - u)^n} \in \gs_{\bm z,u}.
\end{equation}
This function \eqref{swap function} is regular at $z_i$ and its expansion there reads
\begin{equation*}
\iota_{t-z_i}f(t) = - \pdd \sum_{k=0}^{T-1} \sum_{p=0}^{\infty} \frac{\omega^k}{(\omega^k u - z_i)^{p+1}} (\sigma^k A)[p]^{(i)} \in \g\otimes\C[[t-z_i]],
\end{equation*}
where, recall, $A[p]^{(i)} = A^{(i)} \otimes  (t - z_i)^p \in \g\otimes\C((t-z_i))$. 
On the other hand, the expansion of \eqref{swap function} at $u$ has a singular term:
\begin{equation*}
\iota_{t-u}f(t) = A[-n] - \pdd \sum_{k=1}^{T-1} \sum_{p=0}^{\infty} \frac{\omega^{k n}}{(\omega^k - 1)^{n+p} u^{p+1}} (\sigma^k A)[p]  \in \g\otimes\C((t-u)),
\end{equation*}
where $A[p] = A \otimes (t - u)^p \in \g\otimes\C((t-u))$. For all $Y\in \MM^k_N$ and all $X\in \VV_0^k$ we have by definition $[f\on (Y\otimes X)]= 0$ or, equivalently but more explicitly,
\begin{align}
\left[ Y \otimes A[-n]\on X\right] &=  \left[ \pdd  \sum_{k=0}^{T-1} \sum_{p=0}^{\infty} \frac{\omega^k}{(\omega^k u - z_i)^{p+1}} (\sigma^k A)[p]^{(i)} \on Y\otimes X  \right] \nn\\ & \quad
+ \left[ Y \otimes  \pdd \sum_{k=1}^{T-1} \sum_{p=0}^{\infty} \frac{\omega^{k n}}{(\omega^k - 1)^{n+p} u^{p+1}} (\sigma^k A)[p]\on X\right].\nn\end{align}
In particular, for all $x\in \M_N \hookrightarrow \MM^k_N$,  
\begin{align}
\left[ x \otimes A[-n]\on X\right] &=  \left[  \pdd \sum_{k=0}^{T-1} \frac{(\sigma^k A)^{(i)} \on x}{u - \omega^{-k}z_i}  \otimes X  \right] \label{twistedswap}\\ & \quad
+ \left[ x \otimes \pdd \sum_{k=1}^{T-1} \sum_{p=0}^{\infty} \frac{\omega^{k n}}{(\omega^k - 1)^{n+p} u^{p+1}} (\sigma^k A)[p]\on X\right]
\nn\end{align}
The space $\VV_0^k$ is spanned by vectors of the form 
\be A_1[-n_1]A_2[-n_2]\dots A_k[-n_k]\vac,\qquad A_i\in \g,\, n_i\in \Z_{>0},\, 1\leq i\leq k.\label{genvec}\ee 
and there is a natural $\Z$-gradation on $\VV_0^k$ in which the vector \eqref{genvec} has grade $-\sum_{i=1}^k n_i$. In the identity \eqref{twistedswap}, which we shall call the \emph{cyclotomic swapping identity}, both $X$ and $(\sigma^kA)[p]X$, $p\geq 0$, have grades strictly lower than $A[-n]X$. Thus, by applying \eqref{twistedswap} a finite number of times, any class $[x\otimes X]$ can be expressed as a linear combination of classes of the form $[x'\otimes \vac]$, which amounts to computing the map $X(u)$. Call this procedure the \emph{cyclotomic swapping procedure}.

\begin{rem}\label{sirem}
The first term on the right-hand side of \eqref{twistedswap} is the obvious analog of the result of swapping the generator $A[-n]$ over to the spin chain $\M_N$ in the usual case, cf. equation (3.5) in \cite{FFR}. A more subtle new feature of the cyclotomic case is the presence of the extra terms on the right-hand side coming from the ``self-interactions'' between the poles at $u, \omega u, \ldots, \omega^{T-1} u$. To stress the point: since these terms are all of the form $\sigma^k A[p] X$ with $p \geq 0$, they are of lower grade than the initial term $A[-n] X$, and hence by a procedure of repeatedly rewriting these terms in canonical form \eqref{genvec} using the commutation relations of $\gh_{(u)}$ and then ``swapping off'' the outermost factor, a finite number of iterations always suffices to reach the vacuum state $\vac$. But, in general, doing so produces many (apparently intricate) correction terms to the naive result. 
\end{rem}
\subsection{Singular vectors, $\mf Z(\gh)$, and commuting Hamiltonians}\label{singsec}
A vector $X\in \VV_0^k$ is said to be \emph{singular} (or more fully, a \emph{singular vector of imaginary weight}) if
\be A\on X=0\quad\text{ for all } \quad A\in \g \otimes \C[t].\label{singdef}\ee

The singular vectors form a linear subspace of $\VV_0^k$ denoted $\mf Z(\gh)$. 
The crucial thing to check is that the argument given in \cite[Proposition 2]{FFR} to show the commutativity of all maps $\M_N\to \M_N$ associated to singular vectors still goes through in the cyclotomic set-up given above.
\begin{prop}\label{prop:Zcom}
Let $Z_1,Z_2\in \mf Z(\gh)$. Then for any pair of non-zero complex numbers $u$ and $v$ such that the points $z_1,\dots,z_N,u,v$ have pairwise disjoint $\Gamma$-orbits, the corresponding linear operators $Z_1(u)$ and $Z_2(v)$ on $\M_N$ commute. 
\end{prop}
\begin{proof}
Assign to each of the points $u$ and $v$ a copy of $\VV_0^k$. Then 
\be \big(\MM^k_N\otimes \VV_0^k \otimes \VV_0^k \big)\big/ \g_{\bm z,u,v}^\Gamma \cong_\C \M_N \otimes \C \vac \otimes \C\vac \cong_\C \M_N. \ee
Hence, to any $X,Y\in \VV^k_0$ we can associate an endomorphism $(X,Y)(u,v) : \M_N \to \M_N$ defined by
\be (X,Y)(u,v) : \M_N \hookrightarrow \MM^k_N \xrightarrow[\cdot\otimes X\otimes Y]{} \MM^k_N\otimes \VV_0^k\otimes \VV_0^k \twoheadrightarrow \big( \MM^k_N \otimes \VV_0^k \otimes \VV_0^k\big) \big/ \g_{\bm{z}, u,v}^\Gamma  \cong_\C \M_N.\ee
This map has the property that 
\be \left[\left((X,Y)\left(u,v\right)\on x\right) \otimes \vac\otimes \vac \right] 
   = \left[x\otimes X\otimes Y\right].\ee 
We now claim that if $X$ and $Y$ are both singular, then $(X,Y)(u,v) = X(u)Y(v)$ and  $(X,Y)(u,v) = Y(v)X(u)$, and hence  $[X(u),Y(v)]=0$ as required. 

Indeed, consider starting with $[x\otimes X\otimes Y]$, writing $X$  as a linear combination of terms of the form \eqref{genvec}, and then ``swapping'' the factors of $X$ as discussed in \S\ref{sec:twistedswapping}. In addition to the terms on the right of \eqref{twistedswap}, ``swapping'' $A[-n]$ now also produces the following term acting on the copy of $\VV_0^k$ assigned to the point $v$: 
\be \pdd  \sum_{k=0}^{T-1} \sum_{p=0}^{\infty} \frac{\omega^k}{(\omega^k u - v)^{p+1}} (\sigma^k A)[p]^{(v)}. 
\ee
But, by definition of a singular vector, $Y$ is annihilated by this term if $Y$ is singular. In this way $$[x\otimes X\otimes Y] = [X (u)\on x\otimes \vac \otimes Y]= [Y(v)X(u)\on x\otimes \vac\otimes \vac],$$ where in the second step we swap from the point $v$ as usual (noting that the state $\vac$ is of course a singular vector). 
\end{proof}

It is known (see \cite{FFR}) that the space of singular vectors $\mf Z(\gh)$ becomes very large when $k$ takes the value $-h^\vee$, where $h^\vee$ is the dual Coxeter number of $\g$. This value is called the \emph{critical level} and $\VV_0^{-h^\vee}$ the \emph{vacuum Verma module at the critical level}. In the remainder of the paper we always take $k=-h^\vee$.\footnote{Recall that $\langle \cdot,\cdot\rangle = \frac{1}{2h^\vee} (\cdot,\cdot)_K$ where $(\cdot,\cdot)_K$ is the Killing form $(X,Y)_K= \tr_\g \ad_X \ad_Y$ \cite{KacBook}. So if one works throughout with respect to $(\cdot,\cdot)_K$ then the critical level is $-\half$ for all simple $\g$.}

\subsection{The cyclotomic Gaudin algebra}\label{GGA}
Thus far the modules $\M_{(i)}$ assigned to the points $z_i$ have been left unspecified. If we now take each $\M_{(i)}$ to be a copy of $U(\g)$, regarded as a left $U(\g)$-module, then to each $X\in \VV_0^k$  we have a linear map $X(u):U(\g)^{\otimes N} \to U(\g)^{\otimes N}$ and, cf. \S\ref{sec:twistedswapping}, this map is constructed using the  left action of $U(\g^{\oplus N}) \cong U(\g)^{\otimes N}$ on itself. 
But $U(\g)$ is not merely a left $U(\g)$-module but a $U(\g)$-bimodule, i.e. $U(\g)$ acts on itself from the left and from the right and these actions commute. Therefore the maps $X(u)$ commute with the right action of  $U(\g)^{\otimes N}$ on itself. It follows that $X(u)$ must act by left multiplication by a fixed element of $U(\g)^{\otimes N}$, which by abuse of notation we also call $X(u)$. (For indeed, if $1^{\otimes N}\xmapsto{X(u)} X(u)$ then, multiplying on the right by any $a\in U(\g)^{\otimes N}$, we have $a\xmapsto{X(u)} X(u) a$.) 
By construction, for any choices of the modules $\mc M_{(i)}$, $1\leq i\leq N$, the map $X(u):\mc M_N\to \mc M_N$ of \eqref{Xu} agrees with the left action of this element $X(u)\in U(\g)^{\otimes N}$ on $\mc M_N$.

For each $X\in \Vcrit$, $X(u)$ is a meromorphic $U(\g)^{\otimes N}$-valued function of the complex variable $u$. 

Let $\mathscr Z^\Gamma_{\bm z}(\g,\sigma)$ denote the span, in $U(\g)^{\otimes N}$, of all the coefficients of singular terms of Laurent expansions of the operators $Z(u)$ as $Z$ varies in the space of singular vectors $\mf Z(\gh)\subset \Vcrit$. At this stage, we have established the following.
\begin{thm}\label{Zthm}
$\mathscr Z^\Gamma_{\bm z}(\g,\sigma)$ is a commutative subalgebra of $U(\g)^{\otimes N}$.\qed
\end{thm}
Let us call $\mathscr Z^\Gamma_{\bm z}(\g,\sigma)$ the \emph{cyclotomic Gaudin algebra}. It depends on the choice of marked points $\bm z = \{z_1,\dots, z_N\}$, the cyclic group $\Gamma\cong \Z/T \Z$, and the automorphism $\sigma:\g\to \g$. In the special case $\Gamma=\{1\}$ one recovers the Gaudin algebra $\mathscr Z_{\bm z}(\g)$ of \cite{Freview}.

\subsection{The quadratic Hamiltonians}\label{sec:quadH}
Let $I_a\in \g$, $a=1,\dots,\dim(\g)$, be a basis of $\g$ and $I^a\in \g$ a dual basis with respect to $\langle\cdot,\cdot\rangle$. 
The non-trivial singular vector of smallest degree in $\Vcrit$ is 
\be S := \half I^a[-1] I_a[-1] \vac \in \Vcrit,\label{Sdef}\ee 
corresponding to the quadratic Casimir element $\mc C:=\half I^{a}I_a \in Z(U(\g))$. (For brevity, we shall always employ summation convention for the index $a$.)  
Applying the above reasoning to $S$ yields 
\begin{multline} \label{quad swap}
\big[ x \otimes \half I^a[-1] I_a[-1] \vac \big] = \left[ \left( \sum_{i=1}^N \sum_{k=0}^{T-1} \frac{(\sigma^k I^a)^{(i)} }{u - \omega^{-k}z_i} \cdot x \right) \otimes \ha I_a[-1] \vac \right]\\
+ \sum_{p=1}^{T-1} \frac{\omega^p}{(\omega^p - 1) u} \big[ x \otimes \ha (\sigma^p I^a)[0] I_a[-1] \vac \big]
+ \sum_{p=1}^{T-1} \frac{\omega^p}{(\omega^p - 1)^2 u^2} \big[ x \otimes \ha (\sigma^p I^a)[1] I_a[-1] \vac \big].
\end{multline}
The first term on the right is ready for applying the next ``swapping'' and produces
\begin{equation*}
\sum_{k = 0}^{T-1} \sum_{i = 1}^N \frac{\mc C^{(i)} }{(u - \omega^{-k} z_i)^2} + \sum_{k = 0}^{T-1} \sum_{i = 1}^N \frac{\omega^k \sigma^k \widetilde{\mathcal{H}}_i}{u - \omega^{-k} z_i},
\end{equation*}
where $\mc C^{(i)} = \half I^{a (i)} I_a^{(i)}$ (which is invariant under the automorphism $\sigma$) and $\widetilde{\mathcal{H}}_i$ denotes the ``naive'' cyclotomic quadratic Hamiltonians given explicitly by
\begin{equation*}
\widetilde{\mathcal{H}}_i := \sum_{p = 0}^{T-1} \sum_{\substack{j=1\\j \neq i}}^N \frac{I^{a (i)} \sigma^p I_a^{(j)}}{z_i - \omega^{-p} z_j} 
+ \half\sum_{p = 1}^{T-1} \frac{[I^{a (i)}, \sigma^p I_a^{(i)}]_+}{(1 - \omega^{-p}) z_i}.
\end{equation*}
Here $[a, b]_+ := ab + ba$ denotes the anti-commutator of two elements $a, b \in U(\g)$.
Further ``swapping'' also needs to be applied to the last two terms on the right-hand side of \eqref{quad swap}. This will in principle generate correction terms to the above result. 
The second term on the right-hand side of \eqref{quad swap} can be rewritten as
\begin{equation*}
\half\sum_{p=1}^{T-1} \frac{1}{(1 - \omega^{-p}) u} \big[ x \otimes \big[ \sigma^p I^a, I_a \big][-1] \vac \big],
\end{equation*}
which, after applying the swapping procedure as in \eqref{twistedswap}, gives
\begin{align*}
&\half\sum_{k=0}^{T-1} \sum_{i=1}^N \sum_{p=1}^{T-1} \frac{1}{(1 - \omega^{-p}) u (u - \omega^{-k} z_i)} \big[ \sigma^k \big[ \sigma^p I^{a (i)}, I^{(i)}_a \big] x \otimes \vac \big]\\
&= \half\sum_{k=0}^{T-1} \sum_{i=1}^N \sum_{p=1}^{T-1} \frac{1}{(1 - \omega^{-p}) z_i} \left( \frac{1}{u - \omega^{-k} z_i} - \frac{1}{u} \right) \big[ \omega^k \sigma^k \big[ \sigma^p I^{a (i)}, I^{(i)}_a \big] x \otimes \vac \big].
\end{align*}
Here, the simple pole at $u = \omega^{-k} z_i$ leads to a correction in the Hamiltonians, namely
\begin{align}
\mathcal{H}_i &:= \widetilde{\mathcal{H}}_i - \half\sum_{p = 1}^{T-1} \frac{[I^{a (i)}, \sigma^p I_a^{(i)}]}{(1 - \omega^{-p}) z_i}\nn\\ 
&= \sum_{p = 0}^{T-1} \sum_{\substack{j=1\\j \neq i}}^N \frac{I^{a (i)} \sigma^p I_a^{(j)}}{z_i - \omega^{-p} z_j} + \sum_{p = 1}^{T-1} \frac{\sigma^p I^{a (i)} I_a^{(i)}}{(1 - \omega^{-p}) z_i},
\label{tildeHdef}
\end{align}
while the apparent simple pole at $u=0$ is actually vanishing, as follows. For any $k\in \Z_T$, $\sigma^k I_a$ is a basis of $\g$, with dual basis $\sigma^k I^a$. And $[\sigma^p I^a, I_a]$ is an element of $\g$ dependent on $p$ but not on the choice of basis $I_a$. Thus for all $k\in \Z_T$,
$\sigma^k [ \sigma^p I^a, I_a] 
= [\sigma^p (\sigma^k I^a), (\sigma^k I_a)]
= [\sigma^p I^a, I_a]$. Hence in the $1/u$ term we have a factor $\sum_{k\in \Z_T} \omega^k = 0$.

The final term in the right-hand side of \eqref{quad swap} may be written as
\begin{equation*}
\sum_{p=1}^{T-1} \frac{\omega^p}{(\omega^p - 1)^2 u^2} \big[ x \otimes \ha \langle \sigma^p I^a, I_a\rangle k \, \vac \big] = \half \sum_{p=1}^{T-1} \frac{\omega^p\langle \sigma^p I^a, I_a\rangle k}{(\omega^p-1)^2 u^2} \big[ x \otimes \vac \big],
\end{equation*}
which has a double pole at $u = 0$. Putting all of the above together, the element $S(u)\in U(\g)^{\otimes N}$ corresponding to the vector $S = \ha I^a[-1] I_a[-1] \vac$ is
\begin{align}
S(u) &= \sum_{i = 1}^N \sum_{p = 0}^{T-1} \frac{\mc C^{(i)}}{(u - \omega^{-p} z_i)^2} 
 + \sum_{i = 1}^N \sum_{p = 0}^{T-1} \frac{\omega^p \sigma^p \mathcal{H}_i}{u - \omega^{-p} z_i} 
%\nn\\&{}
%+ \frac{1}{2u}\sum_{i = 1}^N \sum_{k = 0}^{T-1} \omega^k \sigma^k \sum_{p = 1}^{T-1} \frac{[\sigma^p I^{a}, I_a^{}]^{(i)}}{(1 - \omega^{-p}) z_i} 
+ \frac{1}{2u^2} \sum_{p=1}^{T-1} \frac{\omega^p\langle \sigma^p I^a, I_a\rangle k}{(\omega^p-1)^2} ,
\end{align}
which has second order poles at $u = \omega^{-k} z_i$, $i = 1, \ldots, N$, $k = 1, \ldots, T - 1$ as well as at $u = 0$.
 
Proposition \ref{prop:Zcom} states in particular that $[S(u),S(v)] = 0$. Therefore the (complex-analytic) residue of $[S(u), S(v)]$ at $u=z_i$, $v=z_j$ vanishes. That is, the Hamiltonians $\mathcal{H}_i \in U(\g)^{\otimes N}$, $i=1,\dots, N$, are mutually commuting.

\section{Automorphism equivariance in the Wakimoto construction}\label{sec:wc}
In the preceding section we constructed the cyclotomic Gaudin algebra $\mathscr Z_{\bm z}^\Gamma(\g,\sigma)$, a commutative subalgebra of $U(\g)^{\otimes N}$. In the remainder of the paper we specialize to spin chains whose Hilbert spaces are tensor products of Verma modules $M_\lambda$ over $\g$. Let us first recall the definition of $M_\lambda$ and its contragredient dual $M_\lambda^*$.

\subsection{Verma and contragredient Verma modules}\label{sec:verm}
Introduce a Cartan decomposition of $\g$,
\be \g = \n_- \oplus \h \oplus \n .\label{cartandecomp}\ee
Let $\Delta^+\subset\h^*$ be the set positive roots of $\g$, and $\{\alpha_i\}_{i\in I}\subset \Delta^+$ a set of simple roots, where $i$ runs over the nodes $I :=\{1,2,\dots,\rank(\g)\}$ of the Dynkin diagram. For each $\alpha\in \Delta^+$, let $E_\alpha\in \n$ (resp. $F_\alpha\in\n_-$) be a root vector of weight $\alpha$ (resp. $-\alpha$) 
and $H_\alpha\equiv \alpha^\vee:=[E_{\alpha},F_{\alpha}]$ the coroot, with normalizations chosen such that $\alpha(\alpha^\vee) = 2$. As usual we identify $\h$ with $\h^*$ by means of the inner product $\langle\cdot,\cdot\rangle$ and then $\alpha^\vee = 2 \alpha\big/\langle \alpha,\alpha \rangle$. 
By abuse of notation we write $H_i\equiv H_{\alpha_i}$ for the simple coroots. Then 
\be \{E_\alpha,F_\alpha\}_{\alpha\in \Delta^+}\cup \{H_i\}_{i\in I}\label{CartanWeyl}\ee 
is a convenient choice of Cartan-Weyl basis of $\g$.

The \emph{Verma module} $M_\lambda$, $\lambda\in \h^*$, over $\g$ is by definition the induced $\g$-module
\be M_{\lambda} = \Ind_{\h\oplus \n}^\g \C v_\lambda :=U(\g) \otimes_{U(\h\oplus \n)} \C v_{\lambda} \ee
where $\C v_{\lambda}$ is the one-dimensional module over $\h\oplus \n$ generated by a vector $v_{\lambda}$ with $\n \on v_{\lambda}=0$ and $h\on v_{\lambda} =  \lambda(h)v_\lambda$ for all $h\in \h$.
Recall that every Verma module is a weight module whose weight spaces are finite-dimensional, i.e. $M_\lambda = \bigoplus_{\mu\in \h^*} (M_\lambda)_{\mu}$ with  $(M_{\lambda})_{\mu} := \{v\in M_{\lambda}: h.v=\mu(h) v\text{ for all } h\in \h \}$, $\dim(M_\lambda)_\mu <\8$. Given any weight module $\mc M$ whose weight spaces are finite-dimensional, its \emph{restricted dual} is the vector space $\mc M^\resd := \bigoplus_{\mu\in \h^*} (\mc M_{\mu})^*$. There is an anti-involution $\iota:\g\to\g$, called the \emph{Cartan anti-involution} of $\g$, given by
\begin{subequations}\label{iotadef}
\be \iota(E_\alpha)=F_\alpha,\quad \iota(F_\alpha) = E_\alpha,
\quad \alpha\in \Delta^+,
\qquad \iota(H_i) = H_i, \quad i\in I.\ee
By means of $\iota$, $\mc M^\resd$ becomes a left $\g$-module called the \emph{contragredient dual} of $\mc M$: 
\be (X\on \kappa)(v) := \kappa(\iota(X)\on v),\ee\end{subequations} for all $\kappa\in \mc M^\resd$, $v\in \mc M$ and $X\in \g$. The module $\mc M^\resd$ is also a weight module whose weight spaces are finite-dimensional (indeed, $\dim \mc M^\resd_\mu = \dim \mc M_\mu$) and there is an isomorphism  of $\g$-modules $(\mc M^\resd)^\resd \cong_\g \mc M$.
The contragredient dual of a Verma module $M_{\lambda}$ is called a \emph{contragredient Verma module} and is, by convention, denoted $M_{\lambda}^*$.

(The Verma module $M_0$ over $\g$ should of course not be confused with the Verma module $\VV^k_0$ over $\gh$ introduced previously.)

\subsection{Contragredient action of Hamiltonians}
We pick weights $\lambda_1,\dots,\lambda_N\in \h^*$. We shall consider Gaudin models whose  Hilbert space is the tensor product of Verma modules
\be M_{(\bm\lambda)} := M_{\lambda_1} \otimes \dots \otimes M_{\lambda_N}.\label{Hilbert}\ee
Commuting Hamiltonians on $M_{(\bm\lambda)}$ can be constructed directly by taking $\mc M_{i} = M_{\lambda_i}$ in \S\ref{ZTGsec}. However, in order to construct simultaneous eigenvectors of the Hamiltonians we in fact begin by taking $\mc M_i=M_{\lambda_i}^*$, so that rather than $M_{(\bm\lambda)}$ we have the tensor product of contragredient Verma modules
\be M^*_{(\bm\lambda)} := M^*_{\lambda_1} \otimes \dots \otimes M^*_{\lambda_N}.\nn\ee
The restricted dual of $M^*_{(\bm\lambda)}$ is $M_{(\bm\lambda)}$, and this is how $M_{(\bm\lambda)}$ is to be thought of in what follows. In particular, the eigenvectors  we construct will be linear maps
\be \eta : M^*_{(\bm\lambda)} \to \C \nn\ee
lying in this restricted dual. In this perspective, the action of $X(u)$, $X\in \Vcrit$, on a general element $\eta\in M_{(\bm\lambda)}$ is via the contragredient dual: i.e. $X(u)\on \eta$ is given by 
\be (X(u) . \eta)(v) := \eta( \iota(X(u)) . v) \nn\ee
for all $v\in M^*_{(\bm\lambda)}$.

To indicate the reason for this indirect approach, let us sketch in outline the construction below. What distinguishes the contragredient Verma module $M_{\lambda}^*$ is that it arises as the grade-zero component of a \emph{Wakimoto module} $W_{\lambda(t)}$, where $\lambda(t)= \lambda \otimes t^{-1}+\mc O(t^0)$. Wakimoto modules are defined initially as $\Z$-graded modules over the Lie algebra $\Heis(\g)\oplus \h\otimes\C((t))$, where $\Heis(\g)$ is the \emph{Heisenberg Lie algebra} associated to $\g$. By means of the \emph{Wakimoto construction}/\emph{free field realization}, Wakimoto modules become also $\Z$-graded modules over $\gh$. The merit of the Lie algebra $\Heis(\g)\oplus \h\otimes\C((t))$ in comparison with $\gh$ is that it is commutative modulo central elements. This simpler structure allows one to construct eigenvectors of the Gaudin Hamiltonians and to find explicit expressions for their eigenvalues.

We now recall the necessary details of the Wakimoto construction. The starting point is the finite-dimensional setting.

\subsection{Realization of $\g$ by differential operators on $N_+$}\label{fds} 
Consider the unital associative algebra generated by $\{x_\alpha,\del_\alpha\}_{\alpha\in \Delta^+}$ subject to the relations 
\be [x_\alpha,x_\beta] = 0,\quad 
[\del_\alpha,x_\beta] = \delta_{\alpha\beta} 1,\quad 
[\del_\alpha,\del_\beta] = 0,
% \quad [\mathbf 1,x_\alpha]=[\mathbf 1,\del_\alpha] = 0
, \qquad \alpha,\beta\in \Delta^+.\nn\ee
Call this $\Weyl(\g)$, the Weyl algebra of $\g$.
There is an injective homomorphism of algebras $\bar\rho :U(\g)\to\Weyl(\g)$ given by
\be
\bar\rho(E_{\alpha}) = \sum_{\beta\in \Delta^+} P_{\alpha}^\beta(x) \del_\beta, \label{barrhodef}\qquad
\bar\rho(H) = -  \sum_{\beta\in \Delta^+} \beta(H) x_\beta \del_\beta,  \qquad
\bar\rho(F_{\alpha}) = \sum_{\beta\in \Delta^+} Q_{\alpha}^\beta(x) \del_\beta,
\ee
for certain polynomials $P_\alpha^\beta(x), Q_\alpha^\beta(x)\in \C[x_\gamma]_{\gamma\in \Delta^+}$ such that 
\be\deg P_\alpha^\beta = \beta -\alpha\quad\text{and}\quad\deg Q_\alpha^\beta = \beta+\alpha,\label{grPQ}\ee 
with respect to the gradation of $\C[x_\gamma]_{\gamma\in \Delta^+}$ by $\g$-weights in which $\deg x_\alpha = \alpha$.
This homomorphism $\bar\rho$ may be understood in two (equivalent) ways. 

For the geometrical perspective, let $G$ denote the connected, simply-connected Lie group associated to $\g$, $N_+\subset G$ the unipotent subgroup with Lie algebra $\n$, and $B_-$ the Borel subgroup with Lie algebra $\mf b_-:= \n_- \oplus \h$. The \emph{flag manifold} associated to these data is the homogeneous space $B_-\big\backslash G$. There is an 
open, dense subset  $U\subset B_-\big\backslash G$ called the \emph{big cell}, defined by $U:= \{ [B_-]n\in B_-\big\backslash G: n\in N_+ \}$.
The big cell $U$ is isomorphic to $N_+$. 
The right action of the group $G$ on $B_-\big\backslash G$ gives rise to a left action of $G$ on functions on $B_-\big\backslash G$ and hence, infinitesimally, an embedding of $\g$ into the Lie algebra $\Vect(B_-\big\backslash G)$ of vector fields on $B_-\big\backslash G$.
\footnote{To $X\in \g$ is associated the vector field $\xi_X$ which sends the function $f:[B_-]g\mapsto f([B_-]g)$ to the function $(\xi_Xf): [B_-]g \mapsto \left.\frac{\del}{\del \eps} f([B_- ]g e^{\eps X})\right|_{\eps=0}$.  Then given $X,Y\in \g$, $(\xi_X(\xi_Yf))(g) = \left.\frac{\del}{\del \eps} (\xi_Yf)([B_- ]g e^{\eps X})\right|_{\eps=0} = \left.\frac{\del}{\del \eps}  \frac{\del}{\del \eta}f([B_- ]g e^{\eps X}e^{\eta Y})\right|_{\eps=0,\eta =0}$. From this one verifies that $\xi_X(\xi_Yf)-\xi_Y(\xi_Xf) = \xi_{[X,Y]}f$.} 
Since the big cell is open, this in turn gives rise to embeddings $\g \hookrightarrow \Vect(U)$ and, hence, $\g\hookrightarrow \Vect(N_+)$. If we let $(x_\alpha)_{\alpha\in \Delta^+}$ be a set of homogeneous coordinates on $N_+$, and think of $\Vect(N_+)\subset \Weyl(\g)$, then this embedding is $\bar\rho$.

For a more representation-theoretic viewpoint, one starts with the contragredient Verma module $M_0^*$ over $\g$. In \S\ref{sec:verm} this was defined as the contragredient dual of the Verma module $M_0$, but it may also be regarded as a \emph{coinduced module}:
\be M_0^* 
\cong_\g \Coind_{U(\mf b_-)}^{U(\g)} \C_0 
:= \Hom^\textup{res}_{U(\mf b_-)}(U(\g), \C_0). \label{M0stardef}\ee
Here $U(\g)$ is considered as a left $U(\mf b_-)$-module, and $\C_0$ is the left $U(\mf b_-)$-module on which $\mf b_-$ acts as zero. So $\Hom_{U(\mf b_-)}(U(\g), \C_0)$ is the space of maps $\eta\in  U(\g)^*$ such that $\eta( b x) = b\eta(x)= 0$
for all $b\in \mf b_-$, $x\in U(\g)$. By the PBW theorem, $U(\g) \cong_{\C} U(\mf b_-) \otimes U(\n)$ as vector spaces, and in this way we identify $U(\g)^* \cong_\C U(\n)^* \otimes U(\mf b_-)^* $. The ``$\textup{res}$'' indicates we consider only those elements of $U(\g)^*$ that belong to $U(\n)^\resd\otimes U(\mf b_-)^*$.
%, where $U(\n)^\vee$ is the restricted dual of $U(\n)$, i.e. $U(\n)^\vee := \bigoplus_{\lambda\in \Z[\alpha_i]\subset \h^*} U(\n)_\lambda^*$, with $U(\n)_\lambda := \{ x\in U(\n): [h,x]= \lambda(h) x \}$. 
Note that if $x_-\otimes x_+\in U(\mf b_-) \otimes U(\n)$ then $\eta(x_-x_+) = \eps(x_-) \eta(x_+)$, where $\eps : U(\mf b_-) \to \C$ is the counit. In this way $M_0^*\cong_\C U(\n)^\vee$. 
The left $U(\g)$-module structure on $M_0^*$ is the \emph{coinduced} one, which is to say that
\be (g\on \eta) (x)  = \eta( xg) \label{coindact}\ee
for all $g,x\in U(\g)$ and $\eta\in M_0^*$.\footnote{To see the isomorphism \eqref{M0stardef}, note that an element $\eta\in M_0^*$ is in particular a map $\eta:U(\g)\to \C$ such that $\eta(gX) = 0$ for all $X\in \mf b_+ := \h \oplus \n$. There is a bijection $\eta\mapsto\eta\circ\iota$ from $M_0^*$ to $\Hom^\textup{res}_{U(\mf b_-)}(U(\g), \C_0)$ where the action of $U(\mf b_-)$ on $U(\g)$ is from the left,
for  if $\eta(gX)=0$ for all $X\in \mf b_+$ then $(\eta\circ\iota)(Xg) = \eta(\iota(g)\iota(X))=0$ for all $X\in \mf b_-$. Next, the contragredient action of a $Y\in \g$ on $\eta\in M_0^*$ is $(Y\on \eta)(g) = \eta(\iota(Y)g)$. This defines an action on elements $(\eta\circ\iota)\in \Hom^\textup{res}_{U(\mf b_-)}(U(\g), \C_0)$ given by  $(Y\on(\eta\circ\iota))(g):= (Y\on\eta)(\iota(g)) = \eta(\iota(Y)\iota(g)) = \eta(\iota(gY)) = (\eta\circ\iota)(gY)$, i.e. multiplication of the argument on the right by $Y$, which agrees with the coinduced module structure of \eqref{coindact}.} Now, $M_0^*$ also has the structure of a commutative algebra, the product $\cdot$ coming from the co-commutative co-algebra structure on $U(\g)$,
\be (\eta \cdot \eta')(x) := \cdot(\eta \otimes \eta')(\Delta x)\ee 
(where on the right-hand side, $\cdot$ is multiplication in $\C$). Since elements $X\in\g\subset U(\g)$ are algebra-like, i.e. $\Delta X= X\otimes 1 + 1\otimes X$, we have that the action of the Lie algebra $\g$ on $M_0^*$ is by derivations with respect to this product, i.e. $X\on (\eta\cdot\eta') = (X\on \eta)\cdot\eta' + \eta\cdot(X\on \eta')$. So we have a homomorphism of Lie algebras $\g\to \Der(M_0^*)$. If we now take $(x_\alpha)_{\alpha\in \Delta^+}$ to be a set of homogeneous (i.e. $H.x_\alpha = \alpha(H)x_\alpha$) generators of $M_0^*$ then we have an identification $M_0^* \cong \C[x_\alpha]_{\alpha\in \Delta^+}$ as commutative algebras;
and $\bar\rho$ is the map specifying the action of $\g$ on the latter by derivations.

(The two perspectives above are identified by noting that $M_0^* \cong \C[N_+]$, the algebra of regular  functions on $N_+$, and $\Der(M_0^*)\cong \Vect(N_+)$ is the Lie algebra of derivations on this algebra.) 

Further details on realizations of Lie algebras by differential operators can be found in e.g. \cite{Draisma} and references therein, in particular \cite{Blattner}.

\subsection{Equivariance, and equivariant coordinates on $N_+$.}
It is always possible to pick the Cartan decomposition \eqref{cartandecomp} in a manner compatible with the automorphism $\sigma:\g \to \g$, i.e. we can assume that  
\be\sigma(\n) = \n,\quad \sigma(\h) = \h,\quad \sigma(\n_-) = \n_-.\label{sigmacartan}\ee
Then in fact \cite[\S8.6]{KacBook}
\be \sigma(E_{\alpha}) = \tau_\alpha E_{\sigma(\alpha)},\qquad \sigma(H_i) = H_{\sigma(i)},\qquad \sigma(F_{\alpha}) = \tau_{\alpha}^{-1} F_{\sigma(\alpha)},\label{sigmaE}\ee
 where, by a slight overloading of notation, $\sigma:\Delta^+\to\Delta^+$ is a symmetry of the root system, coming in turn from a symmetry $\sigma:I\to I$ of the Dynkin diagram, and where the $\tau_\alpha$, $\alpha\in \Delta^+$, are certain elements of $\Gamma$.

Given any complex vector space $A$ equipped with an action of $\Gamma$, we introduce an induced action of $\Gamma$ on maps $\eta : A \to \C$ as\footnote{Note that the action of $\Gamma$ on $\alpha \in \Delta^+$ defined through \eqref{sigmaE} agrees with the action \eqref{lsigmadef} of $\Gamma$ on $\alpha$ viewed as an element of $\h^{\ast}$. Indeed, applying $\sigma$ to the relation $[\sigma^{-1}(H), E_{\alpha}] = \alpha(\sigma^{-1}(H)) E_{\alpha}$ we find $[H, E_{\sigma(\alpha)}] = (L_{\sigma} \alpha)(H) E_{\sigma(\alpha)}$, from which it follows that $L_{\sigma} \alpha = \sigma(\alpha)$.}
\be \lsigma\eta := \eta \circ \sigma^{-1}.\label{lsigmadef}\ee 
This defines in particular a map $\lsigma:M_0^*\to M_0^*$,
\footnote{Indeed, if $\eta\in M^*_0$ and $b\in U(\mf b_-)$ then $\eta(bx) = 0$ so $(\lsigma\eta)(bx) = \eta(\sigma^{-1}(bx))= \eta(\sigma^{-1}(b)\sigma^{-1}(x)) = 0$ which says that  $\lsigma\eta\in M^*_{0}$. But note that for the general contragredient Verma module of $\g$, the map is $M_{\lambda}^*\to M^*_{\lsigma(\lambda)}$.}  
and the homomorphism $\bar\rho$ of \eqref{barrhodef} is equivariant with respect to $\sigma$ in the sense that 
\be \bar\rho(\sigma(X))\circ \lsigma = \lsigma \circ \bar\rho(X);  \label{sigeq}\ee
for indeed, the following diagram commutes
\be\begin{tikzpicture}    
\matrix (m) [matrix of math nodes, row sep=3em,    
column sep=4em, text height=2ex, text depth=1ex]    
{     
\g \otimes M^*_0 & M^*_0  \\    
\g \otimes M^*_0 & M^*_0  \\    
};    
\path[->,font=\scriptsize,shorten <= 2mm,shorten >= 2mm]    
(m-1-1) edge node [above] {} (m-1-2)    
(m-2-1) edge node [above] {} (m-2-2);    
\path[->,shorten <= 2mm,shorten >= 2mm]    
(m-1-1) edge node [left] {$\sigma\otimes \lsigma$} (m-2-1)    
(m-1-2) edge node [right] {$\lsigma$} (m-2-2);    
\end{tikzpicture}\nn\ee    
because, on the one hand $(X,\eta) \mapsto (X\on \eta) \mapsto \lsigma(X\on \eta)$ where $(\lsigma(X\on \eta))(x) = (X\on \eta)(\sigma^{-1}x) = \eta((\sigma^{-1} x)X)$, while on the other hand $(X,\eta) \mapsto (\sigma X,\lsigma\eta) \mapsto (\sigma X)\on (\lsigma \eta)$ where $((\sigma X)\on(\lsigma \eta))(x) = (\lsigma \eta)(x \sigma X) = \eta(\sigma^{-1}(x\sigma X)) = \eta((\sigma^{-1}x) X)$.

\begin{exmp}\label{sl3ex}
Suppose $\g=\mf{sl}_3$. We use the notation $E_{12} = E_{\alpha_1}$, $E_{13}=E_{\alpha_1+\alpha_2}$, $E_{23} = E_{\alpha_2}$ and similarly for $F_{12}$, etc. In the defining representation, elements of $N_+$ are unipotent upper-triangular $3\times 3$ matrices. For each value of a parameter $\gamma$ there is a system of homogeneous coordinate functions $\{x_{12},x_{13},x_{23}\}$ on $N_+$ given in terms of the matrix elements in the defining representation as follows:
\begin{equation} \label{sl3 point}
 \bmx 1 & x_{12} & x_{13} + \gamma x_{12} x_{23}\\ 0 & 1 & x_{23}\\ 0 & 0 & 1 \emx.
\end{equation}
In particular, the values $\gamma=0$, $\gamma=\half$ and $\gamma=1$ correspond respectively to the parameterizations
\be e^{x_{23} E_{23}} e^{x_{13} E_{13}} e^{x_{12} E_{12}}, \quad
    e^{x_{12} E_{12} + x_{13} E_{13} + x_{23} E_{23}},\quad
    e^{x_{12} E_{12}} e^{x_{13} E_{13}} e^{x_{23} E_{23}} \ee
of $N_+$.
One finds by direct computation that the explicit form of the homomorphism $\bar\rho$ of \eqref{barrhodef} is
\begin{align*} \bar\rho(-E_{12}) &= \del_{12} + (1-\gamma) x_{23} \del_{13} \\
              \bar\rho(-E_{13}) &= \del_{13} \\
              \bar\rho(-E_{23}) &= \del_{23} - \gamma x_{12} \del_{13} \\
\bar\rho(-F_{12}) & = - x_{12}^2 \del_{12} + (x_{13}  + \gamma  x_{12} x_{23} )  \del_{23} 
              + (-\gamma x_{12}   x_{13} - \gamma(\gamma-1)     x_{12}^2  x_{23} )  \del_{13}\\
\bar\rho(-F_{13}) & = (-x_{12} x_{13} + (1 - \gamma) x_{12}^2 x_{23})  \del_{12} 
              + (-x_{13}^2 + (\gamma -1 ) \gamma  x_{12}^2  x_{23}^2)  \del_{13} 
              + (-x_{13}  x_{23} - \gamma x_{12}  x_{23}^2)  \del_{23}\\
\bar\rho(-F_{23}) & = ( -x_{13} + (1 - \gamma)  x_{12}  x_{23})  \del_{12} 
    + ((\gamma-1)  x_{13}  x_{23} + (\gamma-1)  \gamma  x_{12}  x_{23}^2 )  \del_{13} 
     - x_{23}^2  \del_{23}. \end{align*}
Now let $\sigma$ be the (involutive) diagram automorphism, i.e. $\sigma E_{12} = E_{23}$, $\sigma E_{23} = E_{12}$, $\sigma E_{13} = -E_{13}$. If we write, for brevity,
\be\nn\tilde x_{12}:= \lsigma x_{12}\quad\text{and}\quad \tilde\del_{12} := \left.\frac{\del}{\del \tilde x_{12}}\right|_{\tilde x_{13},\tilde x_{23}}\equiv \lsigma\circ\del_{12}\circ\lsigma^{-1}\quad\text{etc.},\ee
then
\begin{subequations}\label{xsigbad}
\be \tilde x_{12} = x_{23}, \quad \tilde x_{23} = x_{12}, \quad \tilde x_{13} = -x_{13} + (1-2\gamma) x_{12} x_{23} \ee
and hence, by the chain rule,
\be \tilde \del_{12} = \del_{23} + (1-2\gamma) x_{12} \del_{13},\quad
    \tilde \del_{23} = \del_{12} + (1-2\gamma) x_{23} \del_{13},\quad
    \tilde \del_{13} = -\del_{13}. \ee 
\end{subequations}
One can then check the equivariance of $\bar\rho$ explicitly: for example
\begin{align} \bar\rho(-\sigma E_{12}) &=  \bar\rho(-E_{23}) = \del_{23} -\gamma x_{12}\del_{13} \nn\\ 
&=\tilde \del_{12} + (1-\gamma) \tilde x_{23} \tilde\del_{13}= \lsigma \circ \bar\rho(-E_{12}) \circ \lsigma^{-1} \nn\end{align}
as required.
\end{exmp}
While the equivariance property \eqref{sigeq} of $\bar\rho$ holds independently of any choice of homogeneous coordinates on $N_+$, it is very natural to make the choice $x_\alpha:N_+\to \C; n\mapsto x_\alpha(n)$ given by the parameterization
\be n= \exp\left(\sum_{\alpha\in\Delta^+} x_\alpha(n) E_\alpha\right).\label{gc}\ee 
(In the example above this is the choice $\gamma=\half$).
Observe that these coordinates $(x_\alpha)_{\alpha\in\Delta^+}$ have the property that they themselves, and hence also their Weyl conjugates, behave equivariantly under $\sigma$. That is,
\begin{subequations}\label{xsigdef}
\be \lsigma(x_{\alpha}) = \tau_{\alpha}^{-1} x_{\sigma(\alpha)},\ee
in view of \eqref{sigmaE}, and hence 
\be \left.\frac{\del}{\del(\lsigma x_{\alpha})}\right|_{\lsigma x_\beta: \beta\neq \alpha}
 = \tau_{\alpha} \del_{\sigma(\alpha)}.\ee
\end{subequations}
(Homogeneous coordinates are not equivariant in this sense in general, as \eqref{xsigbad} shows.)

From the representation-theoretic perspective, these equivariant generators $x_{\alpha}$ are constructed as follows. First, suppose $\n^\perp$ is some choice of complementary subspace to $\n$ in the vector space $U(\n)$. Let $\pi$ be the corresponding projection $U(\n)\twoheadrightarrow \n$. Then $(E_\alpha^*\circ \pi)_{\alpha\in\Delta^+}$ are a set of homogeneous generators of $M_0^*$, where $(E_\alpha^*)_{\alpha\in\Delta^+}$ denotes the dual basis to the basis $(E_\alpha)_{\alpha\in\Delta^+}$ of $\n$. These generators are equivariant if and only if $\n^\perp$ is stable under the action of $\sigma$, i.e. if and only if $\sigma\circ \pi = \pi \circ \sigma$. Such a stable complement is defined by the usual vector space isomorphism $U(\g)\cong_\C S(\g)$.

\subsection{Wakimoto construction}\label{wc}
The \emph{Heisenberg Lie algebra} $\Heis(\g)$ is by definition the Lie algebra with generators $a_{\alpha}[n]$, $a^*_{\alpha}[n]$, $\alpha\in \Delta^+$, $n\in \Z$, and central generator $\mathbf 1$, obeying the relations
\be [a_\alpha[n],a_\beta[m]] = 0,\quad 
[a_\alpha[n],a^*_\beta[m]] = \delta_{\alpha\beta} \delta_{n,-m} \mathbf 1,\quad 
[a^*_\alpha[n],a^*_\beta[m]] = 0,\qquad \alpha,\beta\in\Delta^+, n,m\in \Z.\label{HLdef} \ee
%and \be [\mathbf 1,a_\alpha[n]]=[\mathbf 1,a^*_\alpha[n]] = 0 , \qquad \alpha\in\Delta^+, n\in \Z.\ee
%Let $U_1(\Heis(g))$ denote the quotient of the univeral enveloping algebra $U(\Heis(g))$ by the two-sided ideal generated by $\mathbf 1 - 1$. This quotient is the \emph{Heisenberg algebra} of $\g$. 
Define $\Mh$ to be the induced representation of $\Heis(\g)$ generated by a vector $\wac$ obeying the conditions $\mathbf 1 \wac = \wac$,
\be a_\alpha[m]\wac = 0,\, m\in \Z_{\geq 0},\qquad a^*_\alpha[m]\wac=0,\,m\in\Z_{\geq 1}. \label{Mdef}\ee
for all $\alpha\in \Delta^+$. 

Let $b_i[n]:=H_{\alpha_i}\otimes t^n$, $i\in I$, $n\in \Z$, be a basis of a copy of the commutative Lie algebra $\h\otimes\C((t))$, and let $\pi_0\simeq  \C[ b_i[n]]_{i\in I;\,n\leq -1}$ be the induced representation of $\h\otimes\C((t))$ in which $b_i[n]$ acts as $0$ for all $i\in I$ and all $n\in \Z_{\geq 0}$.

Now define 
\be\WW_0:= \Mh\otimes \pi_0,\label{WW0def}\ee 
which is an induced representation of $\Heis(\g)\oplus\h\otimes\C((t))$. There is a $\Z$-grading on $\WW_0$ defined by $\deg\wac=0$ and $\deg a_\alpha[n] = \deg a^*_\alpha[n] = \deg b_i[n]= n$.

Recall, for example from \cite{KacVA,FB}, that a \emph{vertex algebra} is a vector space $V$ over $\C$ with a distinguished vector $|0\rangle \in V$ called the \emph{vacuum} and equipped with a linear map, referred to as the \emph{state-field correspondence} or \emph{vertex operator map},
\begin{equation} \label{Y map}
\begin{split}
Y(\cdot, x) : V &\to \text{Hom}(V, V((x)) ) \subset \text{End}\, V [[x, x^{-1}]],\\
A &\mapsto Y(A, x) = \sum_{n \in \mathbb{Z}} A_{(n)} x^{-n-1}, \qquad A_{(n)} \in \text{End}\, V,
\end{split}
\end{equation}
obeying certain axioms.
%Elements of $V$ are called \emph{states} and we say that $Y(A, x)$ is the \emph{vertex operator} or \emph{field} associated to a given state $A \in V$. More generally, an $(\text{End}\, V)$-valued formal distribution $F(x) \in \text{End}\, V [[x, x^{-1}]]$ with the property that $F(x) A \in V((x))$ for all $A \in V$, \emph{i.e.} $F(x) \in \text{Hom}(V, V((x)) )$, is called a \emph{field}.
In the present paper we choose to avoid a detailed discussion of vertex algebras, reserving these aspects for a companion paper \cite{VY}. In the remainder of this section we merely summarize the results we require.
(It should be emphasized, however, that
vertex algebras play a central role in the proof given in  \cite{VY} of Theorem \ref{thm:coinv} below, which in turn is crucial in the construction of eigenvectors of the cyclotomic Gaudin Hamiltonians.) 

First, both $\Vcrit$ and $\WW_0$ have natural vertex algebra structures. In $\Vcrit$ the vacuum state is taken to be the highest weight vector $\vac$ and the state-field correspondence is defined for states of the form $A[-1] \vac$, $A \in \g$ as
\begin{equation*}
Y(A[-1] \vac, x) = \sum_{n \in \mathbb{Z}} A[n] x^{-n-1},
\end{equation*}
so that $(A[-1] \vac)_{(n)} = A[n]$. 
In $\WW_0$ the vacuum vector is $\wac$ and one defines 
\begin{subequations}\label{WYdef}
\be Y(a_\alpha[-1]\wac,x) := 
%a_\alpha(x) := 
\sum_{n\in \Z} a_\alpha[n] x^{-n-1},\quad
  Y(a^*_\alpha[0]\wac,x) := 
%a^*_\alpha(x) := 
\sum_{n\in \Z} a^*_\alpha[n] x^{-n}\ee
for all $\alpha \in \Delta^+$, and
\be Y(b_i[-1]\wac,x) := 
%b_i(x) := 
\sum_{n\in \Z} b_i[n] x^{-n-1} \ee 
\end{subequations}
for all $i\in I$.
By means of the \emph{reconstruction theorem} -- see e.g. \cite[\S2.3.11 and \S4.4.1]{FB} -- these assignments, together with the specification of the \emph{translation operator} $T$, suffice to define the vertex algebra map $Y$ on the whole of $\Vcrit$ and on the whole of $\WW_0$. 

Next, there is a notion of a \emph{homomorphism} between vertex algebras, and, in particular, of an \emph{automorphism} of a vertex algebra. The map $A[-1]\vac \xrightarrow\sigma (\sigma A)[-1] \vac$ extends to a unique automorphism of $\Vcrit$ as a vertex algebra. There is a unique automorphism (which we also call $\sigma$) of the vertex algebra $\WW_0$ defined by
\begin{subequations} \label{sigmaWdef}
\be \sigma a^*_\alpha[n] := \tau_{\alpha}^{-1} a^*_{\sigma(\alpha)}[n], \qquad
    \sigma a_\alpha[n]   := \tau_{\alpha} a_{\sigma(\alpha)}[n],\ee
cf. \eqref{xsigdef}, and
\be \sigma b_i[n] := b_{\sigma(i)}[n]. \ee 
\end{subequations}

Given any polynomial $p(x)\in \C[x_\alpha]_{\alpha\in\Delta^+}$, denote by  $p(a^*[0])$ the polynomial in $\C[a^*_\alpha[0]]_{\alpha\in\Delta^+}$ obtained by the replacement $x_\alpha\mapsto a^*_\alpha[0]$. 
Let $P_{\alpha_i}^\beta$ and $Q_{\alpha_i}^\beta$ be the polynomials appearing in \eqref{barrhodef}.
It was shown by B.~Feigin and E.~Frenkel, \cite{FF90}, following \cite{Wakimoto}, that there is an injective homomorphism of vertex algebras $\rho:\Vcrit\to\WW_0$ defined by 
\begin{subequations} \label{rhodef}
\begin{align}
\label{rhodefE} \rho(E_{\alpha_i}[-1]\vac) &= \sum_{\beta\in \Delta^+} P_{\alpha_i}^\beta(a^*[0]) a_\beta[-1]\wac \\
\label{rhodefH} \rho(H_{\alpha_i}[-1]\vac) &=   -\sum_{\beta\in \Delta^+} \beta(H_{\alpha_i}) a^*_\beta[0] a_\beta[-1]\wac + b_i[-1]\wac\\
\label{rhodefF} \rho(F_{\alpha_i}[-1]\vac) &= \sum_{\beta\in \Delta^+} Q_{\alpha_i}^\beta(a^*[0]) a_\beta[-1]\wac + c_i a^*_{\alpha_i}[-1]\wac - a^*_{\alpha_i}[0] b_i[-1]\wac. 
\end{align}
\end{subequations}
for certain constants $c_i$, $i\in I$. This homomorphism $\rho$ has the property that $\mf Z(\gh)$, cf. \S\ref{singsec}, is mapped into $\pi_0$.

Moreover, it was shown by M.~Szczesny in \cite{Szc}  that if the polynomials $P_{\alpha_i}^\beta$ and $Q_{\alpha_i}^\beta$ in \eqref{barrhodef} are defined with respect to a set of generators  $\{x_\alpha, \del_\alpha\}_{\alpha\in \Delta^+}$ obeying \eqref{xsigdef}, then this homomorphism $\rho$ is equivariant, i.e.
\be \rho(\sigma(v)) = \sigma \rho(v)\label{sigmaeq}\ee
for all $v\in \Vcrit$, where the action of $\sigma$ in $\WW_0$ is as given in \eqref{sigmaWdef} .

In Theorem 2 of \cite{Szc} the choice of homogeneous coordinates on $N_+$ is what we call the equivariant one, \eqref{gc}. Equivariance of $\rho$ on $E_{\alpha_i}[-1]\vac$ and $H_{\alpha_i}[-1]\vac$ is then immediate; what has to be checked is that $c_i=c_{\sigma(i)}$, which ensures equivariance for $F_{\alpha_i}[-1]\vac$. In \cite{Szc} $\sigma$ was taken to be a diagram automorphism (i.e. $\tau_{\alpha_i}=1$ for each simple root $\alpha_i$) but by inspection one finds that the proof goes through in general. 

\begin{rem}\label{coordrem}
Changes of homogeneous coordinates on $N_+$ naturally induce automorphisms of $\WW_0$. Thus, once \eqref{sigmaeq} is established for this system of homogeneous coordinates, it follows for all others (but of course the definition of $\sigma$ on $\WW_0$ in the new coordinates will be more involved, since one must take \eqref{sigmaWdef} and conjugate it by the change-of-coordinate automorphism, which need not  be linear).
\end{rem}

A \emph{smooth} module over $\Heis(\g)\oplus \h\otimes\C((t))$ is any module $M$ such that for all $v\in M$ there is an $n\in \Z_{\geq 0}$ such that for all $m\geq n$, $0=a_\alpha[m]v=a^*_{\alpha}[m]v = b_i[m]v$ for all $\alpha\in \Delta^+$ and all $i\in I$. It follows from the existence of the homomorphism $\rho$ that
\be\text{
every smooth $\Heis(\g)\oplus \h\otimes\C((t))$-module has the structure of a $\gh$-module}\label{smp}\ee   
in which $K$ acts as $-h^\vee$.

\begin{rem}\label{smprem} Let us recall in outline the reason for this; for details see \cite{FB}. To every vertex algebra $V$ is associated a Lie algebra $U(V)$, the ``big'' Lie algebra of $V$, spanned by all formal modes of all vertex operators in $V$. There is a homomorphism of Lie algebras $\gh\to U(\Vcrit)$ defined by $K\mapsto (-h^\vee \vac)_{[-1]}$ and $A[n]\mapsto (A[-1]\vac)_{[n]}$ for $A\in \g$, $n\in \Z$. The homomorphism $\rho$ induces a homomorphism $U(\Vcrit)\to U(\WW_0)$ between the big Lie algebras. It follows that every $U(\WW_0)$-module pulls back to a $\gh$-module in which $K$ acts as $-h^\vee$. Finally, every smooth module over $\Heis(\g)\oplus \h\otimes\C((t))$ has a canonical $U(\WW_0)$-module structure.\end{rem}

\subsection{The right action of $\n$ on $M_0^*$, and the generators $G_i$.}\label{Gsec}
The coinduced left $U(\g)$-module $M_0^*$ of \S\ref{fds} is also a right $U(\n)$-module, the action coming from the left action of $U(\n)$ on itself. That is, given $\eta\in M_0^*$ and $n\in U(\n)$, the right action is defined by $(\eta . n)(x_-x_+) := \eps(x_-)\eta(nx_+)$, where $x_-\otimes x_+\in U(\mf b_-) \otimes U(\n)$ and we extend by linearity to the whole of $U(\g) \cong_\C U(\mf b_-) \otimes U(\n)$. (Here $\eps:U(\mf b_-)\to \C$ is the counit.) This action commutes with the left action of $U(\g)$. By the same argument as in \S\ref{fds}, any $X\in \n$ then defines a derivation of $M_0^*$ when the latter is viewed as a commutative algebra. However, because the action of $U(\n)$ on $M_0^*$ is from the right, the map $\n\to\Der(M_0^*)$ so defined is an \emph{anti}-homomorphism of Lie algebras. By introducing an overall sign, it becomes a homomorphism. After making the identification $M_0^*\cong_\C \C[x_\alpha]_{\alpha\in\Delta^+}$,  the generators $E_{\alpha_i}$ of $\n$ are mapped under this homomorphism to differential operators $G_i$ of the form
\be G_i = \sum_{\beta\in \Delta^+} R_{i}^\beta(x) \del_\beta\ee
for certain homogeneous polynomials in the $x_\alpha$, $\alpha\in\Delta^+$, with $\deg R_i^\beta = \beta-\alpha_i$. 
One can check that 
\be\label{PRrel} -1 =  P_{\alpha_i}^{\alpha_i}(x) = -R_i^{\alpha_i}(x) \ee
where the $P_\alpha^\beta$ are those of \eqref{barrhodef}.

Let $\mathbb N$ be the vacuum Verma module of $\n\otimes \C((t))$:
\be \mathbb N := \Ind_{\n\otimes \C[[t]]}^{\n\otimes \C((t))} \C\vac = U(\n\otimes \C((t))) \otimes_{U(\n\otimes \C[[t]])} \C\vac \ee
where $\vac$ is a nonzero vector such that $\n[t]\vac = 0$. 
With the natural vertex algebra structure on $\mathbb N$ (in which $Y(A[-1]\vac,x) = \sum_{n\in \Z} A[n] x^{-n-1}$ for all $A\in \n$), there is an injective homomorphism $\mathbb N\mapsto \Mh\subset \WW_0$ of vertex algebras defined by
\begin{equation}
E_{\alpha_i}[-1]\vac \mapsto G_i[-1] \wac := \sum_{\beta\in \Delta^+} R_{i}^\beta(a^*[0]) a_\beta[-1] \wac.\label{NWhom}
\end{equation}
Provided the polynomials $R_i^\beta$ are defined with respect to a set of generators  $\{x_\alpha, \del_\alpha\}_{\alpha\in \Delta^+}$ obeying \eqref{xsigdef}, this homomorphism is $\sigma$-equivariant.

The existence of this homomorphism means that, in addition to \eqref{smp}, every smooth module over $\Heis(\g)\oplus \h\otimes\C((t))$ is endowed also with the structure of a module over this ``right'' copy, call it $$\nG\otimes\C((t)),$$ of $\n\otimes\C((t))$.    
The elements $G_i[-1]\wac\in \Mh$ play an important role in what follows.

\section{Cyclotomic coinvariants of $\Heis(\g)\oplus \h\otimes\C((t))$-modules}\label{sec:Hdef}
The next step is, roughly speaking, to repeat much of \S\ref{ZTGsec} but with the role of $\gh$ replaced by the Lie algebra $\Heis(\g)\oplus \h\otimes\C((t))$ introduced in \S\ref{wc}. 

To the points $\bm z = \{z_1,\dots,z_N\}$, $z_i\in \Cx$, $1\leq i\leq N$, we add additional non-zero marked points $\bm w:= \{w_1,\dots,w_m\}$, $w_j\in \Cx$, $1\leq j\leq m$.
%We require that $\Gamma \bm z \cap \Gamma \bm w = \emptyset$ and that $\Gamma w_i \cap \Gamma w_j = \emptyset$ for all $i\neq j$. 
The points $z_i$ will continue to correspond to the sites of the Gaudin spin chain, while the points $w_i$ will play the role of the Bethe roots. For convenience, let us write $\bm x = \{x_1,\dots, x_{p}\}$ with $(x_1,\dots,x_{p}) = (z_1,\dots,z_N,w_1,\dots,w_m)$ and $p=N+m$. We require that 
\be \Gamma x_i\cap  \Gamma x_j = \emptyset \quad\text{for all $1\leq i\neq j\leq p$}. \ee
In addition, it will also be necessary to introduce a certain carefully chosen module assigned to the point $0$. This is one of the new features of the cyclotomic construction. 

Let $\n_\C$ (resp. $\n^*_\C$) denote the vector space $\n$ (resp. $\n^*$) endowed with the structure of a commutative Lie algebra.
On the commutative Lie algebra  $\n_\C\oplus\n^*_\C$ there is a non-degenerate bilinear skew-symmetric form $\langle\cdot,\cdot\rangle$ defined by
\be \langle X, Y \rangle = Y|_{\n^*}\left(X|_{\n}\right) - X|_{\n^*}\left(Y|_{\n}\right),  \ee
and an action (trivially by automorphisms) of the group $\Gamma$ given by 
\be \omega. X := \sigma(X|_{\n}) \oplus \lsigma(X|_{\n^*}).\label{Gn}\ee
Given any $x\in \C$, observe that the Heisenberg Lie algebra $\Heis(\g)$ of \S\ref{wc} is isomorphic to the central extension of the commutative Lie algebra $(\n_\C\oplus \n^*_\C)\otimes \C((t-x))$, by a one-dimensional centre $\C\mathbf 1$, defined by the cocycle $\res_{t-x} \langle f , g \rangle$.
The identification of generators is
\be a_\alpha[n] = E_{\alpha}\otimes (t-x)^n, \quad a^*_\alpha[n] = E_\alpha^*\otimes (t-x)^{n-1}\ee 
with $E_{\alpha}^*\in \n^*$, $\alpha\in\Delta^+$ the dual basis of $E_{\alpha}\in \n$, $\alpha\in\Delta^+$.
In this way we associate a copy of $\Heis(\g)$, call it $\Heis_{(i)}$,  to each of the points $x_i$, $i=1,2,\dots,p$.

\subsection{The algebras $\Heis_{p}\oplus \h_p$ and $\Heis_{\bm x}^\Gamma(t)\oplus\h_{\bm x}^{\Gamma}$} 
Let $\Heis_p$ be the Lie algebra obtained by extending the commutative Lie algebra
\be \bigoplus_{i=1}^p (\n_\C\oplus \n^*_\C)\otimes \C((t-x_i)),\ee
by a one-dimensional centre $\C\mathbf 1$, defined by the cocycle
\be \Omega^\Heis(f_1,\dots,f_p;g_1,\dots,g_p) := \sum_{i=1}^p \res_{t-x_i} \langle f_i,g_i \rangle.\label{Hcocycle} \ee
Equivalently, $\Heis_p$ is 
%the Lie algebra obtained taking the direct sum of the $\Heis_{(i)}$ and then identifying all their central elements $\mathbf 1_i$. That is, $\Heis_p$ is 
the quotient of $\bigoplus_{i=1}^p\Heis_{(i)}$ by the ideal spanned by $\mathbf 1^{(i)} - \mathbf 1^{(j)}$, $1\leq i\neq j\leq p$.

Concretely, $\Heis_p$ is the Lie algebra generated by $a^*_\alpha[n]^{(i)}$, $a_\alpha[n]^{(i)}$, $1\leq i\leq p$, $\alpha\in\Delta^+$, $n\in \Z$, and central element $\mathbf 1$, with commutation relations
\be 
[a_\alpha[n]^{(i)},a_\beta[m]^{(j)}] = 0,\quad 
[a_\alpha[n]^{(i)},a^*_\beta[m]^{(j)}] = \delta^{ij} \delta_{\alpha\beta} \delta_{n,-m} \mathbf 1,\quad 
[a^*_\alpha[n]^{(i)},a^*_\beta[m]^{(j)}] = 0,\ee
for $\alpha,\beta\in\Delta^+$, $n,m\in \Z$, cf. \eqref{HLdef}.
The identification of generators is
\be  a_\alpha[n]^{(i)} = E_{\alpha}\otimes (t-x_i)^n,\quad a^*_\alpha[n]^{(i)} = E_\alpha^*\otimes (t-x_i)^{n-1}. \label{genid}\ee

Next we define $\Heis_{\bm x}^\Gamma$ to be the commutative Lie algebra
\be \Heis_{\bm x}^\Gamma := 
\left(\n_\C   \otimes \C^\8_{\Gamma\bm x}(t)\right)^{\Gamma,0} \oplus
\left(\n^*_\C \otimes \C^\8_{\Gamma\bm x}(t)\right)^{\Gamma,-1}
\label{Hxdef},\ee
cf. \eqref{AGkdef}. Here the action of the generator $\omega$ of $\Gamma$ on $\n$ is given by $\omega.X=\sigma X$, and its  action on $\n^*$  by $\omega.\eta := \lsigma \eta = \eta\circ \sigma^{-1}$.
There is then an embedding of commutative Lie algebras
\be\iota: \Heis_{\bm x}^\Gamma(t) \longhookrightarrow \bigoplus_{i=1}^p (\n_\C\oplus \n^*_\C)\otimes \C((t-x_i))  \ee
as in \eqref{tin}. By the $\Gamma$-equivariant residue theorem -- see Appendix \ref{sec:tert} --  the restriction of the cocycle \eqref{Hcocycle} to the image of this embedding vanishes, and therefore the embedding lifts to an embedding of Lie algebras 
\be \Heis_{\bm x}^\Gamma \longhookrightarrow \Heis_p.\ee

At the same time we also have commutative Lie algebras
\be \h_p := \bigoplus_{i=1}^p \h\otimes \C((t-x_i)),\quad\text{and}\quad
    \h_{\bm x}^\Gamma := \{ f\in \h\otimes \C^\8_{\Gamma\bm x}(t): \sigma f(t) = f(\omega t) \},\ee
and an embedding 
\be  \h_{\bm x}^\Gamma \longhookrightarrow \h_p.\label{hhi}\ee
Note that in contrast to $\Heis_p$ and $\gh_N$, $\h_p$ is not centrally extended so there is no need to worry about whether this embedding lifts.

\subsection{``Big'' versus ``little'' swapping}\label{sec:functoriality}
At this stage we have the Lie algebra $\Heis_p\oplus \h_p$, which is the direct sum, with central charges identified, of the ``local'' copies $\Heis_{(i)}\oplus \h\otimes\C((t-x_i))$ of the Lie algebra $\Heis(\g)\oplus\h\otimes\C((t))$, and embedded within $\Heis_p\oplus \h_p$ we have the ``global'' Lie algebra $\Heis_{\bm x}^\Gamma\oplus \h_{\bm x}^\Gamma$. 

Suppose now that $M_{(i)}$, $1\leq i\leq p$, are smooth $\Heis(\g)\oplus \h\otimes\C((t))$-modules on which $\mathbf 1$ acts as $1$. Then 
\be M_p:=\bigotimes_{i=1}^p M_{(i)}\nn\ee 
is a module over the Lie algebra $\Heis_{\bm x}^\Gamma\oplus \h_{\bm x}^\Gamma$ via the latter's embedding into $\Heis_p\oplus \h_p$. The $M_{(i)}$ are also $\gh$-modules, as in \eqref{smp}, and hence $M_p$ is a module over $\g_{\bm x}^\Gamma$ via \emph{its} embedding into $\gh_p$. 
One then has two spaces of coinvariants, with respect to these two different ``global'' Lie algebras:
\be M_p \big/ \g_{\bm x}^\Gamma\qquad\text{and}\qquad M_p \big/ \left(\Heis_{\bm x}^\Gamma\oplus \h_{\bm x}^\Gamma\right).\ee 
If the Heisenberg algebra and free-field construction are to be of use in solving the model introduced in \S\ref{ZTGsec}, it is necessary to relate these two spaces.

In the usual case where $\Gamma=\{1\}$ it turns out (for details see \cite[\S14.1.3]{FB})  that there is a well-defined linear map
\be M_p \big/ \g_{\bm x}\longrightarrow M_p \big/\left(\Heis_{\bm x}\oplus \h_{\bm x}\right)\label{untwistedfunctoriality}\ee which sends the class of any $v\in M_p$ in $M_p \big/ \g_{\bm x}$  to its class in $M_p \big/\left(\Heis_{\bm x}\oplus \h_{\bm x}\right)$. That is, the following diagram commutes:\be\begin{tikzpicture}    
\matrix (m) [matrix of math nodes, row sep=3em,    
column sep=4em, text height=2ex, text depth=1ex]    
{     
M_p &   \\    
  M_p \big/\g_{\bm x}  &M_p\big/ \left(\Heis_{\bm x}\oplus \h_{\bm x}\right).   \\    
};    
\path[->,font=\scriptsize,shorten <= 2mm,shorten >= 2mm]    
(m-1-1) edge node [above] {} (m-2-2)  
(m-2-1) edge node [above] {} (m-2-2)    
(m-1-1) edge node [left] {} (m-2-1);    
\end{tikzpicture}\nn\ee
Intuitively speaking, in $M_p \big/ \left(\Heis_{\bm x}\oplus \h_{\bm x}\right)$ one is by definition allowed to ``swap'', cf. \S\ref{sec:twistedswapping}, using rational functions of the form 
\be \frac{ X}{(t- x_i)^n},\quad n\in \Z_{\geq 1},\label{swapfn}\ee 
where $X$ is one of $a_\alpha,a^*_\alpha$, and $b_i$, and what \eqref{untwistedfunctoriality} asserts is that in fact one is also allowed to ``swap'' using functions with $X= E_{\alpha}, F_\alpha$ or $H_{\alpha_i}$.
It should be stressed that this is not a trivial statement: the Lie algebra $\g_{\bm x}$ does not embed into $\Heis_{\bm x}\oplus \h_{\bm x}$ or into its universal envelope, for example.
Call the former ``little'' swapping and the latter ``big'' swapping (cf.  Remark \ref{smprem}).

The cyclotomic case turns out to have subtleties of its own, details of which can be found in a companion paper, \cite{VY}. Here we merely sketch the situation before quoting, from \cite{VY}, the specific result we need. First, the statement \eqref{untwistedfunctoriality} is \emph{not} in general true when $\g_{\bm x}$  and $\Heis_{\bm x}\oplus \h_{\bm x}$ are replaced by their $\Gamma$-equivariant counterparts $\g_{\bm x}^\Gamma$ and $\Heis_{\bm x}^\Gamma\oplus \h_{\bm x}^\Gamma$. To gain insight into why it fails, consider the linear isomorphism
\be \left(M_p \otimes \WW_0\right)\big/ \left(\Heis_{\bm x,u}^\Gamma\oplus \h_{\bm x,u}^\Gamma\right) 
\cong_\C M_p\big/\left(\Heis_{\bm x}^\Gamma\oplus \h_{\bm x}^\Gamma\right)
\ee
where $\WW_0$ is assigned to a point $u\in \C$. Via this isomorphism, given some  $\bm m\otimes X \in M_p\otimes \WW_0$ we obtain an element of $M_p\big/\left(\Heis_{\bm x}^\Gamma\oplus \h_{\bm x}^\Gamma\right)$, call it $F(u)$, which depends rationally on $u$. 

When $\Gamma=\{1\}$, $F(u)$ has poles at most at the points $x_i$, $1\leq i\leq p$ and (one can show that) the statement of swapping using \eqref{swapfn} is nothing but the statement of the residue theorem for the rational function $F(u)/(u-x_i)^n$. 
If $X$ is an ``elementary'' state like $a_\alpha[-1]\wac\in \WW_0$, one recovers the ``little'' swapping that holds in $M_p \big/ \left(\Heis_{\bm x}\oplus \h_{\bm x}\right)$  by definition. But for other states -- and in particular for states like $\rho(F_\alpha[-1]\vac)\in \WW_0$ -- one obtains ``big'' swapping. 

However, when $\Gamma\neq \{1\}$ it can happen that $F(u)$ has, in addition to the poles at the points $\omega^kx_i$, $\omega^k\in \Gamma$, $1\leq i\leq p$, a pole also at $u=0$. This is not unnatural, since the origin is  singled out by being the fixed point of the action of $\Gamma$, but it is nonetheless striking since no module is assigned there. One must include this extra pole in the sum over residues, and then the vanishing of this sum is in general no longer the correct statement of ``big'' swapping.

To cure this problem, we are led to \emph{introduce} a module assigned to the origin, judiciously chosen to eliminate this extra, unwanted, singularity. Thus, instead of  $\left(\Heis_{\bm x}^\Gamma\oplus \h_{\bm x}^\Gamma\right)$ we consider $\left(\Heis_{\bm x,0}^\Gamma\oplus \h_{\bm x,0}^\Gamma\right)$, \ie we allow rational functions with poles at the origin. The Laurent expansion at the origin of an element of $\left(\Heis_{\bm x,0}^\Gamma\oplus \h_{\bm x,0}^\Gamma\right)$ belongs to $\Heis(\g)^{\Gamma} \oplus (\h \otimes \C((t)))^\Gamma$, where $\Heis(\g)^\Gamma$ is the subalgebra of $\Heis(\g)$ given by
\be \Heis(\g)^\Gamma \cong_\C (\n_\C \otimes \C((t))^{\Gamma,0} \oplus (\n^*_\C\otimes \C((t)))^{\Gamma,-1} \oplus \C \mathbf 1.\ee
If we then define $H_{p,0}$ to be the quotient of $\bigoplus_{i=1}^p \Heis_{(i)} \oplus  \Heis(\g)^{\Gamma}$
by the ideal spanned by $\mathbf 1^{(i)} - T\mathbf 1^{(0)}$, $1\leq i\leq p$, then we have an embedding of Lie algebras
\be  \Heis_{\bm x,0}^\Gamma \hookrightarrow H_{p,0}\ee
(using part (2) of the lemma of Appendix A). We can assign to the origin any smooth module $M_0$ over $H(\g)^{\Gamma} \oplus (\h \otimes \C((t)))^\Gamma$ on which $T\mathbf 1^{(0)}$ acts as 1, and form the space of coinvariants
\be (M_p\otimes M_0)\big/\left(\Heis_{\bm x,0}^\Gamma\oplus \h_{\bm x,0}^\Gamma\right).\ee

\begin{thm}[\cite{VY}]\label{thm:coinv}
Suppose that there exists a non-zero vector $m_0\in M_0$ with the property that, in the space of coinvariants $(\WW_0 \otimes M_p \otimes M_0)\big/\left(\Heis_{u,\bm x,0}^\Gamma\oplus \h_{u,\bm x,0}^\Gamma\right)$,
\begin{equation*}
\iota_u \left[\rho(A[-1]\vac) \otimes \bm m \otimes m_0\right]
\end{equation*}
is a Taylor series in $u$, for all $A\in \g$ and $\bm m\in M_p$.
Then there is a well-defined linear map 
\be M_p \big/ \g_{\bm x}^\Gamma\to \left(M_p\otimes M_0 \right) \big/ \left(\Heis_{\bm x,0}^\Gamma\oplus \h_{\bm x,0}^\Gamma\right)\ee 
which sends the class of $v\in M_p$ in $M_p\big/ \g_{\bm x}^\Gamma$  to the class of $v\otimes m_0$ in $\left(M_p\otimes M_0\right) \big/\left(\Heis_{\bm x,0}\oplus \h_{\bm x,0}\right)^\Gamma$.
\qed\end{thm}

If $\mf a$ is a Lie algebra and $M$ an $\mf a$-module, a linear functional $\tau:M\to \C$ on $M$ is said to be \emph{$\mf a$-invariant} if $\tau( a\on x) =0$ for all $a\in \mf a$ and all $x\in M$. 

\begin{cor}\label{ginvcor}
Suppose $\C m_0$ is as in Theorem \ref{thm:coinv}. If $\tau: M_p\otimes \C m_0\to \C$ is an $\Heis_{\bm x, 0}^\Gamma\oplus \h_{\bm x, 0}^\Gamma$-invariant linear functional, then the linear functional $\tau(\cdot \otimes m_0)$ on $M_p$ is  $\g_{\bm x}^\Gamma$-invariant. 
\end{cor}
\begin{proof}
By definition, $\tau(\cdot\otimes m_0)$ is a linear map $M_p\to \C$ that factors through $\left(M_p\otimes \C m_0\right)\big/ \left(\Heis_{\bm x, 0}^\Gamma\oplus \h_{\bm x, 0}^\Gamma\right)$. Theorem \ref{thm:coinv} asserts that any such map can also be factored through $M_p\big/ \g_{\bm x}^\Gamma$, i.e. is $\g_{\bm x}^\Gamma$-invariant.
\end{proof}

At the same time, smooth $\Heis(\g)\oplus \h\otimes\C((t))$-modules are also modules over the copy of $\nG\otimes\C((t))$ of \S\ref{Gsec}. With the obvious modifications, Theorem \ref{thm:coinv} and its corollary also hold with $\gh$ replaced by $\nG\otimes\C((t))$. 
That is, if $m_0 \in M_0$ is such that in the space of coinvariants $(\WW_0 \otimes M_p \otimes M_0)\big/\left(\Heis_{u, \bm x, 0}^\Gamma\oplus \h_{u, \bm x, 0}^\Gamma\right)$  
we have that $\iota_u \left[G_i[-1]\wac \otimes \bm m \otimes m_0\right]$ is a Taylor series in $u$ for all $\bm m \in M_p$ and $i \in I$,
then for every $\Heis_{\bm x,0}^\Gamma\oplus \h_{\bm x,0}^\Gamma$-invariant linear functional $\tau: M_p\otimes M_0\to \C$, $\tau(\cdot\otimes m_0)$ is invariant under $(\nG)_{\bm x}^\Gamma$.

\subsection{Wakimoto modules}\label{sec:wakm} 
The discussion of \S\ref{sec:functoriality} applies to any smooth $\Heis(\g)\oplus \h\otimes\C((t))$-modules $M_{(i)}$ of equal levels assigned to points $x_i$. In addition to copies of $\WW_0$, the other class of such modules we need are the \emph{Wakimoto modules}, whose definition we now recall. Given any $\chi \in \h^*\otimes \C((t))$, let $\C v_\chi$ denote the one-dimensional $\h\otimes \C((t))$-module with
\be f \on v_\chi = v_\chi \res_t \chi(f).\ee 
Then the \emph{Wakimoto module} $W_\chi$ is by definition the $\Heis(\g) \oplus \h\otimes\C((t))$-module 
\be W_\chi := \Mh \otimes \C v_\chi,\ee  
where $\Mh$ is the induced module over $\Heis(\g)$ defined by \eqref{Mdef}, or, equivalently, by
\be \Mh := 
\Ind_{(\n_\C\oplus \n^*_\C) \otimes \C[[t]] \oplus \C\mathbf 1}^{\Heis(\g)} \C\wac
\ee
with $\C\wac$ is the trivial one-dimensional module over $(\n_\C\oplus \n^*_\C) \otimes \C[[t]] \oplus \C\mathbf 1$.

(The module $W_\chi$ should be compared to $\WW_0 := \Mh \otimes \pi_0$, \eqref{WW0def}, which is induced in \emph{both} summands, $\Heis(\g)$ and $\h\otimes\C((t))$.)

Observe that a Wakimoto module $W_{\chi}$, $\chi\in \h^*\otimes \C((t))$, is a smooth module
over $\Heis(\g)\oplus \h\otimes\C((t))$. So by \eqref{smp}, $W_{\chi}$ is also a module over $\gh$.
We shall need two facts concerning the structure of $W_\chi$ as a $\gh$-module. First, recall the definition of $G_i[-1]\wac\in \Mh$ from \S\ref{Gsec}. The following was proved in \cite{FFR}.
Suppose 
\be \mu = -\frac{\alpha_k}{t} + \sum_{n=0}^\8 \mu^{(n)} t^n,\qquad \mu^{(n)}\in \h^*, \label{Gsing}\ee 
for some simple root $\alpha_k$; then $G_k[-1]\wac\otimes v_\mu\in W_{\mu}$ is a singular vector of imaginary weight (cf. \S\ref{singsec})  if and only if $\langle \alpha_k,\mu^{(0)}\rangle = 0$.

Second, the Wakimoto module $W_{\chi}$ inherits the grading of $\Mh$. In particular the subspace of grade 0 -- call it $\widetilde W_{\chi}$ -- is generated from the vacuum $\wac$ by the operators $a^*_{\alpha}[0]$, $\alpha\in \Delta^+$. This subspace is stable under the action of the Lie subalgebra $\g$ of $\gh$, and is isomorphic as a $\g$-module to the contragredient Verma module, \S\ref{sec:verm}, of highest weight $\res_t(\chi)\in \h^*$:
\be \widetilde W_{\chi} \cong_\g M_{\res_t(\chi)}^* 
%:= \Hom^\textup{res}_{U(\mf b_-)}(U(\g), \C_{\res_t(\chi)})
.\label{wmi}\ee
(See \cite{FFR} and also \cite{FB}, \S11.2.6. Also cf. Proposition \ref{prop: Mshift} below.)

We now assign to each of the points $x_i$, $1\leq i\leq p$, a Wakimoto module 
\be W_{\chi_i}:= \Mh_i \otimes \C v_{\chi_i} \ee
over the local copy $\Heis_{(i)}\oplus \h\otimes\C((t-x_i))$ of the Lie algebra $\Heis(\g)\oplus\h\otimes\C((t))$. 

To the origin we assign a module $W_{\chi_0}^\Gamma$ defined as follows. Given  $\chi_0\in (\h^*\otimes \C((t)))^{\Gamma,-1}$ let $\C v_{\chi_0}$ denote the one-dimensional $(\h\otimes \C((t)))^\Gamma$-module defined by
\be f(t)\on v_{\chi_0} = v_{\chi_0} \frac 1 T \res_t \chi_0(f) .\label{vldef}\ee
%That is, explicitly,
%\be b_i[k]\on v_{\lambda_0} \equiv (H_{\alpha_i}\otimes t^k)\on v_{\lambda_0} = \begin{cases} v_{\lambda_0} %\lambda_0(H_{\alpha_i})/T & k=0, \\ 
%0 & k \in\Z_{\geq 1}.\end{cases}\ee 
Let $\Mh^\Gamma$ be the induced $\Heis(\g)^\Gamma$-module
\be \Mh^\Gamma := 
\Ind_{(\n_\C \otimes \C[[t]])^{\Gamma,0} \oplus (\n^*_\C \otimes \C[[t]])^{\Gamma,-1} \oplus \C\mathbf 1}^{\Heis(\g)^\Gamma} \C\wac^\Gamma
\ee
with $\C\wac^\Gamma$ the one-dimensional module over $(\n_\C \otimes \C[[t]])^{\Gamma,0} \oplus (\n^*_\C \otimes \C[[t]])^{\Gamma,-1} \oplus \C\mathbf 1$ on which $\mathbf 1$ acts as $\frac 1 T$ and the first two summands act as zero. Then 
\be W_{\chi_0}^\Gamma := \Mh^\Gamma \otimes \C v_{\chi_0}.\ee

Now $\bigotimes_{i=1}^p \Mh_i \otimes \Mh^\Gamma $ is an induced module over $H_{p,0}$ and, cf. Lemma \ref{complem}, we have that the space of coinvariants with respect to $\Heis_{\bm x,0}^\Gamma$ is of dimension one:
\be \left.\bigotimes_{i=1}^p \Mh_i \otimes \Mh^\Gamma \right/ \Heis_{\bm x,0}^\Gamma \cong_\C \C \wac^{\otimes p} \otimes \C \wac^\Gamma\cong \C.\label{MHcoinv}\ee

Meanwhile, we have the embedding of Lie algebras
\be (\iota_{t-x_1},\dots,\iota_{t-x_p},\iota_{t-0}) : \h_{\bm x,0}^\Gamma \longrightarrow \h_p \oplus (\h \otimes \C((t)))^\Gamma\ee
 -- where we now, cf. \eqref{hhi}, take also the Taylor expansion at the origin -- and in this way $v_{\bm \chi} \otimes v_{\chi_0}$ generates a one-dimensional $\h_{\bm x,0}^\Gamma$-module:
\be f(t) \on  \left(v_{\bm \chi} \otimes v_{\chi_0}\right) =  \left(v_{\bm \chi} \otimes v_{\chi_0}\right)\left(\sum_{i=1}^p \left(\res_{t-x_i}\chi_i(\iota_{t-x_i} f(t)\right) + \frac 1 T \res_{t-0}\chi_0(\iota_{t-0}f(t)))\right) \label{factv}\ee
cf. \eqref{vldef}.

Now, with $\lsigma$ as in \eqref{lsigmadef}, let
\be \h_{\bm x,0}^{*,\Gamma} :=  
\left(\h^*\otimes \C^\8_{\Gamma\bm x\cup\{0\}}(t)\right)^{\Gamma,-1} = 
 \{ f\in \h^*\otimes \C^\8_{\Gamma\bm x\cup\{0\}}(t): f(\omega t) = \omega^{-1} \lsigma f(t) \}.\ee

\begin{prop}\label{prop:vh} The space of coinvariants
$\left.\C v_{\bm \chi}\otimes \C v_{\chi_0}\right/ \h_{\bm x,0}^\Gamma$
is one-dimensional if and only if there exists a $\chi(t)\in \h_{\bm x,0}^{*,\Gamma}$ such that $\chi_i =\iota_{t-x_i} \chi(t)$ for each $i$, $0\leq i\leq p$.
%, and $\lambda_0/t \equiv \iota_t\chi(t)\mod \h^*[[t]]^\Gamma$.
 Otherwise it is zero-dimensional. 
\end{prop}
\begin{proof}
Certainly either  $\h_{\bm x,0}^\Gamma\on \left(\C v_{\bm \chi}\otimes \C v_{\chi_0}\right) = \C v_{\bm \chi}\otimes \C v_{\chi_0}$, in which case the space of coinvariants is zero-dimensional, or $\h_{\bm x,0}^\Gamma\on \left(\C v_{\bm \chi}\otimes \C v_{\chi_0}\right) = \{0\}$ and the space of coinvariants is of full dimension, \ie dimension one. But $\h_{\bm x,0}^\Gamma\on \left(\C v_{\bm \chi}\otimes \C v_{\chi_0}\right) = \{0\}$ means \eqref{factv} vanishes for all $f(t)\in \h_{x,0}^\Gamma$. By the lemma in Appendix \ref{sec:tert} (with $A=\h^*$, $B=\h$,   $\langle\cdot,\cdot\rangle$ the canonical pairing between them, and $\omega$ acting as $\lsigma$ on $\h^*$ and as $\sigma$ on $\h$) this occurs precisely when such a $\chi(t)\in \h_{\bm x,0}^{*,\Gamma}$ exists.
\end{proof}

\subsection{The weight $\lambda_0$ and the $\g^\sigma$-module structure of $\widetilde W_{\chi_0}^\Gamma$} \label{sec: origin}
We now need to show that the module $W_{\chi_0}^\Gamma$ at the origin can in fact be chosen so that the conditions of Theorem \ref{thm:coinv} are met. For any $h \in \h$ we let
\begin{equation} \label{lambda0def}
\lambda_0(h) := \sum_{r=1}^{T-1} \frac {\tr_\n (\sigma^{-r} \ad_h)} {1 - \omega^r}  = \sum_{r=1}^{T-1} \frac{1}{1 - \omega^r} \sum_{\substack{\alpha\in \Delta^+\\\sigma^r(\alpha)=\alpha}} \left( \prod_{p=0}^{r-1} \tau_{\sigma^p(\alpha)}^{-1} \right) \alpha(h),
\end{equation}
where $\ad_h:\n\to\n; X\mapsto [h,X]$ is the adjoint action of $\h$ on $\n$. 
This defines a weight $\lambda_0\in \h^*$. In fact, moreover, 
\be \lambda_0 = \lambda_0\circ \Pi_0\label{lP}\ee 
where $\Pi_k$, $k\in \Z_T$, are the projectors $\Pi_k := \frac 1 T\sum_{m=0}^{T-1} \omega^{-mk} \sigma^m : \g \to \g$, which obey $\sum_{k\in \Z_T} \Pi_k=\id$.  To see \eqref{lP}, note that $[\Pi_k\h,\Pi_\ell \n] \subseteq \Pi_{\ell+ k} \n$, $\sigma(\Pi_\ell \n) = \Pi_\ell\n$, and hence for all nonzero $r,k\in \Z_T$,
$\tr_\n(\sigma^r \ad_{\Pi_k h}) = \sum_{\ell\in \Z_T} \tr_{\Pi_\ell\n}(\sigma^r \ad_{\Pi_k h}) =0$. 

Note that the definition of $\lambda_0$ depends solely on the choice of $\g$, $\sigma$ and $T$. 

Now recall from Remark \ref{smprem} that smooth modules over $\Heis(\g)\oplus \h \otimes \C((t))$ become modules over a ``big'' Lie algebra $U(\WW_0)$ spanned by the formal modes of all states in $\WW_0$, and that $\gh$ embeds in $U(\WW_0)$. There is also a $\Gamma$-equivariant version of this construction \cite[\S5.2]{VY}. Namely, smooth modules over $\Heis(\g)^\Gamma\oplus (\h \otimes \C((t)))^\Gamma$, such as $W^\Gamma_{\chi_0}$, become modules over the subalgebra $U(\WW_0)^\Gamma$ of equivariant elements of $U(\WW_0)$, and the twisted affine algebra $\gh^\Gamma$ embeds in $U(\WW_0)^\Gamma$. 
The action of $U(\WW_0)^\Gamma$ on $W^\Gamma_{\chi_0}$ is specified by a \emph{quasi-module map} $Y_W(\cdot,u): \WW_0 \to \Hom(W^\Gamma_{\chi_0},W^\Gamma_{\chi_0}((u)))$,\footnote{Quasi-modules over vertex algebras were introduced  by Li,  \cite{Li1,Li3}. They are closely related to twisted modules.} $$Y_W(A,u)=: \sum_{n\in \Z} A^W_{(n)} u^{-n-1},$$ and this map $Y_W$ can in fact be defined by the relation 
\begin{equation} \label{YWdef}
\iota_u [X \otimes \bm m \otimes v] = [\wac \otimes \bm m \otimes Y_W(X,u) v].
\end{equation}
Explicitly, the action of the formal mode $A[n] \in U(\WW_0)^\Gamma$, $n \in \Z$ of some state $A \in \WW_0$ on $v \in W^\Gamma_{\chi_0}$ is given by $A[n] v = \frac{1}{T} A^W_{(n)} v$ -- for details see \cite{VY}. 

In this way, $W^\Gamma_{\chi_0}$ is a module over $\gh^\Gamma$. 
In particular, its grade 0 subspace $\widetilde W^\Gamma_{\chi_0}$, cf. \eqref{wmi}, is a module over $\g^\sigma:= \Pi_0\g$. 

\begin{prop}\label{prop: Mshift}
There is an isomorphism of $\g^\sigma$-modules
\begin{equation*}
\widetilde W^\Gamma_{\chi_0} \cong_{\g^\sigma} M^{\ast, \sigma}_{\frac{1}{T} (\res_t(\chi_0) - \lambda_0)}
\end{equation*}
where $M^{\ast, \sigma}_{\lambda}$ denotes the contragredient Verma module over $\g^\sigma$ of highest weight $\lambda\in \h^{*,\sigma}$. 
\end{prop}
\begin{proof}  
First, for all $v\in \widetilde W^\Gamma_{\chi_0}$,
\begin{align} \res_u \iota_u 
\left[\rho\left(E_{\alpha}[-1]\vac\right) \otimes \bm m \otimes v\right]
&= \res_u \iota_u 
\left[\sum_{\beta\in \Delta^+} P^\beta_\alpha(a^*[0]) a_\beta[-1]\wac \otimes \bm m \otimes v \right] \nn\\
&= \res_u \iota_u\frac 1 u 
\left[\sum_{\beta\in \Delta^+} P^\beta_\alpha(a^*[0])\wac \otimes \bm m \otimes \sum_{l=0}^{T-1} \sigma^la_\beta[0] v\right] \nn\\
&= \left[\wac \otimes \bm m \otimes \sum_{\beta\in \Delta^+} P^\beta_\alpha\left(\sum_{k=0}^{T-1} \sigma^k a^*[0]\right) \sum_{l=0}^{T-1} \sigma^la_\beta[0] v \right].
\label{443}\end{align}
The first equality here is by ``swapping'', cf. \S\ref{sec:twistedswapping}, the operator $a_\beta[-1]$ using the rational function
\begin{equation} \label{swapEF}
\sum_{k\in \Z_T} \frac{\sigma^k a_\beta}{\omega^{-k} t-u} \in \Heis_{u}^\Gamma.
\end{equation} 
Note that the ``self-interaction'' term vanishes because \eqref{grPQ} ensures that $a^*_\beta[0]$ cannot occur in the polynomial $P_\alpha^\beta$. For the second equality we repeatedly first ``swap'' using a function of the form $\sum_{k\in \Z_T} \frac{\omega^{-k} \sigma^k a^*_\gamma}{ \omega^{-k} t - u} \in \Heis_u^\Gamma$ and then note that the result can be re-written as the result of ``swapping'' using a function of the form $\sum_{k=0}^{T-1} \frac{\omega^{-k} \sigma^k a_\gamma^*}{\omega^{-k}t} = \frac 1 t\sum_{k=0}^{T-1}\sigma^k a_\gamma^*$, keeping only the leading term in $u$ at each step because we are computing the residue. 

Now, cf. Remark \ref{coordrem}, we can choose to work in coordinates on $N_+$ adapted to $\sigma$. Namely, we can pick a basis $\{E_{(i,\alpha)}: 1\leq i\leq T-1, \alpha\in \Delta_i^+\}$ of $\n$, where for each $i$, $E_{(i,\alpha)}\in \Pi_i\n$ and $\alpha$ runs over the set $\Delta^+_i$ of $\g^\sigma$-weights of $\Pi_i\n$. (The Lie algebra $\g^\sigma$ acts on $\Pi_i\g$ by the adjoint action. The $E_{(i,\alpha)}$ are not all root vectors of $\g$ unless $\sigma$ is inner.) Let $x_{(i,\alpha)}$ and $\del_{(i,\alpha)}$, $1\leq i\leq T-1$, $\alpha \in \Delta^+_i$, be homogeneous coordinates and derivatives such that $\lsigma x_{(i,\alpha)} = \omega^i x_{(i,\alpha)}$ and $\lsigma \circ \del_{(i,\alpha)} \circ \lsigma^{-1} = \omega^i \del_{(i,\alpha)}$. (For example, define $x_{(i,\alpha)}(n)$, $n\in N_+$, by $n= \exp(\sum_{i=0}^{T-1} \sum_{\alpha\in \Delta^+_i} x_{(i,\alpha)} E_{(i,\alpha)})$.) 
Then \eqref{443} gives in particular 
\be\res_u \iota_u 
\left[\rho\left(E_{(0,\alpha)}[-1]\vac\right) \otimes \bm m \otimes v\right]
= \left[\wac \otimes \bm m \otimes \sum_{\beta\in \Delta^+_0} P^{(0,\beta)}_{(0,\alpha)}\left( T a_{(0,\bullet)}^*[0]\right) T a_{(0,\beta)}[0] v \right]
\ee
for these are the only terms that survive the projections $\frac{1}{T} \sum_{k=0}^{T-1}\sigma^k$ on the right of \eqref{443}. That is, cf. \eqref{YWdef},
\begin{equation}
\left(\rho(E_{(0,\alpha)}[-1]\vac)\right)[0] v = \frac{1}{T} \left(\rho\left(E_{(0,\alpha)}[-1]\vac\right)\right)^W_{(0)} \on v = \sum_{\beta\in \Delta^+_0} P^{(0,\beta)}_{(0,\alpha)}\left( T a_{(0,\bullet)}^*[0]\right) a_{(0,\beta)}[0] v.\label{444}
\end{equation}

By definition $\widetilde W^\Gamma_{\chi_0} = \C\big[a^{\ast}_{(0,\alpha)}[0]\big]_{\alpha\in \Delta_0^+} v_{\chi_0}$, which is naturally identified as a vector space with the polynomial algebra $\C[x_{(0,\alpha)}]_{\alpha\in \Delta_0^+}$, with $v_{\chi_0}$ identified with $1$. As in \S\ref{fds} we have the realization of the Lie algebra $\g^\sigma$ in terms of first-order differential operators acting on this polynomial algebra. In particular $E_{(0,\alpha)}$ is realized as $\sum_{\beta\in \Delta^+_0} P_{(0,\alpha)}^{(0,\beta)}(x_{(0,\bullet)}) \del_{(0,\beta)}$. Recall that the generator $\mathbf 1$ acts on $\widetilde W^\Gamma_{\chi_0}$ as $\frac{1}{T}$, so that we have a homomorphism of Weyl algebras $x_{(0,\alpha)} \mapsto T a^{\ast}_{(0, \alpha)}[0]$ and $\partial_{(0,\alpha)} \mapsto a_{(0, \alpha)}[0]$. Thus, what \eqref{444} shows is that the identification  $\widetilde W^\Gamma_{\chi_0} \cong \C[x_{(0,\alpha)}]_{\alpha\in \Delta_0^+}$ is an isomorphism of $\n^\sigma$-modules.

Recall that, as a module over $\n^\sigma$, the contragredient Verma module $M^{\ast, \sigma}_{\lambda}$ is co-free on one co-generator.\footnote{and this module structure is independent of the $\g^\sigma$-weight $\lambda$; indeed, $M^{\ast, \sigma}_\lambda \cong_{\n^\sigma} \Hom^\textup{res}_{\C}(U(\n^\sigma),\C)=: U(\n^\sigma)^\vee$.} 
There is an identification of $\C[x_{(0,\alpha)}]_{\alpha\in \Delta_0^+}$ with $M^{\ast, \sigma}_{\lambda}$ as $\g^\sigma$-modules, given by modifying the realization of $\g^\sigma$ by differential operators on $N^\sigma_+$ by a cocycle specified by the weight $\lambda$; see e.g. \cite[\S11.2.6]{FB}. This modification does not alter the action of $\n^\sigma$. So $\widetilde W^\Gamma_{\chi_0}\cong \C[x_{(0,\alpha)}]_{\alpha\in \Delta_0^+}$ is also co-free as a module over $\n^\sigma$. 

Next we should compute $\left(\rho(H_{\alpha_i}[-1]\vac)\right)[0] v_{\chi_0} = \frac{1}{T} \left(\rho(H_{\alpha_i}[-1]\vac)\right)^W_{(0)}\on v_{\chi_0}$. 
Consider therefore $\iota_u \left[\rho\left(H_{\alpha_i}[-1]\vac\right) \otimes \bm m \otimes v_{\chi_0}\right]$. By swapping using the rational function 
\be \sum_{k\in \Z_T} \frac{\omega^{-k} \sigma^k a^*_\beta }{ \omega^{-k} t - u} \in \Heis_u^\Gamma,\ee
cf. \eqref{genid} and \eqref{Hxdef}, the first term gives
\begin{align*}
\iota_u \Bigg[ - &\sum_{\beta\in \Delta^+} \beta(H_{\alpha_i}) a^*_\beta[0] a_\beta[-1]\wac \otimes \bm m \otimes v_{\chi_0} \Bigg]\\
%&= - \sum_{\beta\in\Delta^+} \beta(H_{\alpha_i}) \sum_{k=1}^{T-1}\sum_{p=0}^\8 \frac{1}{(\omega^k-1)^{p+1} u^{p+1}} \, \iota_u \left[ (\sigma^k a_\beta^*)[p+1] a_\beta[-1] \wac \otimes \bm m \otimes v_{\chi_0} \right] + \mathcal{O}(u^0)\\
& = \sum_{k=1}^{T-1} \frac{1}{(1 - \omega^k) u} \sum_{\beta\in\Delta^+} \left( \prod_{p=0}^{k-1} \tau_{\sigma^p(\beta)}^{-1} \right) \beta(H_{\alpha_i}) \, \iota_u \left[ (a_{\sigma^k(\beta)}^*)[1] a_\beta[-1] \wac \otimes \bm m \otimes v_{\chi_0} \right] + \mathcal{O}(u^0)\\
& = - \frac{1}{u} \left[ \wac \otimes \bm m \otimes v_{\chi_0} \right] \lambda_0(H_{\alpha_i}) + \mathcal{O}(u^0),
\end{align*}
with $\lambda_0$ as in \eqref{lambda0def} and where $\mathcal{O}(u^0)$ denotes terms in non-negative powers of $u$. The other term in $\iota_u \left[\rho\left(H_{\alpha_i}[-1]\vac\right) \otimes \bm m \otimes v_{\chi_0}\right]$ is
\begin{align*}
\iota_u [b_i[-1]\wac \otimes \bm m \otimes v_{\chi_0}]
 &= \left[ \wac \otimes \bm m \otimes \frac 1 u \sum_{r\in \Z_T} (\sigma^r b_i)[0]  v_{\chi_0}\right] + \mathcal{O}(u^0)\\
 &= \frac 1 u  \left[ \wac \otimes \bm m \otimes   v_{\chi_0}\right]\frac 1 T \sum_{r\in \Z_T}(\res_t\chi_0)\left(\sigma^r H_{\alpha_i}\right) + \mathcal{O}(u^0)\\
 &= \frac 1 u  \left[ \wac \otimes \bm m \otimes   v_{\chi_0}\right] (\res_t\chi_0)\left(\Pi_0 H_{\alpha_i}\right) + \mathcal{O}(u^0),
\end{align*}
in view of \eqref{vldef} and \eqref{lP}.

Thus, $\widetilde W^\Gamma_{\chi_0}$ is a $\g^\sigma$-module containing a non-zero vector, $v_{\chi_0}$, such that
\be \n^\sigma \on v_{\chi_0} = 0 \nn,\qquad h \on v_{\chi_0} = v_{\chi_0} \frac{1}{T} ( \res_t (\chi_0) - \lambda_0)(h).\ee
Since $M^{\ast, \sigma}_{\frac{1}{T} (\res_t(\chi_0) - \lambda_0)}$ is a coinduced $\g^\sigma$-module, it follows by the universal property of coinduced modules that there is a homomorphism of $\g^\sigma$-modules $\widetilde W^\Gamma_{\chi_0} \to M^{\ast, \sigma}_{\frac{1}{T} (\res_t(\chi_0) - \lambda_0)}$ sending $v_{\chi_0}$ to the highest weight vector in $M^{\ast, \sigma}_{\frac{1}{T} (\res_t(\chi_0) - \lambda_0)}$. We need to show it is a bijection. But it is in particular a homomorphism of two co-free $\n^\sigma$-modules which identifies their co-generators, so it must be an isomorphism of $\n^\sigma$-modules. Hence it is indeed bijective.
\end{proof}
 
\begin{comment}
It is therefore enough to show that the space of singular vectors for $\n^\sigma$ in $\C[x_{(0,\alpha)}]_{\alpha\in \Delta_0^+}$ is spanned by $1$. Trivially, $1$ spans the space of constant polynomials, \ie polynomials in the intersection of the kernels of the $\del_{(0,\alpha)}$, $\alpha\in\Delta^+_0$. But recall that $E_{(0,\alpha)}$ is realized as
\be \del_{(0,\alpha)} + \sum_{\beta} P_{(0,\alpha)}^{(0,\beta)}(x_{(0,\bullet)}) \del_{(0,\beta)} \ee
where the sum is over $\beta\in \Delta^+_i$ strictly higher than $\alpha$ in the principal gradation. Hence, by an induction on the height of $\alpha$ in the principal gradation starting from the highest root and working downwards, we find that the intersection of the kernels of the $E_{(0,\alpha)}$, $\alpha\in \Delta^+_0$, is also spanned by $1$, which is what we had to show. 
\end{comment}

Now we have the following, cf. Theorem \ref{thm:coinv}.
\begin{prop}
Suppose $\chi_0 = \lambda_0/t + \mc O(t^0)\in (\h\otimes \C((t)))^{\Gamma,-1}$. Then, 
in the space of coinvariants $(\WW_0 \otimes M_p \otimes W^\Gamma_{\chi_0})\big/\left(\Heis_{u,\bm x,0}^\Gamma\oplus \h_{u,\bm x,0}^\Gamma\right)$, for any $\bm m \in M_p$ we have that
\begin{equation*} \label{EFgone}
\iota_u \left[\rho(A[-1]\vac) \otimes \bm m \otimes v_{\chi_0}\right],
% \quad
%\iota_u \left[\rho(F_\alpha[-1]\vac) \otimes \bm m \otimes v_{\chi_0}\right] \quad\text{for $\alpha\in\Delta^+$}
\quad\text{for all $A\in \g$,}
\end{equation*}
 and
\begin{equation*} \label{Hgone}
\iota_u \left[G_i[-1]\wac \otimes \bm m \otimes v_{\chi_0}\right], \quad
%\iota_u \left[\rho(H_{\alpha_i}[-1]\vac) \otimes \bm m \otimes v_{\chi_0}\right] \quad
\text{for all $i\in I$,}
\end{equation*}
are Taylor series in $u$.
\end{prop}
\begin{proof}
Let $A\in \g$. By inspection, $\iota_u \left[\rho\left(A[-1]\vac\right) \otimes \bm m \otimes v_{\chi_0}\right]$ has at most a simple pole at $u$. We must show that this pole vanishes, \ie that $\left(\rho\left(A[-1]\vac\right)\right)^W_{(0)} v_{\chi_0} = 0$. The symmetries of $Y_W$ -- see \cite[Lemma 5.4]{VY} -- imply that $\left(\rho\left(A[-1]\vac\right)\right)^W_{(0)} v_{\chi_0} = \left(\rho\left(\Pi_0A[-1]\vac\right)\right)^W_{(0)} v_{\chi_0}$. Thus, what has to be checked is that the $\g^\sigma$-submodule through $v_{\chi_0}$ is the trivial module.  
By the preceeding proposition, if $\chi_0 = \lambda_0/t + \mc O(t^0)$ then $\widetilde W^\Gamma_{\chi_0}\cong_{\g^\sigma} M^{\ast, \sigma}_0$. It is a standard fact about $M^{\ast, \sigma}_0$, easily verified, that the submodule through its highest weight vector is indeed the trivial module. 

It remains to consider $\iota_u \left[G_i[-1]\wac \otimes \bm m \otimes v_{\chi_0}\right]$. For this one sees directly, by an argument just as for $E_\alpha$ in the preceeding proof, that there is no pole term (again, there are no non-zero ``self-interaction'' terms because $\deg R_i^\beta = \beta - \alpha_i$). 
\end{proof}

\subsection{Bethe equations}
Now we pick a tuple $(c(1),\dots,c(m))\in I^m$ of nodes of the Dynkin diagram of $\g$, and consider the element 
\be \lambda(t) := \sum_{r\in \Z_T} \left(\sum_{i=1}^N\frac{\lsigma^r\lambda_i}{t-\omega^rz_i} - \sum_{j=1}^m \frac{\lsigma^r\alpha_{c(j)}}{t-\omega^rw_j}\right) + \frac{\lambda_0}{t} \in \h_{\bm x,0}^{*,\Gamma}.\label{lambdadef}\ee
Here the weights $\lambda_1,\dots,\lambda_N\in \h^*$ are those of \eqref{Hilbert} and $\lsigma\lambda:=\lambda\circ \sigma^{-1}$ as in \eqref{lsigmadef}. 

Let  
\be (\chi_1,\dots,\chi_N,\mu_1,\dots\mu_m) := \iota(\lambda(t)).\label{chimudef}\ee
For each $j$, $1\leq j\leq m$, 
\be \mu_j = -\frac{\alpha_{c(j)}}{t-w_j} 
+ 
\sum_{r=0}^{T-1} \sum_{i=1}^N\frac{\lsigma^r\lambda_i}{w_j-\omega^rz_i}
-
\sum_{r=0}^{T-1}\sum_{\substack{k=1\\k\neq j}}^m \frac{\lsigma^r\alpha_{c(k)}}{w_j-\omega^rw_k}
+
\frac{1}{w_j} \left( \sum_{r=1}^{T-1} \frac{\lsigma^r\alpha_{c(j)}}{\omega^r - 1} + \lambda_0\right)
+\mc O(t-w_j).\nn
\ee
Hence, cf. \eqref{Gsing}, we have that for each $j$, $1\leq j\leq m$, the vector $G_{c(j)}[-1]\wac\in W_{\mu_j}$ is singular if and only if 
\begin{equation} \label{tbe}
0= \sum_{r=0}^{T-1} \sum_{i=1}^N\frac{\langle \alpha_{c(j)},\lsigma^r\lambda_i\rangle}{w_j-\omega^rz_i} - \sum_{r=0}^{T-1} \sum_{\substack{k=1\\k\neq j}}^m \frac{\langle \alpha_{c(j)},\lsigma^r\alpha_{c(k)}\rangle}{w_j-\omega^rw_k} +
\frac{1}{w_j}
\left(- \frac{1}{2} \sum_{r=1}^{T-1} \langle \alpha_{c(j)},\lsigma^r\alpha_{c(j)}\rangle + \langle \alpha_{c(j)}, \lambda_0 \rangle \right).
\end{equation}
Here we have used the fact that $\sum_{r=1}^{T-1} \frac{\langle \alpha_{c(j)},\lsigma^r\alpha_{c(j)}\rangle}{\omega^r - 1} = \sum_{r=1}^{T-1} \frac{\langle \alpha_{c(j)},\lsigma^r\alpha_{c(j)}\rangle}{\omega^{-r} - 1}$ in order to rewrite the first term in brackets as $- \sum_{r=1}^{T-1} \frac{\langle \alpha_{c(j)},\lsigma^r\alpha_{c(j)}\rangle}{2}$.
Let us call these equations \eqref{tbe}, for $1\leq j\leq m$, the \emph{cyclotomic Bethe equations}. For each fixed choice of tuple $(c(1),\dots, c(m))\in I^m$, they form a set of equations on the points $w_1,\dots,w_m$.

\subsection{The cyclotomic weight function and Schechtmann-Varchenko formula}\label{sec:bv} We now specialize to considering the tensor product $\bigotimes_{i=1}^N W_{\chi_i} \otimes \bigotimes_{j=1}^m W_{\mu_j}$, cf. \eqref{chimudef}, of Wakimoto modules assigned to the points $z_1,\dots,z_N,w_1,\dots,w_m$. From \eqref{MHcoinv} and Proposition \ref{prop:vh}, we have
\begin{subequations}\label{taugdef} 
\be \left(\bigotimes_{i=1}^N W_{\chi_i} \otimes \bigotimes_{j=1}^m W_{\mu_j}\otimes \C v_{\chi_0}\right)\bigg/\Heis_{\bm x}^\Gamma \oplus \h_{\bm x}^\Gamma 
\cong_\C \C .\label{isoms}
\ee
There is therefore a unique, up to normalization, $\Heis_{\bm x}^\Gamma\oplus \h_{\bm x}^\Gamma$-invariant linear functional 
\be \tau_\Gamma: \bigotimes_{i=1}^N W_{\chi_i} \otimes \bigotimes_{j=1}^m W_{\mu_j}\otimes \C v_{\chi_0} \to \C.\ee
\end{subequations}
To fix the normalization we may, cf. \eqref{MHcoinv}, set $\tau_{\Gamma}(\wac^{\otimes p}) = 1$, where $\wac^{\otimes p}$ is the vacuum state in $ \Mh^{\otimes p}\cong_\C \bigotimes_{i=1}^N W_{\chi_i} \otimes \bigotimes_{j=1}^m W_{\mu_j}\otimes \C v_{\chi_0}$.
%(By Corollary \ref{thm:coinv}, $\tau_\Gamma$ is also $\g_{\bm x}^\Gamma$-invariant).  
Hence, for each choice of the tuple $(c(1),\dots,c(m))\in I^m$, we have a linear functional 
\be \tau_\Gamma(\cdot , G_{c(1)}[-1] \wac, \dots, G_{c(m)}[-1]\wac)
:  \bigotimes_{i=1}^N W_{\chi_i} \to \C.\ee
Now define the \emph{cyclotomic weight function},
\be \psi_\Gamma= \psi_\Gamma(c(1),\dots,c(m);w_1,\dots,w_m) : \bigotimes_{i=1}^N M^*_{\lambda_i} \to \C \ee
to be the restriction of this functional $\tau_\Gamma$ to the subspace $\bigotimes_{i=1}^N\widetilde W_{\chi_i}\cong_\g \bigotimes_{i=1}^N M^*_{\lambda_i}$, cf. \eqref{wmi} and \eqref{chimudef}. A more explicit expression for $\psi_\Gamma$ is given by the following proposition. 

An \emph{ordered partition} of $\{1,\dots,m\}$ into $N$ parts is a composition  $p_1+p_2+\dots + p_N=m$, $(p_1,\dots,p_N)\in \Z_{\geq 0}^N$, of $m$ into $N$ parts, together with an $N$-tuple 
\be \bm n = (n^1_1, \ldots, n^1_{p_1}; n^2_1,\dots, n^2_{p_2}; \dots ; n^N_1, \ldots, n^N_{p_N})\nn\ee 
whose elements are a permutation of $(1,2,\dots,m)$.  Let $P_{m, N}$ be the set of all such ordered partitions. We shall often say $\bm n \in P_{m,N}$, leaving the composition $(p_1,\dots,p_N)$ implicit. 

Given $\lambda\in \h^*$, let $\mathsf{v}_{\lambda} \in M_{\lambda}$ be a highest weight vector normalized by $\mathsf{v}_{\lambda}(\wac)=1$. Here we regard $M_{\lambda}$ as the contragredient dual of $M^*_{\lambda}$, cf. \S\ref{sec:verm}, and note that $\wac\in\widetilde W_{\lambda\otimes t^{-1}+\mc O(t^0)}\cong_\g M^*_{\lambda}$ is a highest weight vector.

As a convenient shorthand, we write
\be \check\sigma(x) := \omega\sigma(x),\quad\text{for}\quad x\in \g.\nn\ee 

\begin{prop}\label{psiprop} The cyclotomic weight function $\psi_\Gamma$ is an element of $M_{(\bm\lambda)}=\bigotimes_{i=1}^N M_{\lambda_i}$ and is given explicitly by
\begin{align} \psi_\Gamma = 
 (-1)^m\!\!\!\!\!\! \sum_{\substack{\bm n\in P_{m,N}\\  (k_1,\dots,k_m) \in \Z_T^m}}
  \bigotimes_{i=1}^N \frac{ \check\sigma^{k_{n^i_1}}(F_{c(n^i_{1})})\check\sigma^{k_{n^i_2}}(F_{c(n^i_2)})\dots \check\sigma^{k_{n^i_{p_{i}-1}}}(F_{c(n^i_{p_{i}-1})}) \check\sigma^{k_{n^i_{p_i}}}(F_{c(n^i_{p_i})}) \mathsf{v}_{\lambda_i}}
   {\Big(\omega^{k_{n^i_1}} w_{n^i_1} - \omega^{k_{n^i_2}} w_{n^i_2}\Big)\dots
    \Big(\omega^{k_{n^i_{p_{i}-1}}}w_{n^i_{p_{i}-1}} - \omega^{k_{n^i_{p_i}}} w_{n^i_{p_i}}\Big)
        \Big(\omega^{k_{n^i_{p_i}}}w_{n^i_{p_i}} -   z_i\Big)  }.
\nn\end{align}
\end{prop}
\begin{proof}
Let $v\in \bigotimes_{i=1}^N\widetilde W_{\chi_i}$.
For each $s$, $1\leq s\leq m$, the map $\tau_\Gamma$ is invariant under, in particular, the rational function
\be \sum_{k\in \Z_T}\frac{\sigma^kG_{c(s)}}{\omega^{-k}t-w_s}\in (\nG)_{\bm x}^\Gamma.\ee
(See the comment following Corollary \ref{ginvcor}.) Consequently we have
\begin{align} &\tau_\Gamma(v,G_{c(1)}[-1]\wac,\dots,G_{c(s-1)}[-1]\wac,G_{c(s)}[-1]\wac,\wac,\dots,\wac)\nn\\
 &= \sum_{i=1}^N\sum_{k\in \Z_T}\frac{\tau_\Gamma (\sigma^k(G_{c(s)})[0]^{(i)} v , G_{c(1)}[-1]\wac,\dots,G_{c(s-1)}[-1]\wac,\wac,\wac,\dots,\wac)}{w_s-\omega^{-k}z_i} \nn\\
& + \sum_{j=1}^{s-1} \sum_{k\in \Z_T} \frac{1}{w_s-\omega^{-k}w_j} 
  \tau_\Gamma(v,G_{c(1)}[-1]\wac,\dots,G_{c(j-1)}[-1]\wac,\nn\\
&\qquad\qquad\qquad\qquad\qquad\qquad\qquad  [\sigma^k(G_{c(s)}),G_{c(j)}][-1]\wac,\nn\\
&\qquad\qquad\qquad\qquad\qquad\qquad\qquad\qquad G_{c(j+1)}[-1]\wac,\dots,G_{c(s-1)}[-1]\wac,\wac,\wac,\dots,\wac)\nn\end{align}
using $\sigma^k(G_{c(s)})[0]\wac = 0$ and $\sigma^k(G_{c(s)})[0] G_{c(j)}[-1]\wac = [ \sigma^k(G_{c(s)}),G_{c(j)}][-1]\wac$. We are then effectively in the setting of Corollary \ref{SVcor} of Appendix \ref{app:SV}. Applying this corollary gives 
\be \tau_\Gamma(v, G_{c(1)}[-1] \wac, \dots, G_{c(m)}[-1]\wac) 
  = \tau_\Gamma( \widetilde v,\wac,\dots,\wac)\ee
where
\begin{align} \widetilde v := 
 \!\!\!\!\!\! \sum_{\substack{\bm n\in P_{m,N}\\  (k_1,\dots,k_m) \in \Z_T^m}}
  \bigotimes_{i=1}^N \frac{ \check\sigma^{k_{n^i_1}}(G_{c(n^i_{1})})[0]\check\sigma^{k_{n^i_2}}(G_{c(n^i_2)})[0]\dots \check\sigma^{k_{n^i_{p_{i}-1}}}(G_{c(n^i_{p_{i}-1})})[0] \check\sigma^{k_{n^i_{p_i}}}(G_{c(n^i_{p_i})})[0] v_{\lambda_i}}
   {\Big(\omega^{k_{n^i_1}} w_{n^i_1} - \omega^{k_{n^i_2}} w_{n^i_2}\Big)\dots
    \Big(\omega^{k_{n^i_{p_{i}-1}}}w_{n^i_{p_{i}-1}} - \omega^{k_{n^i_{p_i}}} w_{n^i_{p_i}}\Big)
        \Big(\omega^{k_{n^i_{p_i}}}w_{n^i_{p_i}} -   z_i\Big)  }.
\nn\end{align}
Now note that 
\be \tau_{\Gamma}(a^*_\alpha[0]^{(i)}w,\wac,\dots,\wac) = 0, \label{asinv}\ee 
for all $i$, $1\leq i\leq N$, and all $\alpha\in \Delta^+$ and $w\in \bigotimes_{i=1}^N\widetilde W_{\chi_i}$. This follows from the invariance of $\tau_\Gamma$ under the $\Gamma$-equivariant rational function
\be \sum_{k\in \Z_T} \frac{\lsigma^k E^*_\alpha}{t-\omega^kx_i}\in \Heis_{\bm x}^\Gamma,\ee
the leading term of whose expansion at $x_i$ is $E^*_{\alpha}\otimes (t-x_i)^{-1} = a^*_{\alpha}[0]^{(i)}$, cf. \eqref{genid} and \eqref{Hxdef}. (All terms in the expansions at other points $x_j$, $j\neq i$, and the subleading terms with $j=i$, are of the form $\sigma^kE^*_\alpha\otimes (t-x_j)^p= \sigma^ka^*_\alpha[p+1]^{(j)}$ with $p\in \Z_{\geq 0}$ and therefore annihilate $w$ and $\wac$.)
Consequently for all $w\in \bigotimes_{i=1}^N\widetilde W_{\chi_i}\cong_\g \bigotimes_{i=1}^N M^*_{\lambda_i}$,
\be \tau_\Gamma(w,\wac,\dots,\wac) = \mathsf{v}_{\bm\lambda}(w) \ee
where $\mathsf{v}_{\bm\lambda}:=\bigotimes_{i=1}^N \mathsf{v}_{\lambda_i}\in M_{\bm \lambda}$. In view of \eqref{PRrel}, and noting that $\mathsf{v}_{\bm\lambda}(x^{(i)}_\beta w) = 0$ for all $\beta\in \Delta^+$, we have
\be \mathsf{v}_{\bm\lambda}(G_\alpha^{(i)} w) = - \mathsf{v}_{\bm\lambda}(E_\alpha^{(i)} w) = - (F_{\alpha}^{(i)} \on \mathsf{v}_{\bm\lambda})(w) \ee
for all $w\in \bigotimes_{i=1}^N M^*_{\lambda_i}$; here the second equality is by \eqref{iotadef}. The result follows.
\end{proof}

\subsection{Eigenvectors of the cyclotomic Gaudin Hamiltonians}\label{sec:evecs}
Recall from \S\ref{wc} the induced representation $\pi_0$ of the commutative Lie algebra $\h\otimes\C((t))$. This representation has the structure of a commutative algebra, $\pi_0\simeq \C[ b_i[-n]]_{i\in I;\,n \in \Z_{\geq 1}}$, and there is a derivation $\del:\pi_0\to\pi_0$ defined by $\del b_i[-n]= n \, b_i[-n-1]$. 
We also have the commutative algebra $\C^\8_{\Gamma\bm x}(u)$ of functions in a variable $u$ (which we may treat as formal) that vanish at infinity and have poles at most at the points $\omega^kx_i$, $k\in \Z_T$, $1\leq i\leq p=N+m$. The derivative $\del_u=\del/\del u$ acts as a derivation on  $\C^\8_{\Gamma\bm x}(u)$ and there is then a homomorphism of differential algebras
\begin{equation} r_\Gamma: (\pi_0,\del) \to \left(\C^\8_{\Gamma\bm x}(u),\del_u\right)\end{equation}
given by, cf. \eqref{lambdadef},
\begin{equation} 
r_{\Gamma}(b_s[-1]) := \lambda(u)(H_{s})  = \frac{\langle \lambda_0,\alpha^\vee_s\rangle}{u} + \sum_{k\in \Z_T} \sum_{i=1}^N \frac{\langle \lsigma^k\lambda_{i},\alpha^\vee_s\rangle}{u-\omega^kz_i} 
- \sum_{k\in \Z_T} \sum_{j=1}^m \frac{\langle\lsigma^k\alpha_{c(j)},\alpha^\vee_s\rangle}{u-\omega^kw_j}.
\end{equation}
That is, explicitly,
\begin{equation} r_\Gamma(b_{s_1}[-n_1]\dots b_{s_M}[-n_M]) = \prod_{k=1}^M \frac{1}{(n_k-1)!}\left(\frac{\del}{\del u}\right)^{n_k-1}  \lambda(u)(H_{{s_k}}). \label{eval}
\end{equation}
\begin{thm}\label{evalthm}
If the $w_j$, $1\leq j\leq m$ satisfy the cyclotomic Bethe equations \eqref{tbe} then for all $Z\in \mf Z(\gh)$, the cyclotomic weight function $\psi_\Gamma=\psi_\Gamma(c(1),\dots,c(m);w_1,\dots,w_m)\in M_{\bm \lambda}$ is an eigenvector of $\iota(Z(u))$ with eigenvalue $(r_\Gamma\circ\rho)(Z(u))$.
\end{thm}
\begin{proof}
As in \S\ref{sec:twistedswapping}, we introduce an additional non-zero point $u$ whose $\Gamma$-orbit is disjoint from those of the $x_i$, $1\leq i\leq p$. To the point $u$ we assign a copy of the $\Heis(\g)\oplus \h\otimes\C((t-u))$-module $\WW_0=\Mh\otimes\pi_0$, cf. \S\ref{wc}. There is a unique $\Heis_{\bm x,u}^\Gamma\oplus\h_{\bm x,u}^\Gamma$-invariant linear functional 
\be \tau:\bigotimes_{i=1}^N W_{\chi_i} \otimes \bigotimes_{j=1}^m W_{\mu_j}\otimes \WW_0 \otimes \C v_{\chi_0}\to \C\ee 
normalized such that $\tau(\wac^{\otimes(N+m+1)})=1$, cf. \eqref{taugdef}.

Recall the embedding $\rho:\Vcrit\to \WW_0$ from \eqref{rhodef} and consider 
\be \tau( v, G_{c(1)}[-1]\wac,\dots,G_{c(m)}[-1]\wac, \rho(Z)).\ee 
This quantity can be evaluated in two ways. 

First, by Corollary \ref{ginvcor}, the functional $\tau$ is $\g_{\bm x,u}^{\Gamma}$-invariant. We can therefore ``swap'' from $u$ in the same manner as in \S\ref{sec:twistedswapping}. Now if the $w_j$, $1\leq j\leq m$, satisfy the cyclotomic Bethe equations \eqref{tbe} then for each $j$, $1\leq j\leq m$, the vector $G_{c(j)}[-1]\wac\in W_{\mu_j}$ is singular for the action of $\gh_{(w_j)}$. Therefore
\begin{align} 
\tau( v, G_{c(1)}[-1]\wac,\dots,G_{c(m)}[-1]\wac, \rho(Z) ) 
&= \tau( Z(u)\on v, G_{c(1)}[-1]\wac,\dots,G_{c(m)}[-1]\wac,\wac) \nn\\
&= \tau_\Gamma( Z(u)\on v, G_{c(1)}[-1]\wac,\dots,G_{c(m)}[-1]\wac) \nn\\
&= \psi_\Gamma( Z(u)\on v ) = \left(\iota(Z(u))\on \psi_{\Gamma}\right)(v).
\end{align}
%where $Z(u)\in U(\g)^{\otimes N}$ is the image of $Z\in \mf Z(\gh)$ under the map of \S\ref{GGA}.

But at the same time, it is known \cite{FFR} that $\rho(\mf Z(\gh)) \subset \pi_0$. That means $\rho(Z)$ can be written as a linear combination of terms of the form $b_{s_1}[-n_1]\dots b_{s_M}[-n_M]$, $s_1,\dots,s_M\in I$, $n_1,\dots,n_M\in \Z_{\geq 1}$. By using the $\h_{\bm x,u}^\Gamma$ invariance of $\tau$ and ``swapping'' using the functions of the form
\be \frac{1}{(n-1)!}\left(\frac{\del}{\del u}\right)^{n-1} \sum_{r\in \Z_T} \frac{\sigma^rb_s}{\omega^{-r} t- u} \in \h_{\bm x,u}^\Gamma\ee
one has, by definition of the Wakimoto modules $W_{\chi_i}$ and $W_{\mu_j}$, \S\ref{sec:wakm}, that
\be \tau(v, G_{c(1)}[-1]\wac,\dots,G_{c(m)}[-1]\wac, b_s[-n] w) 
= r_\Gamma(b_s[-n]) \tau(v, G_{c(1)}[-1]\wac,\dots,G_{c(m)}[-1]\wac, w). \label{le}\ee
The result follows.
\end{proof}

\section{Examples and special cases}\label{sec:examples}

\subsection{Eigenvalues of the cyclotomic quadratic Hamiltonians}
Recall the quadratic Casimir $\mc C=\half I^aI_a\in Z(U(\g))$ and corresponding singular vector $S\in \mf Z(\gh)$ from \eqref{Sdef}.
We must express these in terms of the Cartan-Weyl basis of \eqref{CartanWeyl}.
The dual basis is  
\begin{equation} 
E^\alpha = \frac{F_\alpha}{\langle E_\alpha,F_\alpha\rangle} = \frac{\langle \alpha, \alpha \rangle} 2 F_\alpha,\quad F^\alpha = \frac{\langle \alpha,\alpha\rangle} 2 E_\alpha,\quad
\text{and}\quad 
 H^i = \sum_{j\in I} \langle \omega_i,\omega_j\rangle H_{j},
\end{equation}
where $\{\omega_i\}_{i\in I}\subset \h^*\cong \h$ are the fundamental weights of $\g$. 
To see these, note that
$\langle H_\alpha,H\rangle = \langle [E_\alpha,F_\alpha], H\rangle = \alpha(H) \langle F_\alpha,E_\alpha\rangle$, by invariance of $\langle\cdot,\cdot\rangle$, so that $\langle E_\alpha,F_\alpha \rangle = \langle H_\alpha,H_\beta\rangle/\alpha(H_\beta)= \langle \alpha^\vee,\beta^\vee\rangle/\langle \alpha,\beta^\vee \rangle = 2/\langle \alpha,\alpha\rangle$. And $\langle\omega_i,\omega_j\rangle$ is the inverse matrix to $\langle H_i, H_j\rangle = \frac 1 {h^\vee} \tr_\n \ad_{H_i} \ad_{H_j}$: indeed, by definition $\delta_{ij} = \langle \omega_j ,\alpha^\vee_i\rangle$; hence $\alpha^\vee_i = \sum_{k\in I}\omega_k \langle \alpha^\vee_k, \alpha^\vee_i\rangle$ and so $\delta_{ij} = \sum_{k\in I} \langle \omega_j,\omega_k\rangle \langle \alpha^\vee_k,\alpha^\vee_i\rangle$. 
\begin{comment}
Thus the quadratic Casimir in $\mc C\in Z(U(\g))$ is 
\begin{align} 
 \mc C:= \half I^aI_a &= \half\sum_{s,t\in I} \langle\omega_s,\omega_t\rangle H_s H_t 
+  \sum_{\alpha\in \Delta^+}\frac{\langle\alpha,\alpha\rangle} 4 (E_\alpha F_\alpha+F_\alpha E_\alpha) \\
  &= \half \sum_{s,t\in I} \langle\omega_s,\omega_t \rangle H_s H_t 
+ \sum_{\alpha\in \Delta^+} \frac{\langle\alpha,\alpha\rangle} 4  H_\alpha 
+ \sum_{\alpha\in \Delta^+} \frac{\langle\alpha,\alpha\rangle}{2} F_\alpha E_\alpha.
\end{align}
and its image under the usual Harish-Chandra homomorphism $\mathrm{HC}:Z(U(\g))\to U(\h)$ is therefore 
\be \mathrm{HC}(\mc C):= \half \sum_{s,t\in I} \langle\omega_s,\omega_t\rangle H_s H_t +  \sum_{\alpha\in \Delta^+} \frac{\langle\alpha,\alpha\rangle} 4 H_\alpha\label{HCC}.\ee
\end{comment}

Recall the singular vector $S\in \mf Z(\gh)$ from \eqref{Sdef}. It is known that (see, e.g., \cite[\S 8.1.4]{Fre07})
\begin{equation*}
\rho(S) = \half \sum_{s,t\in I} \langle\omega_s,\omega_t\rangle b_s[-1] b_t[-1] \wac
- \sum_{\alpha\in \Delta^+}\frac{\langle\alpha,\alpha\rangle} 4  b_\alpha[-2]\wac.
\end{equation*}
Hence, with $r_\Gamma$ as in \eqref{eval}, we have
\begin{align} (r_\Gamma\circ\rho)(S) &= 
\half \sum_{s,t\in I} \langle\omega_s,\omega_t\rangle \langle \lambda(u),\alpha^\vee_s\rangle\langle \lambda(u),\alpha^\vee_t\rangle - \sum_{\alpha\in \Delta^+} \frac{\langle\alpha,\alpha\rangle} 4  \lambda'(u)(H_\alpha) \notag\\
&= \half \langle \lambda(u),\lambda(u) \rangle - \langle \lambda'(u) , \varrho \rangle \label{rS}
\end{align}
where $\varrho:= \half\sum_{\alpha\in \Delta^+}\alpha$. 

In view of \eqref{lambdadef} we therefore find the following expression for the eigenvalue $E_i$ of the cyclotomic quadratic Gaudin Hamiltonians $\mathcal{H}_i$ defined in \eqref{tildeHdef}:
\be \left(\mathcal{H}_i-E_i \right) \psi_\Gamma(c(1),\dots,c(m);w_1,\dots,w_m) = 0,\ee
\begin{equation} \label{evalue Hi}
E_i := \sum_{\substack{j=1\\j\neq i}}^N \sum_{s=0}^{T-1} \frac{\langle \lambda_i,\lsigma^s \lambda_j\rangle}{z_i-\omega^sz_j} 
- \sum_{j=1}^m \sum_{s=0}^{T-1} \frac{\langle \lambda_i,\lsigma^s \alpha_{c(j)} \rangle}{z_i-\omega^sw_j}
+ \frac{1}{z_i} \left( \langle \lambda_i, \lambda_0\rangle
+ \frac{1}{2} \sum_{s=1}^{T-1} \langle \lambda_i,\lsigma^s \lambda_i\rangle \right).
\end{equation}
The second term in brackets originates from an expression of the form $\sum_{s=1}^{T-1} \frac{\langle \lambda_i,\lsigma^s \lambda_i\rangle}{1-\omega^s}$ which can be rewritten as $\sum_{s=1}^{T-1} \frac{\langle \lambda_i,\lsigma^s \lambda_i\rangle}{1-\omega^{-s}}$ and hence simplifies to $\sum_{s=1}^{T-1} \frac{\langle \lambda_i,\lsigma^s \lambda_i\rangle}{2}$.

We also find certain identities, by comparing the double pole terms in $S(u)$.  

First, at the double pole at $u=\omega^kz_i$ we find the correct value $\Delta(\lambda_i)$ of the quadratic Casimir $\mc C = \half I^aI_a$ on the Verma module $M_{\lambda_{i}}$ -- which one recognises from, for example, the Freudenthal multiplicity formula --  
\begin{equation} \label{c lambda i}
\Delta(\lambda_i) =   \half \langle \lambda_i,\lambda_i \rangle + \langle \lambda_i, \varrho \rangle,
\end{equation}
(note that $\Delta(\lambda_i) = \half \langle \lsigma^k\lambda_i,\lsigma^k\lambda_i \rangle + \langle \lsigma^k\lambda_i, \varrho \rangle$
for all $k\in \Z_T$). 

More non-trivially, from the double pole at $u=0$, we find the identity
\be 
-\frac{h^{\vee}}{2} \sum_{r=1}^{T-1} \frac{\omega^r \langle \sigma^r I^a, I_a \rangle}{(\omega^r - 1)^2} = 
\half \langle\lambda_0,\lambda_0\rangle + \langle \lambda_0, \varrho\rangle.
\ee
Here the right-hand side is the value $\Delta(\lambda_0)$ of $\mc C$ on the Verma module $M_{\lambda_0}$.

\subsection{Inner automorphisms (and in particular $\sigma=\id$)}
Suppose the automorphism $\sigma:\g\to\g$ is inner. Then it fixes, pointwise, our choice of Cartan subalgebra, cf.  \eqref{sigmaE}. For each positive root $\alpha\in \Delta^+$ there is a unique number $\chi_\alpha\in\{0,1,\dots,T-1\}$ such that 
\be\label{chidef}\tau_\alpha = \omega^{\chi_{\alpha}}.\ee
(Note, in particular, that $\sigma=\id$ corresponds to $\chi_\alpha=0$ for all $\alpha\in \Delta^+$.)

We have the identities
\be \sum_{r = 1}^{T-1} \frac 1 {1-\omega^r} = \frac{T-1}{2}, \qquad
 \sum_{r = 0}^{T-1} \frac 1{w-\omega^r z} = \frac{Tw^{T-1}}{w^T-z^T}.\ee
In view of these, 
the Bethe equations \eqref{tbe} then become, with $\tilde z_i := z_i^T$ and $\tilde w_i= w_i^T$,
\be 0= \sum_{i=1}^N\frac{\langle \alpha_{c(j)},\lambda_i\rangle}{\tilde w_j-\tilde z_i} - 
\sum_{\substack{k=1\\k\neq j}}^m \frac{\langle \alpha_{c(j)},\alpha_{c(k)}\rangle}{\tilde w_j-\tilde w_k}
+
\frac{1}{T} \frac1 {\tilde w_j} \sum_{r = 1}^{T-1} 
\frac{1}{\omega^r-1} \left(\langle \alpha_{c(j)},\alpha_{c(j)}\rangle 
- \sum_{\alpha\in \Delta^+} \omega^{- \chi_\alpha r} \langle  \alpha_{c(j)},\alpha\rangle \right)
\ee
for $j=1,\dots, m$.
Recalling that $\varrho = \half \sum_{\alpha\in \Delta^+}\alpha = \sum_{i\in I} \omega_i$, one has 
\be \sum_{\alpha\in \Delta^+}\langle \alpha,\alpha_{c(j)}\rangle = 2\sum_{k\in I} \langle \omega_k, \alpha_{c(j)}\rangle = 
\langle \alpha_{c(j)},\alpha_{c(j)}\rangle \sum_{k\in I} \langle \omega_k, \alpha^\vee_{c(j)}\rangle =  \langle \alpha_{c(j)},\alpha_{c(j)}\rangle.\nn\ee
Therefore 
if in fact $\sigma=\id$ then the term in $1/\tilde w_j$ actually vanishes and the Bethe equations become
\begin{equation} \label{untwisted BAE}
0= \sum_{i=1}^N\frac{\langle \alpha_{c(j)},\lambda_i\rangle}{\tilde w_j-\tilde z_i} 
-
\sum_{\substack{k=1\\k\neq j}}^m \frac{\langle \alpha_{c(j)},\alpha_{c(k)}\rangle}{\tilde w_j-\tilde w_k},
\end{equation}
for $j=1,\dots,m$. These are nothing but the Bethe equations for the \emph{usual} Gaudin model with Verma modules $M_{\lambda_i}$ assigned to the points $\tilde z_i$, with Bethe roots $\tilde w_j$. 
More generally, the Bethe equations when $\sigma$ is purely inner are
\be 0= \sum_{i=1}^N\frac{\langle \alpha_{c(j)},\lambda_i\rangle}{\tilde w_j-\tilde z_i} - 
\sum_{\substack{k=1\\k\neq j}}^m \frac{\langle \alpha_{c(j)},\alpha_{c(k)}\rangle}{\tilde w_j-\tilde w_k}
-
\frac{1}{\tilde w_j}  \sum_{\alpha\in\Delta^+} \frac{[\chi_\alpha]}{T} 
  \langle  \alpha_{c(j)},\alpha\rangle,
\ee
where $[k]:= k \mod T\in \{0,1,\dots,T-1\}$ for any $k \in \Z$ and we note that $\sum_{r = 1}^{T-1} \frac{\omega^{- k r} - 1}{\omega^r-1} = [k]$. These again can be interpreted as the Bethe equations for an ordinary Gaudin model, but now with an additional Verma module assigned to the origin.

Let us turn to the cyclotomic weight function $\psi_\Gamma$ of Proposition \ref{psiprop}. When $\sigma$ is inner, $\psi_\Gamma$ is a $\g$-weight vector. Indeed,  
\begin{align} \psi_\Gamma = 
 (-1)^m\!\!\!\!\!\! \sum_{\substack{\bm n\in P_{m,N}\\  (k_1,\dots,k_m) \in \Z_T^m}}
  \bigotimes_{i=1}^N 
\frac{ \omega^{k_{n^i_1}\big(1-\chi_{\alpha_{c(n^i_1)}}\big) + k_{n^i_2}\big(1-\chi_{\alpha_{c(n^i_2)}}\big)+  \dots  + k_{n^i_{p_i}}\big(1-\chi_{\alpha_{c(n^i_{p_i})}}\big)}  
   F_{c(n^i_1)}F_{c(n^i_2)}\dots F_{c(n^i_{p_i})} \mathsf v_{\lambda_i}}
   {\Big(\omega^{k_{n^i_1}} w_{n^i_1} - \omega^{k_{n^i_2}} w_{n^i_2}\Big)\dots
    \Big(\omega^{k_{n^i_{p_{i}-1}}}w_{n^i_{p_{i}-1}} - \omega^{k_{n^i_{p_i}}} w_{n^i_{p_i}}\Big)
        \Big(\omega^{k_{n^i_{p_i}}}w_{n^i_{p_i}} -   z_i\Big)  }.
\nn\end{align}
and by repeatedly applying the identity,
\begin{equation*}
\sum_{k\in \Z_T} \frac {\omega^{-rk}}{\omega^ku- v} = \frac{Tv^{T-1-[r]}u^{[r]}}{u^T-v^T},
\end{equation*}
one has 
\begin{align} \psi_\Gamma& = 
 (-1)^mT^m   \sum_{\bm n\in P_{m,N}} \bigotimes_{i=1}^N 
   % \nn\\&\qquad\qquad\qquad\times 
\frac{f\big(w_{n^i_1},\dots,w_{n^i_{p_i}},z_i; \chi_{\alpha_{c(n^i_1)}}, \dots, \chi_{\alpha_{c(n^i_{p_{i}})}}\big)}
 {\big(\tilde w_{n^i_1} -\tilde w_{n^i_2}\big)\dots
    \big(\tilde w_{n^i_{p_{i}-1}} - \tilde w_{n^i_{p_i}}\big)
        \big(\tilde w_{n^i_{p_i}} -   \tilde z_i\big)  }F_{c(n^i_1)}F_{c(n^i_2)}\dots F_{c(n^i_{p_i})} \mathsf v_{\lambda_i}\nn\end{align}
where
\begin{equation*}
f(w_1,\dots,w_p,z;\chi_1,\dots,\chi_p) := w_1^{[\chi_1-1]} \left(\prod_{s=2}^p w_s^{T-1-[\chi_1+\dots+\chi_{s-1}-1] + [\chi_1+\dots+\chi_s-1]}\right) z^{T-1-[\chi_1+\dots+\chi_p-1]}.
\end{equation*}

In particular if $\sigma = \id$ then
\begin{align} \psi_\Gamma& = 
 (-1)^mT^m (w_1\dots w_m)^{T-1}  \sum_{\bm n\in P_{m,N}} \bigotimes_{i=1}^N 
   % \nn\\&\qquad\qquad\qquad\times 
\frac{F_{c(n^i_1)}F_{c(n^i_2)}\dots F_{c(n^i_{p_i})} \mathsf v_{\lambda_i}}
 {\big(\tilde w_{n^i_1} -\tilde w_{n^i_2}\big)\dots
    \big(\tilde w_{n^i_{p_{i}-1}} - \tilde w_{n^i_{p_i}}\big)
        \big(\tilde w_{n^i_{p_i}} -   \tilde z_i\big)  }\nn\end{align}
which is the usual Schechtman-Varchenko expression for the weight function in the variables $\tilde w_j$, $1\leq j\leq m$ and $\tilde z_i$, $1\leq i\leq N$, up to the constant overall factor $T^m(w_1\dots w_m)^{T-1}$. 

For the eigenvalues \eqref{evalue Hi}, one finds by use of the same identities that
\begin{equation*}
E_i = T z_i^{T-1} \left( \sum_{\substack{j=1\\j\neq i}}^N \frac{\langle \lambda_i, \lambda_j\rangle}{\tilde{z}_i - \tilde{z}_j}
- \sum_{j=1}^m \frac{\langle \lambda_i, \alpha_{c(j)} \rangle}{\tilde{z}_i - \tilde{w}_j} \right)
+ \frac{1}{z_i} \left( \sum_{\alpha \in \Delta^+} \left( \frac{T - 1}{2} - [\chi_{\alpha}] \right) \langle \alpha, \lambda_i\rangle + \frac{T - 1}{2} \langle \lambda_i, \lambda_i\rangle \right).
\end{equation*}
In particular, when $\sigma = \text{id}$, i.e. $[\chi_{\alpha}] = 0$ for all $\alpha \in \Delta^+$, this reduces to
\begin{equation*}
E_i = T z_i^{T-1} \left( \sum_{\substack{j=1\\j\neq i}}^N \frac{\langle \lambda_i, \lambda_j\rangle}{\tilde{z}_i - \tilde{z}_j}
- \sum_{j=1}^m \frac{\langle \lambda_i, \alpha_{c(j)} \rangle}{\tilde{z}_i - \tilde{w}_j} \right)
+ \frac{T - 1}{z_i} c(\lambda_i),
\end{equation*}
where $c(\lambda_i)$ is defined in \eqref{c lambda i}. The expression in brackets is precisely the eigenvalue of the Gaudin Hamiltonian of the usual Gaudin model with Verma modules $M_{\lambda_i}$ assigned to the point $\tilde{z}_i$, where the Bethe roots $\tilde{w}_i$ satisfy the usual Bethe equations \eqref{untwisted BAE}. This corresponds to the fact that when $\sigma = \text{id}$, the expression \eqref{tildeHdef} can be resumed to yield
\begin{equation*}
\mathcal{H}_i = T z_i^{T-1} \sum_{\substack{j=1\\j \neq i}}^N \frac{I^{a (i)} I_a^{(j)}}{\tilde{z}_i - \tilde{z}_j} + \frac{T-1}{z_i} \frac{1}{2} I^{a (i)} I_a^{(i)},
\end{equation*}
which is nothing but a linear combination of the Hamiltonian for the usual Gaudin model and the quadratic Casimir $\mathcal{C}^{(i)}$.

\subsection{An example of a diagram automorphism.}
The opposite extreme to the situation of the preceding subsection is when $\tau_{\alpha_i}=1$ for all simple roots $\alpha_i$, $i \in I$ in \eqref{sigmaE} and $\sigma\neq \id$ is a non-trivial diagram automorphism. Suppose that we are in this situation, and moreover that $\sigma$ has order 2 (which follows necessarily unless we are in type $D_4$). 
Then the Bethe equations are
\begin{align}
0&= \sum_{i=1}^N\frac{\langle \alpha_{c(j)},\lambda_i\rangle}{w_j-z_i}
 + \sum_{i=1}^N\frac{\langle \alpha_{\sigma(c(j))},\lambda_{i}\rangle}{w_j+ z_i} 
-  \sum_{\substack{k=1\\k\neq j}}^m \frac{\langle \alpha_{c(j)},\alpha_{c(k)}\rangle}{w_j-w_k}
-  \sum_{\substack{k=1\\k\neq j}}^m \frac{\langle \alpha_{\sigma(c(j))},\alpha_{c(k)}\rangle}{w_j+w_k} \nn\\
&\qquad +
\frac{1}{w_j}
\left( -\half\langle \alpha_{c(j)},\alpha_{\sigma(c(j))}\rangle 
+\langle \alpha_{c(j)}, \lambda_0 \rangle \right),
\nn\end{align}
where now $\lambda_0$ reduces to
\begin{equation*}
\lambda_0 = \half \sum_{\substack{\alpha\in \Delta^+\\\sigma(\alpha)=\alpha}} \tau_{\alpha}^{-1} \, \alpha.
\end{equation*}

We shall consider one simple example of this type. Let $\g$ be $\mf a_2$ and $\sigma$ its unique diagram automorphism. Suppose there are $N=2$ sites, to both of which we assign the 3-dimensional defining representation $L_{\omega_1}$ of $\mf {sl}_3\cong\mf a_2$. Recall that $L_{\omega_1}$ is the irreducible quotient of the Verma module $M_{\omega_1}$. The Hamiltonian $\mathcal{H}_1$ defined in \eqref{tildeHdef} is then a $9\times 9$ matrix whose eigenvalues and eigenvectors can be found by direct computation. Its  eigenvalues are
\be
\frac{z_2^2+z_1z_2+2z_1^2}{3z_1^3-3z_1z_2^2},\quad 
\frac{z_2^2-5z_1z_2-4z_1^2}{3z_1^3-3z_1z_2^2},\quad
\frac{z_2^2+10z_1z_2-7z_1^2}{3z_1^3-3z_1z_2^2},
\label{direct evals}\ee
with multiplicities respectively $5$, $3$, $1$. These multiplicities agree with the decomposition of $L_{\omega_1}\otimes L_{\omega_1} 
%\cong_{\mf{sl}_3} L_{2\omega_1}\oplus L_{\omega_2}
$ into irreducible representations of the $\mf {sl}_2$ subalgebra stabilized by $\sigma$, which is linearly generated by
\be E:= E_{12} + E_{23},\quad  F:= E_{21}+ E_{32}, \quad H:=[E,F]\label{sl20}\ee 
(the notation $E_{12}$ etc. is that of Example \ref{sl3ex}) and whose positive root is $\alpha:=\half(\alpha_1+\alpha_2)$: 

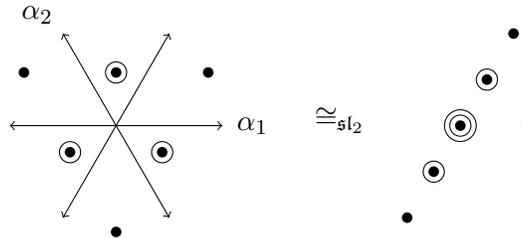
\be\nn\begin{tikzpicture}[baseline=0,scale=1] 
\draw (0:2) node[right=-15pt]{$\alpha_1$};\draw (120:1.414) node[above left]{$\alpha_2$};
\foreach\x in {0,1,...,5} \draw[->] (0,0) -- (60*\x:1.414);
\foreach\x in {0,1,2} \fill (\x*120+30:2/1.414) circle (2pt);
\foreach\x in {0,1,2}{\fill (\x*120+90:1/1.414) circle (2pt); \draw (\x*120+90:1/1.414) circle (4pt);}
\end{tikzpicture}
\quad\cong_{\mf{sl}_2} \quad
\begin{tikzpicture}[baseline=0,scale=1] 
\foreach\x in {-2,-1,0,1,2} \fill (60:\x*1.414/2) circle (2pt);
\foreach\x in {-1,0,1} \draw (60:\x*1.414/2) circle (4pt);
\draw (60:0) circle (6pt);;
\end{tikzpicture}.
\ee

The $5$ dimensional representation of $\mf{sl}_2$ is the irreducible representation through the vacuum vector $v\otimes v$. 

With $m=1$ and $c(1)=1$ (i.e. one lowering operation, in the direction of the simple root $\alpha_1$) the Bethe equation \eqref{tbe} and its unique solution are
\be 0=\frac 1{w_1-z_1} + \frac 1{w_1-z_2}\implies w_1=\frac{z_1+z_2} 2.\nn\ee
Note here that $-\half\langle\alpha_1,\alpha_{\sigma(1)}\rangle=-\half(-1)=\half$ is cancelled by $\langle\alpha_1,\lambda_0\rangle = -\half\langle\alpha_1,\alpha_1+\alpha_2\rangle = -\frac{2-1}{2} = -\half$ so that in total there is no pole at $w_1$.
With $m=1$ and $c(1)=2$ one has
\be 0=\frac 1{w_1+z_1} + \frac 1{w_1+z_2}\implies w_1=-\frac{z_1+z_2} 2,\nn\ee
which is the same equation with $w_1$ replaced by its twist-image $\omega w_1=-w_1$.
The cyclotomic weight functions for these solutions are the actually proportional, as one expects since there is only one $\mf{sl}_2$ singular vector at this $\mf{sl}_2$ weight\footnote{By analogy with the familiar untwisted case, one expects that the cyclotomic weight function evaluated at any solution to the cyclotomic Bethe equations yields a singular vector for the subalgebra of $\g$ stablized by $\sigma$. We have not proved this here.}. Indeed, when $c(1)=1$ one has
\begin{align} \psi_\Gamma &= \frac{E_{21} v \otimes v}{w_1-z_1} + \frac{-E_{32}v\otimes v}{-w_1-z_1} 
                 +\frac{v \otimes E_{21}v}{w_1-z_2} + \frac{v\otimes(-E_{32} v)}{-w_1-z_2} \nn\\
%&= \frac{E_{21} v \otimes v}{w_1-z_1} + \frac{E_{32}v\otimes v}{w_1+z_1} 
%                 +\frac{v \otimes E_{21}v}{w_1-z_2} + \frac{v\otimes E_{32} v}{w_1+z_2}\\
&= \frac{2}{z_2-z_1} \left( E_{21}v \otimes v - v\otimes E_{21}v\right) 
\nn\end{align}
on substituting the solution to the Bethe equation and noting that $E_{32}v=0$, while when $c(1)=2$,
\begin{align} \psi_\Gamma &= \frac{E_{32} v \otimes v}{w_1-z_1} + \frac{-E_{21}v\otimes v}{-w_1-z_1} 
                 +\frac{v \otimes E_{32}v}{w_1-z_2} + \frac{v\otimes(-E_{21} v)}{-w_1-z_2}\nn\\ 
%&= \frac{E_{32} v \otimes v}{w_1-z_1} + \frac{E_{21}v\otimes v}{w_1+z_1} 
%                 +\frac{v \otimes E_{32}v}{w_1-z_2} + \frac{v\otimes E_{21} v}{w_1+z_2}
&=-\frac{2}{z_2-z_1} \left( E_{21}v \otimes v - v\otimes E_{21}v\right). 
\nn\end{align}

With $m=2$, $(c(1),c(2))=(1,2)$ the cyclotomic Bethe equations \eqref{tbe} are
\begin{equation} \label{12bes}
0=\frac 1{w_1-z_1} + \frac 1{w_1-z_2} - \frac 2{w_1+w_2} - \frac {-1}{w_1-w_2}, \quad 
0=\frac 1{w_2+z_2} + \frac 1{w_2+z_1} - \frac 2{w_2+w_1} - \frac {-1}{w_2-w_1}.
\end{equation}
From these, the equations for $(c(1),c(2))=(1,1)$ are obtained by replacing $w_2$ by its twist-image $-w_2$. For $(c(1),c(2))=(2,2)$ one must send $w_1\mapsto -w_1$. For $(c(1),c(2))=(2,1)$ one must send $w_1\mapsto -w_1$ \emph{and} $w_2\mapsto -w_2$; this yields the same set of equations but with $w_1$ and $w_2$ exchanged, as it should. 
Up to one choice of branch in the square roots there is a unique solution to \eqref{12bes}, namely
\begin{equation*}
w_1 = \frac{z_1+z_2-\sqrt{(z_2-5z_1)(5z_2-z_1)}}{6},
\quad w_2 = -\frac{z_1+z_2+\sqrt{(z_2-5z_1)(5z_2-z_1)}}{6}.
\end{equation*}
The non-zero terms in the weight function $\psi_\Gamma\in L_{\omega_1}\otimes L_{\omega_1}$ are
\begin{align} \psi_\Gamma &= E_{32} E_{21} v \otimes v \left(\frac 1 {(w_2-w_1)(w_1-z_1)} 
                                                       + \frac 1 {(-w_1+w_2)(-w_2-z_1)}\right)\nn\\
             &{}            + v\otimes E_{32} E_{21} v  \left(\frac 1 {(w_2-w_1)(w_1-z_2)} 
                                                       + \frac 1 {(-w_1+w_2)(-w_2-z_2)}\right)\nn\\
             &{}            + E_{21} v\otimes E_{21} v \left( \frac{-1}{(w_1-z_1)(-w_2-z_2)} + 
                                                         \frac{-1}{(-w_2-z_1)(w_1-z_2)}\right)\nn\end{align}
and on substituting the solution above one finds eventually 
\be \psi_\Gamma =\frac 9 {(z_1+z_2)^2} \left( E_{32} E_{21} v \otimes v +  v\otimes E_{32} E_{21} v -E_{21} v\otimes E_{21} v\right).\ee
This again is singular for the $\mf{sl}_2$ of \eqref{sl20}, as expected. 
It is a straightforward check to verify that the eigenvalues \eqref{evalue Hi} computed from the Bethe ansatz agree with the result \eqref{direct evals} of the direct computation.

\appendix

\section{The $\Gamma$-equivariant Strong Residue Theorem}\label{sec:tert}
Suppose $A$ and $B$ are complex vector spaces equipped with a non-degenerate bilinear pairing $\langle\cdot,\cdot\rangle:A\times B\to \C$. 
(This covers, of course, the special case when $A=B$ and $\langle\cdot,\cdot\rangle$ is a non-degenerate bilinear form on $A$.) 

As in the main text, let $\Gamma$ be the cyclic group $\{1,\omega,\omega^2,\dots,\omega^{T-1}\}\subset \Cx$ acting on $\C$ by multiplication, and let  $\bm x=(x_1,\dots,x_p)$ be a collection of non-zero complex numbers whose $\Gamma$-orbits are disjoint.
Suppose $\Gamma$ acts on $A$ and $B$ and $\langle \omega\on a,\omega\on b\rangle = \langle a,b\rangle$ for all $a\in A, b\in B$. 
Define 
\be A_{\bm x}^{\Gamma,k} := \{ f\in A\otimes \C^\8_{\Gamma\bm x}(t): f(\omega t) = \omega^k (\omega\on f)(t)  \} \label{AGkdef}\ee
for each $k\in \Z_T$, and likewise $B_{\bm x}^{\Gamma,k}$. 
There is an injection
\be \iota: A_{\bm x}^{\Gamma,k}  \longhookrightarrow \bigoplus_{i=1}^p A\otimes \C((t-x_i));\quad
    f(t) \longmapsto (\iota_{t-x_1}f(t),\dots,\iota_{t-x_p}f(t)).\label{tin}\ee
%We have  
%\be f \in  A_{\bm x}^{\Gamma,k} \implies  f' \in A_{\bm x}^{\Gamma, (k-1\!\!\!\mod T)}\label{tsh}\ee
%\Ron (because $(\del_{s} f(s))|_{s = \omega t} = \omega^{-1} \del_t f(\omega t) = \omega^{k-1}\sigma \del_t f(t)$.)
We also allow the possibility of a pole at the point zero, which is special because it is the fixed point of the map $\C\to \C;z\mapsto \omega z$. Thus, let 
\be A_{\bm x,0}^{\Gamma,k} := \{ f\in A\otimes \C^\8_{\Gamma\bm x\cup\{0\}}(t): f(\omega t) = \omega^k (\omega\on f)(t)  \}.\ee 
Note that the image of $\iota_{t-0}$ is in $\left(A\otimes\C((t))\right)^{\Gamma,k} := \{ f\in A\otimes \C((t)): f(\omega t) = \omega^k (\omega\on f)(t) \}$. 

\begin{lem}[$\Gamma$-equivariant residue theorem]\label{terttwo}  $ $
\begin{enumerate}
\item\label{tert1} An element $(f_1,\dots,f_p)\in \bigoplus_{i=1}^p A\otimes \C((t-x_i))$ is in $\iota(A_{\bm x}^{\Gamma,k})$ if and only if $$0=\sum_{i=1}^p  \res_{t-x_i} \langle f_i, \, \iota_{t-x_i}(g)\rangle$$
for every $g\in B_{\bm x}^{\Gamma,-k-1}$.
\item\label{tert2} An element $(f_1,\dots,f_p,f_0)\in \bigoplus_{i=1}^p A\otimes \C((t-x_i))\oplus \left(A\otimes\C((t))\right)^{\Gamma,k}$ is in $(\iota,\iota_t)(A_{\bm x,0}^{\Gamma,k})$ if and only if $$0=\sum_{i=1}^p  \res_{t-x_i} \langle f_i, \, \iota_{t-x_i}(g)\rangle + \frac 1 T \res_t \langle f_0,\,\iota_t(g)\rangle$$
for every $g\in B_{\bm x,0}^{\Gamma,-k-1}$.
\end{enumerate}
\end{lem}
\begin{proof}
Consider part (\ref{tert1}). For the ``only if'' direction, let $f\in A_{\bm x}^{\Gamma,k}$ and $g\in B_{\bm x}^{\Gamma,(-k-1)}$. In particular $f\in A\otimes \C^\8_{\Gamma\bm x}(t)$  and $g\in B\otimes \C^\8_{\Gamma\bm x}(t)$ so that, just as in the proof of Lemma \ref{srt},
\be \sum_{i = 1}^p \sum_{\alpha\in \Gamma} 
              \res_{t-\alpha x_i} \langle \iota_{t-\alpha x_i}f,\iota_{t-\alpha x_i} g \rangle
= \sum_{i=1}^p \sum_{\alpha\in \Gamma} \res_{t-\alpha x_i} \langle f,g\rangle = 0 \ee
(since this is the sum over all residues).
Now by the invariance of $\langle\cdot,\cdot\rangle$, we have that $\langle f,g\rangle \in \C^\8_{\Gamma\bm x}(t)^{\Gamma,-1}$ where  
\be \C^\8_{\Gamma\bm x}(t)^{\Gamma,k} := \{ h \in \C^\8_{\Gamma \bm x}(t): h(t) = \omega^{-k} h(\omega t) \}.\label{CGammadef}\ee 
For any $h\in \C^\8_{\Gamma\bm x}(t)$, $\res_{t-x_i} h(t)  
%= \omega^{-1} \res_{\omega t - \omega x_i} h(t)
                                    = \omega^{-1} \res_{t- \omega x_i} h(\omega^{-1} t)$ .
Hence if $h\in \C^\8_{\Gamma\bm x}(t)^{\Gamma,k}$ then $\res_{t-x_i} h(t)  = \omega^{-1-k} \res_{t-\omega x_i} h(t)$. In this way one has that for all $\alpha\in \Gamma$, 
\be  \res_{t-\alpha x_i} \langle f, g \rangle
  =  \res_{t-x_i}  \langle f,g\rangle. \ee
Thus in fact $0 = T\sum_{i = 1}^p  \res_{t-x_i} \langle \iota_{t-x_i} f,\iota_{t-x_i} g \rangle$ as required.

Turning to the ``if'' direction, let us first establish that for all $k\in \Z_T$,
\be\bigoplus_{i=1}^p A\otimes\C((t-x_i))\cong_\C  \iota\!\left(A_{\bm x}^{\Gamma,k} (t)\right) \oplus
\bigoplus_{i=1}^p A\otimes\C[[t-x_i]].\label{gll}\ee
Indeed, let $f_i^-\in A\otimes (t-x_i)^{-1} \C[(t-x_i)^{-1}]$ denote the pole part of $f_i\in A\otimes \C((t-x_i))$ and set  $f := \sum_{i=1}^p \sum_{\alpha\in \Gamma}\alpha^{k} (\alpha. f_{i}^-)(\alpha^{-1} t-x_i)\in A_{\bm x}^{\Gamma,k}$. Then $(f_1,\dots,f_p)$ splits uniquely as the direct sum of the function $f$  and the tuple  $(f_1 - \iota_{t-x_1}f,\dots,f_p-\iota_{t-x_p}f) \in \bigoplus_{i=1}^pA\otimes\C[[t - x_i]]$. Now using the ``only if''  part we have that
\be \sum_{i=1}^p \res_{t-x_i} \langle f_i,  \iota_{t-x_i} g\rangle = \sum_{i=1}^p \res_{t-x_i} \langle\left(f_i - \iota_{t-x_i}f\right), \iota_{t-x_i} g\rangle.\label{mv}\ee
Note $f_i - \iota_{t-x_i}f \in A\otimes \C[[t - x_i]]$. Suppose for a contradiction that $f_i - \iota_{t-x_i}f$ is non-zero and let $a\otimes\left(t-x_i\right)^n$ be the leading term, with $a\in A$, $a\neq 0$, and $n\in \Z_{\geq 0}$. By picking $g = \sum_{\alpha\in \Gamma} \alpha^{-k-1} (\alpha.b) /(\alpha^{-1}t-x_i)^{n+1}\in B_{\bm x}^{\Gamma,-k-1}$ one sees that the vanishing of \eqref{mv} implies that $\langle a,b\rangle=0$ for all $b\in B$; and thus, since $\langle\cdot,\cdot\rangle$ is non-degenerate, that $a=0$: a contradiction. Hence in fact $f_i=\iota_{t-x_i}f$ for each $i$, as required. The proof of part (\ref{tert2}) is similar.
\end{proof}

\section{Cyclotomic Schechtman-Varchenko formula}\label{app:SV}
The key identity used in this section is what, following \cite{SV}, we will call the \emph{circle lemma}: for any $n\in \Z_{\geq 2}$, if $x_{i}$ are pairwise distinct complex numbers with $i$ running over $\Z_n$ (so $x_{n+i}\equiv x_i$) then
\begin{equation} 
0= \sum_{i\in \Z_n } \prod_{\substack{j\in \Z_n\\ j\neq i}} \frac 1 {x_j-x_{j+1}}.
\nn
\end{equation}
For example, the case $n=3$ is
\begin{equation}\nn
0= \frac 1 {(x_1-x_2)(x_2-x_3)} 
  + \frac 1 {(x_2-x_3)(x_3-x_1)} 
  + \frac 1 {(x_3-x_1)(x_1-x_2)}.
\end{equation}
There is a useful graphical representation of such identities, in which $(x_i - x_j)^{-1}$ is represented by a directed edge from a vertex labelled $j$ to a vertex labelled $i$:
\begin{tikzpicture}[baseline=-1mm]
  \draw[<-, thick, shorten <= 1.5mm, shorten >= 1.5mm] (0,0) -- (.7,0);
  \draw[thick] (0,0) node[left]{\small $i$} node[cross=.8mm] {};
  \draw[thick] (.7,0) node[right]{\small $j$} node[cross=.8mm] {};
\end{tikzpicture}.
Products of such factors are then represented by directed graphs, with one vertex for each variable that appears and one directed edge for each factor. In this way, the circle lemma becomes
\begin{equation}
0 = \sum_{i \in \mathbb{Z}_n}
\raisebox{-14mm}{
\begin{tikzpicture}
\newcommand*{\arr}{51.43}
  \draw[<-, thick, shorten <= 1mm, shorten >= 1mm] (0:10mm) arc (0 : 0 + 42 : 10mm);
  \draw[thick] (0-5:1cm) node[below right]{\small $i+1$} node[cross=.8mm] {};

  \draw[<-, thick, shorten <= 1mm, shorten >= 1mm] (\arr:10mm) arc (\arr : \arr + 42 : 10mm);
  \draw[thick] (\arr-5:1cm) node[above right]{\small $i+2$} node[cross=.8mm] {};
  
  \draw[<-, thick, loosely dashed, shorten <= 1mm, shorten >= 1mm] (2*\arr:1) arc (2*\arr : 3*\arr + 42 : 1);
  \draw[thick] (2*\arr-5:1cm) node[cross=.8mm] {};

  \draw[<-, thick, shorten <= 1mm, shorten >= 1mm] (4*\arr:10mm) arc (4*\arr : 4*\arr + 42 : 10mm);
  \draw[thick] (4*\arr-5:1cm) node[cross=.8mm] {};

  \draw[<-, thick, shorten <= 1mm, shorten >= 1mm] (5*\arr:10mm) arc (5*\arr : 5*\arr + 42 : 10mm);
  \draw[thick] (5*\arr-5:1cm) node[below=1.2pt]{\small $i-1$} node[cross=.8mm] {};

  \draw[thick] (6*\arr-5:1cm) node[below right]{\small $i$} node[cross=.8mm] {};
\end{tikzpicture}
}\label{circleid}
.\end{equation}

As in the main text, let $\omega$ be a primitive $T$th root of unity, $T\in \Z_{\geq 1}$. Let $z_1,\dots,z_N,w_1,\dots,w_m$ ($N\in \Z_{\geq 1}$, $m\in \Z_{\geq 0}$) be nonzero complex numbers whose orbits under the multiplicative action of $\omega$ are pairwise disjoint.

We shall use three types of nodes in order to distinguish between different types of points and their images under multiplication by $\omega$. Namely, we write
\be  \tikz[baseline=-.7mm]{\draw[thick] (0,0) node[left]{\small $i$} circle (0.5mm);} = w_i,\qquad 
     \tikz[baseline=-.7mm]{\draw[thick] (.7,0) node[left]{\small $i$} node[inner sep=0,draw,rectangle,minimum size =1.2mm]{};} = z_i;
\qquad  \ee
and also
\be  \tikz[baseline=-.7mm]{\draw[fill,thick] (0,0) node[left]{\small $i$} circle (0.5mm);} = \omega^{k_i}w_i,\ee
where $k_i$, $1\leq i\leq m$, are elements of $\Z_T$. 

Thus, for example,
\be\begin{tikzpicture}[baseline=-6mm]
  \draw[thick] (4,0) node[above right]{\small $1$} node[inner sep=0,draw,rectangle,minimum size =1.2mm]{};
  \draw[->, thick, shorten <= 1.5mm, shorten >= 1.5mm] (4,0) -- (3.6,-.6);
  \draw[thick,fill=black] (3.6,-.6) node[right]{\small $2$} circle (0.5mm);
  \draw[->, thick, shorten <= 1.5mm, shorten >= 1.5mm] (3.6,-.6) -- (3.2,-1.2);
  \draw[thick,fill=black] (3.2,-1.2) node[right]{\small $1$} circle (0.5mm);
\end{tikzpicture}
=
\frac{1}{(\omega^{{k}_1} w_1 - \omega^{{k}_2} w_2)(\omega^{{k}_2} w_2 - z_1)}\nn\ee and
\be\begin{tikzpicture}[baseline=-6mm]
  \draw[thick] (4,0) node[above right]{\small $3$} circle (0.5mm);
  \draw[->, thick, shorten <= 1.5mm, shorten >= 1.5mm] (4,0) -- (3.6,-.6);
  \draw[thick,fill=black] (3.6,-.6) node[left]{\small $2$} circle (0.5mm);
  \draw[->, thick, shorten <= 1.5mm, shorten >= 1.5mm] (3.6,-.6) -- (3.2,-1.2);
  \draw[thick,fill=black] (3.2,-1.2) node[left]{\small $1$} circle (0.5mm);
  \draw[<-, thick, shorten <= 1.5mm, shorten >= 1.5mm] (4,0) -- (4.4,-.6);
  \draw[thick,fill=black] (4.4,-.6) node[right]{\small $4$} circle (0.5mm);
  \draw[->, thick, shorten <= 1.5mm, shorten >= 1.5mm] (4.4,-.6)  .. controls (3.95,-.95) .. (3.2,-1.2);
\end{tikzpicture}
=
\frac{1}{(\omega^{{k}_1} w_1 - \omega^{{k}_4} w_4)(\omega^{{k}_1} w_1 - \omega^{{k}_2} w_2)(\omega^{{k}_2} w_2 - w_3)(w_3 - \omega^{{k}_4} w_4)}.\nn\ee

Recall from \S\ref{sec:bv} the set $P_{m, N}$ of all ordered partitions of $\{ 1, \ldots, m \}$ into $N$ parts. 
Let $\bm n^i$ denote the $i$th part of an element $\bm n\in P_{m,N}$, i.e.
\be \bm n = (\bm n^1; \bm n^2;\dots; \bm n^N) = (n^1_1, \ldots, n^1_{p_1}; n^2_1,\dots, n^2_{p_2}; \dots ; n^N_1, \ldots, n^N_{p_N}).\ee

Let $\mc A$ be an associative unital algebra, $\sigma:\mc A\to \mc A$ an automorphism whose order divides $T$, and $V$ a left $\mc A$-module. As in the main text, it is convenient to write $\check\sigma(x) := \omega\sigma(x)$.
\begin{prop} \label{prop: SV}

Suppose we are given a linear map $\tau: V^{\otimes N}\otimes \mc A^{\otimes m}\to \C$ such that for all $s$, $1\leq s\leq m$, we have
\begin{align} &\tau( x_1,\dots, x_N; y_1,\dots y_s, 1,\dots 1)\label{taup}\\
&=  \sum_{i=1}^N \sum_{j\in \Z_T} \frac{\tau(x_1,\dots,x_{i-1} ,\sigma^j(y_s)x_i,x_{i+1},\dots, x_N;y_1,\dots, y_{s-1},1,\dots,1)}{w_{s} - \omega^{-j}z_i} \nn\\
&+  \sum_{i=1}^{s-1} \sum_{j \in \Z_T}\frac{\tau(x_1,\dots, x_N;y_1,\dots, y_{i-1},\sigma^j(y_s) y_i-y_i\sigma^j(y_s),y_{i+1}, \dots, y_{s-1},1,\dots,1)}{w_{s} - \omega^{-j}w_{i}}\nn\end{align}
for all $x_1,\dots, x_N\in V$ and all $y_1,\dots,y_s\in \mc A$.
Then for any $x_1, \ldots, x_N \in V$ and $a_1, \ldots, a_m \in \mathcal{A}$ we have
\begin{equation} \label{SVformula}
\tau(x_1,\dots,x_N; a_1,\dots,a_m) = \sum_{{\bm n} \in P_{m, N}} \tau\big( \tilde x_1({\bm n}^1),\dots,\tilde x_N({\bm n}^N); 1,\dots,1 \big)
\end{equation}
where for ${\bm n} = (n_1, \ldots, n_p)$, with $n_1, \ldots, n_p \in \{ 1, \ldots, m \}$ all distinct, we denote ${k}_{\bm n} := ({k}_{n_1}, \ldots, {k}_{n_p})$ and define
\begin{equation} \label{tilde x}
\tilde x_i({\bm n}) := \sum_{{k}_{\bm n} \in \Z^p_T}
\left(\;\;
\begin{tikzpicture}[baseline=-12mm]
  \draw[thick] (0,0) node[right]{\small $i$} node[inner sep=0,draw,rectangle,minimum size =1.2mm]{};
  \draw[->, thick, shorten <= 1.5mm, shorten >= 1.5mm] (0,0) -- (0,-.6);
  \draw[thick,fill=black] (0,-.6) node[right]{\small $n_p$} circle (0.5mm);
  \draw[->, thick, dashed, shorten <= 1.5mm, shorten >= 1.5mm] (0,-.6) -- (0,-1.8);
  \draw[thick, fill=black] (0,-1.8) node[right]{\small $n_2$} circle (0.5mm);
  \draw[->, thick, shorten <= 1.5mm, shorten >= 1.5mm] (0,-1.8) -- (0,-2.4);
  \draw[thick,fill=black] (0,-2.4) node[right]{\small $n_1$} circle (0.5mm);
\end{tikzpicture}
\right)
\overrightarrow{\prod_{t=1}^p} \check{\sigma}^{{k}_{n_t}}\big(a_{n_t} \big) x_i.
\end{equation}
\end{prop}
\begin{proof}
We prove a slightly more general result. Consider the set $P_{m - s, N + 2s}$ of ordered partitions of $\{ 1, \ldots, m - s \}$ into $N + 2s$ parts, for $0 \leq s \leq m$.  
It is convenient to write an element ${\bm n} \in P_{m - s, N + 2s}$ 
\begin{align}
{\bm n} & = (\bm n^1;\dots ;\bm n^N;\bm l^1; \bm r^1; \dots ; \bm l^s; \bm r^s ) \nn\\
&= (n^1_1, \ldots, n^1_{p_1}; \ldots; n^N_1, \ldots, n^N_{p_N}; l^1_1, \ldots, l^1_{L_1}; r^1_1, \ldots, r^1_{R_1}; \ldots; l^s_1, \ldots, l^s_{L_s}; r^s_1, \ldots, r^s_{R_s}),\label{big n}
\end{align} 
where $(p_1,\dots,p_N,L_1,R_1,\dots,L_s,R_s)$ is an $(N+2s)$-composition of $m-s$. 
In terms of this notation, we claim that for any $0 \leq s \leq m$ we have
\begin{align} \label{SVformula2}
\tau(x_1,&\dots,x_N; a_1,\dots,a_m) \notag\\
&= \sum_{{\bm n} \in P_{m - s, N + 2s}} \tau\big( \tilde x_1({\bm n}^1),\dots,\tilde x_N({\bm n}^N); \tilde y_1({\bm l^1}, {\bm r^1}),\dots,\tilde y_s({\bm l^s}, {\bm r^s}),1,\dots,1 \big)
\end{align}
where $\tilde x_i(\bm n)$ is as in \eqref{tilde x} and for ${\bm l} = (l_1, \ldots, l_L)$ and ${\bm r} = (r_1, \ldots, r_R)$, with $l_1, \ldots, l_L, r_1, \ldots, r_R \in \{ 1, \ldots, m-s \}$ all distinct we set
\begin{equation} \label{tilde y}
\tilde y_i({\bm l}, {\bm r}) := \sum_{\substack{{k}_{\bm l} \in \Z^L_T\\ {k}_{\bm r} \in \Z^R_T}}
\left(
\begin{tikzpicture}[baseline=-12mm]
  \draw[thick] (.1,0) node[above]{\small $i$} circle (0.5mm);
  \draw[->, thick, shorten <= 1.5mm, shorten >= 1.5mm] (.1,0) -- (-.2,-.6);
  \draw[thick,fill=black] (-.2,-.6) node[left]{\small $l_L$} circle (0.5mm);
  \draw[->, thick, dashed, shorten <= 1.5mm, shorten >= 1.5mm] (-.2,-.6) -- (-.2,-1.8);
  \draw[thick, fill=black] (-.2,-1.8) node[left]{\small $l_2$} circle (0.5mm);
  \draw[->, thick, shorten <= 1.5mm, shorten >= 1.5mm] (-.2,-1.8) -- (-.2,-2.4);
  \draw[thick,fill=black] (-.2,-2.4) node[left]{\small $l_1$} circle (0.5mm);

  \draw[<-, thick, shorten <= 1.5mm, shorten >= 1.5mm] (.1,0) -- (.4,-.6);
  \draw[thick,fill=black] (.4,-.6) node[right]{\small $r_R$} circle (0.5mm);
  \draw[<-, thick, dashed, shorten <= 1.5mm, shorten >= 1.5mm] (.4,-.6) -- (.4,-1.8);
  \draw[thick, fill=black] (.4,-1.8) node[right]{\small $r_2$} circle (0.5mm);
  \draw[<-, thick, shorten <= 1.5mm, shorten >= 1.5mm] (.4,-1.8) -- (.4,-2.4);
  \draw[thick,fill=black] (.4,-2.4) node[right]{\small $r_1$} circle (0.5mm);
\end{tikzpicture}
\right)
\overrightarrow{\prod_{t=1}^L} \check{\sigma}^{{k}_{l_t}}\big(a_{l_t} \big) \,\,a_i\,\, \overleftarrow{\prod_{u=1}^R} \check{\sigma}^{{k}_{r_u}}(a_{r_u})
\end{equation}
for $1 \leq i \leq s$.
The statement of the proposition corresponds to the particular case when $s = 0$.

We proceed to show \eqref{SVformula2} by induction on $s$ (starting at $s=m$ and working downwards). When $s = m$ the statement \eqref{SVformula2} is empty. So suppose it holds for some $0 < s \leq m$. Applying \eqref{taup} to each term in the sum on the right hand side and abbreviating $\tilde{x}_i({\bm n}^i)$ and $\tilde{y}_i({\bm l}^{i}, {\bm r}^{i})$ as $\tilde{x}_i$ and $\tilde{y}_i$ respectively, we have
\begin{align} \label{s swapped}
&\sum_{{\bm n} \in P_{m - s, N + 2s}} \tau( \tilde{x}_1, \dots, \tilde{x}_N; \tilde{y}_1,\dots \tilde{y}_s, 1,\dots 1) \notag\\
&= \sum_{i=1}^N\sum_{{\bm n} \in P_{m - s, N + 2s}}  \tau\bigg( \tilde{x}_1,\dots, \tilde{x}_{i-1}, \sum_{j \in \Z_T} \frac{\check{\sigma}^j\big(\tilde{y}_s \big) \tilde{x}_i}{\omega^j w_s - z_i}, \tilde{x}_{i+1}, \dots, \tilde{x}_N; \tilde{y}_1,\dots, \tilde{y}_{s-1},1,\dots,1 \bigg) \notag\\
&\quad + \sum_{i=1}^{s-1} \sum_{{\bm n} \in P_{m - s, N + 2s}} \tau\bigg( \tilde{x}_1,\dots, \tilde{x}_N; \tilde{y}_1,\dots, \tilde{y}_{i-1}, \sum_{j \in \Z_T} \frac{\check{\sigma}^j\big(\tilde{y}_s \big) \tilde{y}_i}{\omega^j w_s - w_i}, \tilde{y}_{i+1}, \dots, \tilde{y}_{s-1},1,\dots,1 \bigg) \notag\\
&\quad + \sum_{i=1}^{s-1}\sum_{{\bm n} \in P_{m - s, N + 2s}}  \tau\bigg( \tilde{x}_1,\dots, \tilde{x}_N; \tilde{y}_1,\dots, \tilde{y}_{i-1}, \sum_{j \in \Z_T} \frac{\tilde{y}_i \check{\sigma}^j\big(\tilde{y}_s \big)}{w_i - \omega^j w_s}, \tilde{y}_{i+1}, \dots, \tilde{y}_{s-1},1,\dots,1 \bigg).
\end{align}

In the first line of the right-hand side of \eqref{s swapped}, consider the $i$th term in the sum. Given an ordered partition ${\bm n} \in P_{m-s, N + 2s}$, let  
\be \tilde{\bm n} = (\tilde{\bm n}^1;\dots ;\tilde{\bm n}^N;\tilde{\bm l}^1; \tilde{\bm r}^1; \dots ; \tilde{\bm l}^{s-1}; \tilde{\bm r}^{s-1} ) \in P_{m-s+1,N+2s-2}  \label{tildendef}\ee
be the ordered partition whose $i$th part is 
\be  \tilde{\bm n}^i = (\tilde n^i_1,\dots,\tilde n^i_{\tilde p_i}) 
:= (l^s_1, \ldots, l^s_{L_s}, s, r^s_{R_s}, \dots r^s_1, n^i_{1},\dots, n^i_{p_i})\nn\ee
(so $\tilde p_i=p_i+1+L_s+R_s$) and whose remaining parts are unaltered:
\begin{alignat}{2} \tilde{\bm n}^j &:= \bm n^j& \quad  &1\leq j\leq N,\,j\neq i,\nn\\
 \tilde{\bm l}^t &:= \bm l^t, &\quad &1\leq t\leq s-1,\nn\\
  \tilde{\bm r}^t &:= \bm r^t, &\quad &1\leq t\leq s-1. \nn
\end{alignat}
One can show from definitions \eqref{tilde x} and \eqref{tilde y} that 
\begin{equation*}
\sum_{j \in \Z_T} \frac{\check{\sigma}^j\big(\tilde{y}_s(\bm l^s, \bm r^s) \big) \tilde{x}_i(\bm n^i)}{\omega^j w_s - z_i} =- \sum_{{k}_{\tilde{\bm n}^i} \in \Z^{\tilde p_i}_T}
\left(
\begin{tikzpicture}[baseline=-12mm]
  \draw[thick,fill=black] (.1,0) node[above]{\small $s$} circle (0.5mm);
  \draw[->, thick, shorten <= 1.5mm, shorten >= 1.5mm] (.1,0) -- (-.2,-.6);
  \draw[thick,fill=black] (-.2,-.6) node[left]{\small $\tilde{n}^i_{L_s}$} circle (0.5mm);
  \draw[->, thick, dashed, shorten <= 1.5mm, shorten >= 1.5mm] (-.2,-.6) -- (-.2,-1.8);
  \draw[thick, fill=black] (-.2,-1.8) node[left]{\small $\tilde{n}^i_2$} circle (0.5mm);
  \draw[->, thick, shorten <= 1.5mm, shorten >= 1.5mm] (-.2,-1.8) -- (-.2,-2.4);
  \draw[thick,fill=black] (-.2,-2.4) node[left]{\small $\tilde{n}^i_1$} circle (0.5mm);

  \draw[<-, thick, shorten <= 1.5mm, shorten >= 1.5mm] (.1,0) -- (.4,-.6);
  \draw[thick,fill=black] (.4,-.6) node[right]{\small $\tilde{n}^i_{L_s+2}$} circle (0.5mm);
  \draw[<-, thick, dashed, shorten <= 1.5mm, shorten >= 1.5mm] (.4,-.6) -- (.4,-1.8);
  \draw[thick, fill=black] (.4,-1.8) node[right]{\small $\tilde{n}^i_{L_s+R_s}$} circle (0.5mm);
  \draw[<-, thick, shorten <= 1.5mm, shorten >= 1.5mm] (.4,-1.8) -- (.4,-2.4);
  \draw[thick,fill=black] (.4,-2.4) node[right]{\small $\tilde{n}^i_{L_s+R_s+1}$} circle (0.5mm);
 
  \draw[thick] (2.2,0) node[above right]{\small $i$} node[inner sep=0,draw,rectangle,minimum size =1.2mm]{};
  \draw[->, thick, shorten <= 1.5mm, shorten >= 1.5mm] (2.2,0) -- (2.2,-.6);
  \draw[thick,fill=black] (2.2,-.6) node[right]{\small $\tilde{n}^i_{\tilde p_i}$} circle (0.5mm);
  \draw[->, thick, dashed, shorten <= 1.5mm, shorten >= 1.5mm] (2.2,-.6) -- (2.2,-1.8);
  \draw[thick, fill=black] (2.2,-1.8) node[right]{\small $\tilde{n}^i_{L_s+R_s+3}$} circle (0.5mm);
  \draw[->, thick, shorten <= 1.5mm, shorten >= 1.5mm] (2.2,-1.8) -- (2.2,-2.4);
  \draw[thick,fill=black] (2.2,-2.4) node[right]{\small $\tilde{n}^i_{L_s+R_s+2}$} circle (0.5mm);

  \draw[->, thick, shorten <= 1.5mm, shorten >= 1.5mm] (.1,0) -- (2.2,0);
%.. controls (.65,.1) ..
\end{tikzpicture}
\right)
\overrightarrow{\prod_{t=1}^{\tilde p_i}} \check{\sigma}^{{k}_{\tilde{n}^i_t}}(a_{\tilde{n}^i_t}) x_i.
\end{equation*}
Now, the ordered partition $\tilde{\bm n}$ belongs to the set of those elements of $P_{m-s+1,N+2s-2}$ for which $s$ belongs to the $i$th part. Call this subset $P^{(i)}_{m-s+1,N+2s-2}\subset P_{m-s+1,N+2s-2}$.
The map $$P_{m-s,N+2s}\to \tilde P^{(i)}_{m-s+1,N+2s-2}; \,\,\bm n\mapsto \tilde{\bm n}$$ is not injective, and indeed summing over the pre-images of a fixed $\tilde{\bm n}$ amounts to summing over pairs $(R_s,p_i)$ with $R_s+p_i$ is fixed. When one performs this sum on the expression above and applies the circle lemma \eqref{circleid}, one finds 
\begin{equation*}
\sum_{{k}_{\tilde{\bm n}^i} \in \Z^{\tilde p_i}_T}
\left(
\begin{tikzpicture}[baseline=-12mm]
  \draw[thick,fill=black] (.1,0) node[above]{\small $s$} circle (0.5mm);
  \draw[->, thick, shorten <= 1.5mm, shorten >= 1.5mm] (.1,0) -- (-.2,-.6);
  \draw[thick,fill=black] (-.2,-.6) node[left]{\small $\tilde{n}^i_{L_s}$} circle (0.5mm);
  \draw[->, thick, dashed, shorten <= 1.5mm, shorten >= 1.5mm] (-.2,-.6) -- (-.2,-1.8);
  \draw[thick, fill=black] (-.2,-1.8) node[left]{\small $\tilde{n}^i_2$} circle (0.5mm);
  \draw[->, thick, shorten <= 1.5mm, shorten >= 1.5mm] (-.2,-1.8) -- (-.2,-2.4);
  \draw[thick,fill=black] (-.2,-2.4) node[left]{\small $\tilde{n}^i_1$} circle (0.5mm);

  \draw[<-, thick, shorten <= 1.5mm, shorten >= 1.5mm] (.1,0) -- (.4,-.6);
  \draw[thick,fill=black] (.4,-.6) node[right]{\small $\tilde{n}^i_{L_s+2}$} circle (0.5mm);
  \draw[<-, thick, dashed, shorten <= 1.5mm, shorten >= 1.5mm] (.4,-.6) -- (.4,-1.8);
  \draw[thick, fill=black] (.4,-1.8) node[right]{\small $ $} circle (0.5mm);
  \draw[<-, thick, shorten <= 1.5mm, shorten >= 1.5mm] (.4,-1.8) -- (.4,-2.4);
  \draw[thick,fill=black] (.4,-2.4) node[left]{\small $ $} circle (0.5mm);
 
  \draw[thick] (1.7,0) node[above right]{\small $i$} node[inner sep=0,draw,rectangle,minimum size =1.2mm]{};
  \draw[->, thick, shorten <= 1.5mm, shorten >= 1.5mm] (1.7,0) -- (1.7,-.6);
  \draw[thick,fill=black] (1.7,-.6) node[right]{\small $\tilde{n}^i_{\tilde p_i}$} circle (0.5mm);
  \draw[->, thick, dashed, shorten <= 1.5mm, shorten >= 1.5mm] (1.7,-.6) -- (1.7,-1.8);
  \draw[thick, fill=black] (1.7,-1.8) node[right]{\small $ $} circle (0.5mm);
  \draw[->, thick, shorten <= 1.5mm, shorten >= 1.5mm] (1.7,-1.8) -- (1.7,-2.4);
  \draw[thick,fill=black] (1.7,-2.4) node[right]{\small $ $} circle (0.5mm);

  \draw[<-, thick, shorten <= 1.5mm, shorten >= 1.5mm] (.4,-2.4) -- (1.7,-2.4);
%.. controls (.65,.1) ..
\end{tikzpicture}
\right)
\overrightarrow{\prod_{t=1}^{\tilde p_i}} \check{\sigma}^{{k}_{\tilde{n}^i_t}}(a_{\tilde{n}^i_t}) x_i = \tilde{x}_i(\tilde{\bm n}^i).
\end{equation*}
Putting all this together, one has that the first line on the right hand side of \eqref{s swapped} is
\be\nn\sum_{i=1}^N \sum_{\tilde{\bm n}\in P^{(i)}_{m-s+1,N+2s-2}} 
 \tau\big( \tilde x_1(\tilde{\bm n}^1),\dots,\tilde x_N(\tilde{\bm n}^N);
 \tilde y_1(\tilde{\bm l}^{1}, \tilde{\bm r}^{1}),\dots,\tilde y_{s-1}(\tilde{\bm l}^{s-1}, \tilde{\bm r}^{s-1}),1,\dots,1 \big).\ee
These are some of the required terms; it remains to show that the second and third lines on the right of \eqref{s swapped} yield the remaining terms in the sum over $\tilde{\bm n}\in P_{m-s-1,N+2s-2}$, i.e. those in which $s$ belongs to one of the last $2s-2$ parts.

Consider the $i$th term of the sum on the second line on the right of \eqref{s swapped}. This time, given an ordered partition ${\bm n} \in P_{m-s, N + 2s}$, we now let  
\be \tilde{\bm n} = (\tilde{\bm n}^1;\dots ;\tilde{\bm n}^N;\tilde{\bm l}^1; \tilde{\bm r}^1; \dots ; \tilde{\bm l}^{s-1}; \tilde{\bm r}^{s-1} ) \in P_{m-s+1,N+2s-2}\nn \ee
be the ordered partition with 
\be  \tilde{\bm l}^i = (\tilde l^i_1,\dots,\tilde l^i_{\tilde L_i}) 
:= (l^s_1, \ldots, l^s_{L_s}, s, r^s_{R_s}, \dots r^s_1, l^i_{1},\dots, l^i_{L_i})\nn\ee
(so $\tilde L_i = L_i + 1 + L_s+R_s$) and with the remaining parts unaltered:
\begin{alignat}{2} \tilde{\bm n}^j &:= \bm n^j& \quad  &1\leq j\leq N,\nn\\
 \tilde{\bm l}^t &:= \bm l^t, &\quad &1\leq t\leq s-1,\, t\neq i,\nn\\
  \tilde{\bm r}^t &:= \bm r^t, &\quad &1\leq t\leq s-1. \nn
\end{alignat}
Using the definition \eqref{tilde y} one can show that
\begin{eqnarray*}
\sum_{j \in \Z_T} \frac{\check{\sigma}^j\big(\tilde{y}_s(\bm l^s, \bm r^s) \big) 
       \tilde{y}_i(\bm l^{i},\bm r^i)}{\omega^j w_s - w_i}
= - \sum_{{k}_{\tilde{\bm l}^{i}} \in \Z^{\tilde L_i}_T}
\left(
\begin{tikzpicture}[baseline=-12mm]
  \draw[thick,fill=black] (-.1,0) node[above]{\small $s$} circle (0.5mm);
  \draw[->, thick, shorten <= 1.5mm, shorten >= 1.5mm] (-.1,0) -- (-.4,-.6);
  \draw[thick,fill=black] (-.4,-.6) node[left]{\small $\tilde{l}^{i}_{L_s}$} circle (0.5mm);
  \draw[->, thick, dashed, shorten <= 1.5mm, shorten >= 1.5mm] (-.4,-.6) -- (-.4,-1.8);
  \draw[thick, fill=black] (-.4,-1.8) node[left]{\small $\tilde{l}^{i}_2$} circle (0.5mm);
  \draw[->, thick, shorten <= 1.5mm, shorten >= 1.5mm] (-.4,-1.8) -- (-.4,-2.4);
  \draw[thick,fill=black] (-.4,-2.4) node[left]{\small $\tilde{l}^{i}_1$} circle (0.5mm);

  \draw[<-, thick, shorten <= 1.5mm, shorten >= 1.5mm] (-.1,0) -- (.2,-.6);
  \draw[thick,fill=black] (.2,-.6) node[right]{\small $\tilde{l}^{i}_{L_s+2}$} circle (0.5mm);
  \draw[<-, thick, dashed, shorten <= 1.5mm, shorten >= 1.5mm] (.2,-.6) -- (.2,-1.8);
  \draw[thick, fill=black] (.2,-1.8) node[right]{\small $\tilde{l}^{i}_{L_s+R_s}$} circle (0.5mm);
  \draw[<-, thick, shorten <= 1.5mm, shorten >= 1.5mm] (.2,-1.8) -- (.2,-2.4);
  \draw[thick,fill=black] (.2,-2.4) node[right]{\small $\tilde{l}^{i}_{L_s+R_s+1}$} circle (0.5mm);
  
  \draw[thick] (2.7,0) node[above]{\small $i$} circle (0.5mm);
  \draw[->, thick, shorten <= 1.5mm, shorten >= 1.5mm] (2.7,0) -- (1.9,-.6);
  \draw[thick,fill=black] (1.9,-.6) node[right=2mm]{\small $\tilde{l}^{i}_{\tilde L_i}$} circle (0.5mm);
  \draw[->, thick, dashed, shorten <= 1.5mm, shorten >= 1.5mm] (1.9,-.6) -- (1.9,-1.8);
  \draw[thick, fill=black] (1.9,-1.8) node[right]{\small $\tilde{l}^{i}_{L_s+R_s+3}$} circle (0.5mm);
  \draw[->, thick, shorten <= 1.5mm, shorten >= 1.5mm] (1.9,-1.8) -- (1.9,-2.4);
  \draw[thick,fill=black] (1.9,-2.4) node[right]{\small $\tilde{n}^{l,i}_{L_s+R_s+2}$} circle (0.5mm);

  \draw[<-, thick, shorten <= 1.5mm, shorten >= 1.5mm] (2.7,0) -- (3.6,-.6);
  \draw[thick,fill=black] (3.6,-.6) node[right]{\small $\tilde{r}^{i}_{R_i}$} circle (0.5mm);
  \draw[<-, thick, dashed, shorten <= 1.5mm, shorten >= 1.5mm] (3.6,-.6) -- (3.6,-1.8);
  \draw[thick, fill=black] (3.6,-1.8) node[right]{\small $\tilde{r}^{i}_2$} circle (0.5mm);
  \draw[<-, thick, shorten <= 1.5mm, shorten >= 1.5mm] (3.6,-1.8) -- (3.6,-2.4);
  \draw[thick,fill=black] (3.6,-2.4) node[right]{\small $\tilde{r}^{i}_1$} circle (0.5mm);

  \draw[->, thick, shorten <= 1.5mm, shorten >= 1.5mm] (-.1,0) -- (2.7,0);
%.. controls (.75,.1) ..
\end{tikzpicture}
\right)
\overrightarrow{\prod_{t=1}^{\tilde L_i}} \check{\sigma}^{{k}_{\,\tilde{l}^{i}_t}}\big(a_{\tilde{l}^{i}_t}\big)
a_i
\overleftarrow{\prod_{u=1}^{R_i}} \check{\sigma}^{{k}_{\tilde{r}^{i}_u}}\big(a_{\tilde{r}^{i}_u}\big).
\end{eqnarray*}
Arguing as before, when one sums this expression over those ordered partitions $\bm n$ that yield a given $\tilde{\bm n}$ and applies the circle lemma \eqref{circleid}, one indeed finds
\be  \sum_{{k}_{\tilde{\bm l}^{i}} \in \Z^{\tilde L_i}_T}
\left(
\begin{tikzpicture}[baseline=-12mm]
  \draw[thick,fill=black] (-.1,0) node[above]{\small $s$} circle (0.5mm);
  \draw[->, thick, shorten <= 1.5mm, shorten >= 1.5mm] (-.1,0) -- (-.4,-.6);
  \draw[thick,fill=black] (-.4,-.6) node[left]{\small $\tilde{l}^{i}_{L_s}$} circle (0.5mm);
  \draw[->, thick, dashed, shorten <= 1.5mm, shorten >= 1.5mm] (-.4,-.6) -- (-.4,-1.8);
  \draw[thick, fill=black] (-.4,-1.8) node[left]{\small $\tilde{l}^{i}_2$} circle (0.5mm);
  \draw[->, thick, shorten <= 1.5mm, shorten >= 1.5mm] (-.4,-1.8) -- (-.4,-2.4);
  \draw[thick,fill=black] (-.4,-2.4) node[left]{\small $\tilde{l}^{i}_1$} circle (0.5mm);

  \draw[<-, thick, shorten <= 1.5mm, shorten >= 1.5mm] (-.1,0) -- (.2,-.6);
  \draw[thick,fill=black] (.2,-.6) node[right]{\small $\tilde{l}^{i}_{L_s+2}$} circle (0.5mm);
  \draw[<-, thick, dashed, shorten <= 1.5mm, shorten >= 1.5mm] (.2,-.6) -- (.2,-1.8);
  \draw[thick, fill=black] (.2,-1.8) node[right]{\small $ $} circle (0.5mm);
  \draw[<-, thick, shorten <= 1.5mm, shorten >= 1.5mm] (.2,-1.8) -- (.2,-2.4);
  \draw[thick,fill=black] (.2,-2.4) node[right]{\small $ $} circle (0.5mm);
  
  \draw[thick] (2.0,0) node[above]{\small $i$} circle (0.5mm);
  \draw[->, thick, shorten <= 1.5mm, shorten >= 1.5mm] (2.0,0) -- (1.6,-.6);
  \draw[thick,fill=black] (1.6,-.6) node[right]{\small $\tilde{l}^{i}_{\tilde L_i}$} circle (0.5mm);
  \draw[->, thick, dashed, shorten <= 1.5mm, shorten >= 1.5mm] (1.6,-.6) -- (1.6,-1.8);
  \draw[thick, fill=black] (1.6,-1.8) node[right]{\small $ $} circle (0.5mm);
  \draw[->, thick, shorten <= 1.5mm, shorten >= 1.5mm] (1.6,-1.8) -- (1.6,-2.4);
  \draw[thick,fill=black] (1.6,-2.4) node[right]{\small $ $} circle (0.5mm);

  \draw[<-, thick, shorten <= 1.5mm, shorten >= 1.5mm] (2.0,0) -- (2.4,-.6);
  \draw[thick,fill=black] (2.4,-.6) node[right]{\small $\tilde{r}^{i}_{R_i}$} circle (0.5mm);
  \draw[<-, thick, dashed, shorten <= 1.5mm, shorten >= 1.5mm] (2.4,-.6) -- (2.4,-1.8);
  \draw[thick, fill=black] (2.4,-1.8) node[right]{\small $\tilde{r}^{i}_2$} circle (0.5mm);
  \draw[<-, thick, shorten <= 1.5mm, shorten >= 1.5mm] (2.4,-1.8) -- (2.4,-2.4);
  \draw[thick,fill=black] (2.4,-2.4) node[right]{\small $\tilde{r}^{i}_1$} circle (0.5mm);

  \draw[->, thick, shorten <= 1.5mm, shorten >= 1.5mm] (1.6,-2.4) -- (.2,-2.4);
%.. controls (.75,.1) ..
\end{tikzpicture}
\right)
\overrightarrow{\prod_{t=1}^{\tilde L_i}} \check{\sigma}^{{k}_{\,\tilde{l}^{i}_t}}\big(a_{\tilde{l}^{i}_t}\big)
a_i
\overleftarrow{\prod_{u=1}^{R_i}} \check{\sigma}^{{k}_{\tilde{r}^{i}_u}}\big(a_{\tilde{r}^{i}_u}\big)
= \tilde{y}_i(\tilde{\bm l}^{i}, \tilde{\bm r}^{i}).
\ee
Similar reasoning applies to the sum in the last line of \eqref{s swapped}, and we obtain finally
\begin{align*}
\tau(x_1,&\dots,x_N; a_1,\dots,a_m) \notag\\
&= \sum_{\tilde{\bm n} \in P_{m-s+1, N+2s-2}} \tau\big( \tilde x_1(\tilde{\bm n}^1),\dots,\tilde x_N(\tilde{\bm n}^N);
 \tilde y_1(\tilde{\bm l}^{1}, \tilde{\bm r}^{1}),\dots,\tilde y_{s-1}(\tilde{\bm l}^{s-1}, \tilde{\bm r}^{s-1}),1,\dots,1 \big),
\end{align*}
which concludes the proof of the inductive step.
\end{proof}

\begin{cor}\label{SVcor}
Let  $\mf a$ be a Lie algebra, $\sigma:\mf a\to \mf a$ an automorphism of $\mf a$ whose order divides $T$, and $V$ an $\mf a$-module.
Suppose $\overline\tau: V^{\otimes N} \otimes (\mf a\oplus \C)^{\otimes m}  \to \C$ is a linear map such that for each $s$, $1\leq s\leq m$,
\begin{align} \label{swappingtau} \overline\tau( x_1,\dots, x_N, y_1,\dots y_s,1,\dots,1)
&=  \sum_{i=1}^N \sum_{j\in \Z_T} \frac{\overline\tau(x_1,\dots,x_{i-1} ,\sigma^j(y_s) x_i,x_{i+1},\dots x_N,y_1,\dots, y_{s-1},1,\dots,1)}{w_{s} - \omega^{-j}z_{i}} \\\nn
&+  \sum_{i=1}^{s-1}\sum_{j\in \Z_T} \frac{\overline\tau(x_1,\dots, x_N,y_1,\dots, y_{i-1},[\sigma^{j}(y_s), y_i] ,y_{i+1}, \dots, y_{s-1},1,\dots,1)}{w_{s} - \omega^{-j}w_{i}}\end{align}
for all $x_1,\dots, x_N\in V$ and all $y_1,\dots,y_s\in \mf a$. Then formula \eqref{SVformula} holds for $\overline\tau$, for all $x_1,\dots,x_N\in V$ and all $a_1,\dots,a_m\in \mf a$. 
\end{cor}
\begin{proof}
Any such $\overline\tau$ extends to a map $\tau:V^{\otimes N}\otimes U(\mf a)^{\otimes m}\to \C$ satisfying \eqref{taup}. Indeed, one may set $\tau(x_1,\dots,x_N;1,\dots,1) := \overline\tau(x_1,\dots,x_N;1,\dots,1)$ and then use the relations \eqref{taup} recursively to \emph{define} a map $\tau: V^{\otimes N}\otimes U(\mf a)^{\otimes m}\to \C$. By  definition of the universal envelope $U(\mf a)$ the restriction of $\tau$ to $V^{\otimes N} \otimes (\mf a\oplus \C)^{\otimes m}$ coincides with $\overline\tau$. The result follows.
\end{proof}

\def\cprime{$'$}
\providecommand{\bysame}{\leavevmode\hbox to3em{\hrulefill}\thinspace}
\providecommand{\MR}{\relax\ifhmode\unskip\space\fi MR }
\providecommand{\MRhref}[2]{%
  \href{http://www.ams.org/mathscinet-getitem?mr=#1}{#2}
}
\providecommand{\href}[2]{#2}

\end{document}